\documentclass[oneside]{amsart}
    \usepackage{amsmath}
    \usepackage{amsfonts}
    \usepackage{amsthm}
    \usepackage{amssymb}
    \usepackage{hyperref}
    \usepackage{amscd}
    \usepackage{color}
    \usepackage{enumerate}
    \usepackage{xypic}
    \usepackage{tikz}
    \usepackage{tikz-cd}
    \usepackage{colortbl}
    \usepackage[normalem]{ulem}
    \usepackage{multirow}
    \usepackage{makecell}

    \usepackage{booktabs}

    \usepackage{dsfont}
    
    \makeatletter
    \def\l@subsection{\@tocline{2}{0pt}{2.5pc}{5pc}{}}
    \makeatother

    \newcommand{\G}{\mathcal{G}}
    \newcommand{\A}{\mathcal{A}}
    \newcommand{\ZZ}{\mathbb{Z}}
    \newcommand{\KK}{\mathbb{K}}
    \newcommand{\LL}{\mathbb{L}}
    \newcommand{\FF}{\mathbb{F}}
    \newcommand{\RR}{\mathbb{R}}
    \newcommand{\PP}{\mathbb{P}}
    \newcommand{\CC}{\mathbb{C}}
    \newcommand{\QQ}{\mathbb{Q}}

    \newcommand{\cS}{\mathcal{S}}
    \newcommand{\cQ}{\mathcal{Q}}
    \newcommand{\cJ}{\mathcal{J}}
    
    \newcommand{\mcH}{\mathcal{H}}

    \newcommand{\HOT}{(\star\star\star)}
    \newcommand{\bu}{\mathbf{u}}
    \newcommand{\bv}{\mathbf{v}}
    \newcommand{\bw}{\mathbf{w}}
    \newcommand{\br}{\mathbf{r}}
    \newcommand{\bm}{\mathbf{m}}

    \newcommand{\TP}{{\mathbb{T}^n}}
    \newcommand{\TT}{{\mathbb{T}}}
    \newcommand{\CSP}{{\mathfrak{C}}}

    \DeclareMathOperator{\rowspan}{rowspan}
    \DeclareMathOperator{\codim}{codim}
    
    \DeclareMathOperator{\spec}{Spec}
    \DeclareMathOperator{\trop}{Trop}
    
    \DeclareMathOperator{\conv}{conv}
    \DeclareMathOperator{\val}{\mathfrak{v}}

    \DeclareMathOperator{\Gr}{Gr}
    \DeclareMathOperator{\row}{Row}
    
    \DeclareMathOperator{\vspan}{span}
    \DeclareMathOperator{\Hom}{Hom}

    \newtheorem{thm}[equation]{Theorem}
    \newtheorem{prop}[equation]{Proposition}
    \newtheorem{lemma}[equation]{Lemma}
    \newtheorem{cor}[equation]{Corollary}

    \theoremstyle{definition}
    
    \newtheorem{defn}[equation]{Definition}
    
    \newtheorem{alg}[equation]{Algorithm}
    
    \newtheorem{rem}[equation]{Remark}
    \newtheorem{ass}[equation]{Assumption}
    \newtheorem{warning}[equation]{Warning}
    \newtheorem{oldex}[equation]{Example}
    
    \newenvironment{ex}
      {\pushQED{\qed}
      \oldex}
      {\popQED\endoldex}

    \newcommand{\hide}[1]{}
    
    \numberwithin{equation}{subsection}

    \newcolumntype{g}{>{\columncolor{yellow}}c}
    \title{Tropical tangents for complete intersection curves}
    
    \author{Nathan Ilten}
    \address{Department of Mathematics, Simon Fraser University,
    8888 University Drive, Burnaby BC V5A1S6, Canada}
    \email{\href{mailto:nilten@sfu.ca}{nilten@sfu.ca}}
    
    \author{Yoav Len}
    \address{Mathematical Institute, University of St Andrews, St Andrews KY16 9SS, UK}
    \email{\href{mailto:yoav.len@st-andrews.ac.uk}{yoav.len@st-andrews.ac.uk}}

\begin{document}
    
    \begin{abstract}
    We consider the tropicalization of tangent lines to a complete intersection curve $X$ in $\PP^n$. Under mild hypotheses, we describe a procedure for computing the tropicalization
    of the image of the Gauss map of $X$ in terms of the tropicalizations of the hypersurfaces cutting out $X$. We apply this to obtain descriptions 
    of the tropicalization of the dual variety $X^*$ and tangential variety $\tau(X)$ of $X$. In particular, we are able to compute the degrees of $X^*$ and $\tau(X)$ and the Newton polytope of $\tau(X)$ without using any elimination theory.
    
    \end{abstract}
    \maketitle
    
    \setcounter{tocdepth}{2}
    \tableofcontents
    
    \section{Introduction}
    \subsection{Background and related work}
    The notion of tangency plays an important role for many classical constructions in algebraic geometry. For example, the \emph{tangential variety} to  a projective variety $X\subset \PP^n$ is
    \[
    	\tau(X)=\overline{\{Q\in \PP^n\ |\ Q\in T_PX \ \textrm{for some smooth point}\ P\in X\}}.
    \]
    Similarly, the \emph{dual variety} to $X$ is
    \[
    	X^*=\overline{\{H\in (\PP^n)^*\ |\ T_PX\subset H\ \textrm{for some smooth point}\ P\in X\}}.
    \]
    Here $(\PP^n)^*$ is the dual projective space, whose points are hyperplanes in $\PP^n$.
    See \cite[\S15]{harris}.
    It is frequently of interest to describe basic invariants of these varieties such as dimension and degree. For example, if $X$ is an integral plane curve of degree $d$ with $\alpha$ nodes and $\beta$ cusps and no other singularities, then the famous Pl\"ucker formula tells us that $X^*$ is a plane curve of degree \[
    \deg X^*=d(d-1)-2\alpha-3\beta.
    \]
    This formula may be applied, for instance, to count the number of bitangent lines of a plane curve of any degree~\cite[Chapter 2.4]{GriffithHarris}. 
    
    In this paper, we will use \emph{tropical geometry} to study $\tau(X)$ and $X^*$ when $X\subset \PP^n$ is a complete intersection curve. Let $\KK$ be an algebraically closed field of characteristic zero with non-trivial non-Archimedean valuation $\val:\KK\to \RR$ that takes trivial values on the integers. Then $\val$ gives rise to a \emph{tropicalization} map 
    \begin{align*}
    \trop:(\KK^*)^n&\to \RR^n\\
    (x_1,\ldots,x_n)&\mapsto (\val(x_1),\ldots,\val(x_n)).
    \end{align*}
    In most applications,  $\KK$ will be the field of Puiseux series, in which the valuation of the parameter $t$ is chosen to be $1$ (see e.g.~\cite[Example 2.1.3]{tropical}). 
    Given a projective variety $Y\subset \PP^n$, its tropicalization $\trop (Y)$ is the closure of the image of $Y\cap (\KK^*)^n$ under the tropicalization map.
    The set $\trop(Y)$ can be endowed with the structure of a polyhedral complex, and many features of $Y$, including its dimension and degree, may be recovered from $\trop(Y)$.
    
    Our goal in this paper is to describe $\trop(\tau(X))$ and $\trop(X^*)$ when $X\subset \PP^n$ is a complete intersection curve satisfying some mild hypotheses. We emphasize that while tropical geometry is interesting in its own right, our motivation comes from algebraic geometry. In particular, our techniques provide a new method for computing the degrees of $\tau(X)$ and $X^*$, and the Newton polytope of $X^*$.
    
    Tropical geometry has been previously used to study projective dual varieties in a number of situations. Z.~Izhakian showed that for a hypersurface $X\subset \PP^n$ whose defining polynomial has sufficiently generic coefficients, $\trop(X^*)$ may be determined directly from $\trop(X)$ (albeit in a rather non-explicit fashion) \cite{izhakian}. A.~Dickenstein, E.~Feichtner, and B.~Sturmfels gave an explicit description of $\trop(X^*)$ when $X$ is a toric variety \cite{sturmfels}. Finally, the present authors recently gave explicit descriptions of $\trop(X^*)$ when $X$ is a tropically smooth plane curve or surface in three-space \cite{ilten-len}. Tropical geometry has also been used to approach other problems involving tangencies, such as describing the bitangent  lines of complex and real plane curves \cite{LenMarkwig_Bitangents, CuetoMarkwig_Bitangents} and inflection points of real curves \cite{BrugalleMedrano_Inflection}.
    
    The approach we describe below involves describing the tropicalization of the image of the Gauss map. This is related to \emph{tropical elimination theory}, as studied in \cite{elimination}. However, the more explicit results of that paper do not apply in our setting, see Remark \ref{rem:elim} for a discussion. 
    
    \subsection{Our approach and results}
    
    Both the tangential and dual varieties are intimately related to the \emph{Gauss Map}. Given an $m$-dimensional projective variety $X\subset \PP^n$, its Gauss map is the rational map
    \[\G_X:X\dashrightarrow \Gr(m+1,n+1)\]
    sending a smooth point $P\in X$ to its tangent plane $T_PX$. Denote the closure of the image of this map by $\G(X)$. 
    Both $\tau(X)$ and $X^*$ arise as projections of projectivizations of naturally defined vector bundles on $\G(X)$, see \S\ref{sec:bundlep}.

    Our first step is thus to understand $\trop(\G(X))$ for a complete intersection curve $X\subset \PP^n$. Here, we are considering $\G(X)$ embedded via the Pl\"ucker embedding. More specifically, for any point $\alpha\in\trop(X)$, we wish to describe all $\beta\in\trop (\Gr(2,n+1))$ such that there exists a smooth point $P\in X$ with 
    \[
    \trop(P)=\alpha\text{\, and \,}\trop(T_PX)=\beta.
    \]
    By perhaps a slight abuse of terminology, we call such $\beta$ a \emph{tropical tangent} to $\alpha$. 
    
    To give a small taste of our results on tropical tangents, let us fix notation. We set $\TP=\RR^{n+1}/(1,\ldots,1)\cong \RR^n$ and let $e_0,\ldots,e_n$ be the images of the standard basis of $\RR^{n+1}$ in $\TP$. For any subset $J\subset \{0,\ldots, n\}$, we let $\langle J \rangle$ be the span in $\TP$ of those $e_i$ such that $i\in J$. Likewise, for any subset $\Lambda \subset \TP$, let $\langle \Lambda \rangle$ denote the linear span of all differences of elements of $\Lambda$. We use the term \emph{affine tangent space} to describe the standard tangent space, as opposed to the tropical one. That is, given a point $\alpha\in \trop(X)$, its affine tangent space is the affine span of $\trop(X)$ in a small neighborhood of $\alpha$.
    We refer the reader to Sections~\S\ref{sec:basics} and \ref{sec:ci} for the definition of a \emph{tropical complete intersection of tropically smooth hypersurfaces}.
    
    \begin{thm}[See Theorem \ref{thm:genericrestate}]\label{thm:generic}
    Let $X\subset \PP^n$ be a curve such that $\trop(X)\subset \TP$ is a tropical complete intersection of tropically smooth hypersurfaces. Fix a point $\alpha\in\trop(X)$, and let $\Lambda\subset \TP$ denote the affine tangent space to $\trop(X)$ at $\alpha$. Assume that $\langle \Lambda \rangle$ is not contained in any  $\langle J \rangle $ with $|J|=n-1$.
    \begin{enumerate}
    	\item If $\alpha$ is in the relative interior of an edge of $\trop(X)$, then there is a unique tropical tangent $\beta$ to $\alpha$, and its tropical Pl\"ucker coordinates are
    \[
    	\beta_I=\sum_{j\notin I} -\alpha_j\qquad I\subset \{0,\ldots,n\},\ |I|=2.
    \]
    
    \item Suppose that $\alpha$ is a vertex of $\trop(X)$, and $\langle \Lambda\rangle \cap (\langle e_i+e_j\rangle + \langle J'\rangle)=\{0\}$ for all $i,j$ and $J'$ with $|J'|=n-3$. Then $\beta$ is a tropical tangent to $\alpha$ if and only if 
    \[
    	\beta_I=\lambda_I+\sum_{j\notin I} -\alpha_j\qquad I\subset \{0,\ldots,n\},\ |I|=2
    \]
    with all $\lambda_I\geq 0$, at most one $\lambda_I\neq 0$, and $\lambda_I=0$ if $\langle E \rangle \subset \langle \{0,\ldots,n\}\setminus I \rangle$ for an edge $E$ of $\trop(X)$ at $\alpha$.
    \end{enumerate}
    \end{thm}
    \noindent 
    See Corollary \ref{cor:geom} for a restatement of this theorem in which the tropical tangents to $\alpha$ are described in geometric terms, as opposed to via Pl\"ucker coordinates.
    
    The above theorem only describes tropical tangents to $\alpha$ when the affine tangent space of $\trop(X)$ at $\alpha$ is not contained in certain special hyperplanes. However, under relatively mild hypotheses on $X$ we are actually able to compute all tropical tangents regardless of the positioning of the affine tangent space.  See Theorem \ref{thm:alltangents} for a precise statement.
    While our general result is a procedure for computing all tropical tangents, in special cases (such as that of Theorem \ref{thm:generic}) the set of tropical tangents may understood more explicitly.

    Armed with an understanding of the tropicalization of $\G(X)$ for $X$ a complete intersection curve, we then proceed to study $\trop(\tau(X))$ and $\trop(X^*)$. Once have we have computed $\trop(\G(X))$, we readily obtain explicit descriptions of the tropicalizations of the total spaces of the bundles whose projections are $\tau(X)$ and $X^*$. While describing the image of a projection in algebraic geometry requires elimination theory, in tropical geometry it is more straightforward. We thus obtain explicit descriptions of $\trop(\tau(X))$ and $\trop(X^*)$. Under the mild hypotheses we will be making, our descriptions of $\trop(\tau(X))$ and $\trop(X^*)$ depend only on tropical data, and not on the variety $X$ itself. See Theorem \ref{thm:dual}. We also make some simpler statements in special cases, see Propositions \ref{prop:bergman}, \ref{prop:vertcontrib}, and \ref{prop:vertcontribtau}.
    
    The top-dimensional cells in any tropical variety $\trop(Y)$ can be endowed with \emph{multiplicities}. These play a role in intersection theory, and allow one to recover the degree of the original variety $Y$. We will show how to explicitly compute the multiplicities for $\trop(\G(X))$, $\trop(\tau(X))$, and $\trop(X^*)$. See Theorems \ref{thm:gmult}, \ref{thm:mult}, and \ref{thm:multtau}.
    
    \subsection{Organization}
    The rest of the paper is organized as follows. In \S\ref{sec:prelim} we fix some notation and introduce preliminaries. 
    In particular, we prove Theorem \ref{thm:gausszero} which gives a tropical characterization of when the image of the Gauss map of \emph{any} projective variety is contained in a hyperplane where a Pl\"ucker coordinate vanishes.
    In \S\ref{sec:gauss} we cover the tropical Gauss map; this is the technical heart of the paper. 
    After fixing our notation and assumption, we use the lifting result of Osserman and Payne \cite{lifting} to analyze the possible valuations that Pl\"ucker coordinates of tangent lines can have.
    
    In \S\ref{sec:comb} we give combinatorial interpretations of some aspects of our description of $\trop(\G(X))$. This gives insight into the structure of tropical tangents, and allows us to prove Theorem \ref{thm:generic}. We turn our attention to $\trop(\tau(X))$ and $\trop(X^*)$ in \S\ref{sec:bundle}. We show how to obtain them from $\trop(\G(X))$, and proceed to analyze some special cases. We conclude in \S\ref{sec:mult} with a discussion of multiplicities. After first computing multiplicites for the tropicalization of the graph of the Gauss map, we then show how to recover them for $\trop(\G(X))$, $\trop(\tau(X))$, and $\trop(X^*)$.
    
    Throughout the paper, we consider two running examples: a cubic plane curve
(see Examples \ref{ex:runningP}, \ref{ex:runningP2}, \ref{ex:nuP}, \ref{ex:critP},  \ref{ex:runningnocancelP}, \ref{ex:simultP}, \ref{ex:combP1}, \ref{ex:multP}, \ref{ex:multP2}) and a degree $9$ curve in $\PP^3$ (see Examples \ref{ex:running}, \ref{ex:running2}, \ref{ex:nu}, \ref{ex:crit},  \ref{ex:runningnocancel}, \ref{ex:simult}, \ref{ex:comb1}, \ref{ex:comb2}, \ref{ex:rundt}, \ref{ex:runningtrop}, \ref{ex:mult}, \ref{ex:mult2}).
    Using the techniques described in this paper, we are able to recover the Newton polytope of the projective duals of these curves.  
      
      We conclude the introduction with an overview of the plane curve example (Example \ref{ex:runningP}).
       While duals of tropically smooth plane curve were already computed in \cite{ilten-len}, the current paper allows us to deal with the case of edge multiplicity as well, leading to considerably different behaviour than in the smooth case: those multiplicities contribute directly to the multiplicity of edges of the dual, and there are additional multiplicities resulting  
    from projection from the co-normal variety. 
    
    \begin{ex}[A curve in $\PP^2$]\label{ex:planeCurve}
	    Let $X$ be the curve \[V(x_0^2x_1 + x_0^2x_2 + x_1^2x_2)\subseteq \PP^2.\]
    Its tropicalization is pictured in Figure \ref{fig:planarExample} and consists of a vertex $V$, and edges $E_+,E_-,E'$ (see Example \ref{ex:runningP}).
    The edge $E'$ has multiplicity $2$ (see Example \ref{ex:multP}).
        Since the dimension of the ambient space is $n=2$, the image of the Gauss map  ${\G}(X)$ coincides with the dual curve $X^*$ after identifying the $J$-th Pl\"ucker coordinate of the former with the $\overline J$-th coordinate of the latter, where $|J| = 2$ and $\overline J = \{0,1,2\}\setminus J$. 
        
The edges $E_+,E_-,E'$ contribute respectively the edges $-E_+,-E_-,-E'$ to $\trop(X^*)$ (see Example \ref{ex:runningnocancelP}).
Likewise, the vertex $V$ contributes an edge $\RR_{\geq 0}\cdot e_1$ (see Example \ref{ex:simultP}).

The multiplicities of these edges are labeled in Figure \ref{fig:planarExample}; see Example \ref{ex:multP2} for the computation.

  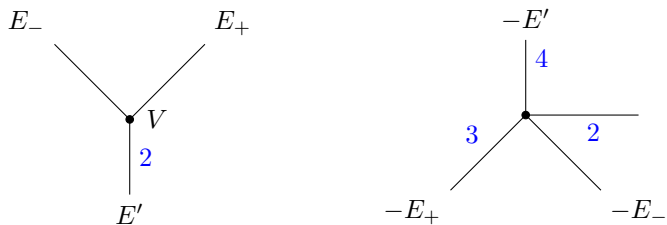
\begin{figure}
    	\begin{tikzpicture}[scale=.5]
	
	\draw (0,0) -- (0,-2);
	\draw (0,0) -- (-2,2);
	\draw (0,0) -- (2,2);
	\draw[fill] (0,0) circle [radius=0.1];
	\node[right] at (0,-1) {\color{blue}$2$};
	\node[below] at (0,-2) {$E'$};
	\node[above left] at (-2,2) {$E_-$};
	\node[above right] at (2,2) {$E_+$};
	\node[right] at (0.2,0) {$V$};
\end{tikzpicture}
\qquad\qquad
\begin{tikzpicture}[scale=.5]
	\draw (0,0) -- (0,2);
	\draw (0,0) -- (2,-2);
	\draw (0,0) -- (-2,-2);
		\draw (0,0) -- (3,0);
	\draw[fill] (0,0) circle [radius=0.1];
	\node[right] at (0,1.5) {\color{blue}$4$};
	\node[below] at (1.8,0) {\color{blue}$2$};
	\node[above left] at (-1,-1) {\color{blue}$3$};

	\node[above] at (0,2) {$-E'$};
	\node[below right] at (2,-2) {$-E_-$};
	\node[below left] at (-2,-2) {$-E_+$};
       	\end{tikzpicture}
    	\caption{A plane tropical curve and its tropical dual (with multiplicities)}\label{fig:planarExample}
    \end{figure}

 In passing, we remark that by choosing a family $\trop(X_t)$ of tropically smooth tropical curves converging to $\trop(X)$ in the example above, the tropical duals of $X_t$ (which could have been computed via the results of \cite{ilten-len}) converge to the tropical dual of $X$ just computed.     
 The reason this holds true in this special case is because the curves $X_t$ and $X$ all necessarily have the same singularities: a single ordinary double point.
 \end{ex}

    \subsection*{Acknowledgements} The first author was partially supported by NSERC. We thank Jake Levinson for useful conversations and Hannah Markwig for helpful comments on an earlier draft. We also thank the anonymous referee for their very insightful suggestions and remarks.

    \section{Preliminaries}\label{sec:prelim}
    \subsection{Tropical basics}\label{sec:basics}
    We refer the reader to \cite{tropical} for an introduction to tropical geometry.
    As mentioned in the introduction, we will always be working over an algebraically closed field $\KK$ of characteristic zero along with a non-trivial non-Archimedean valuation $\val:\KK\to \RR\cup\{\infty\}$. Here we use that convention that $\val(0)=\infty$. We will furthermore assume that $\val$ takes trivial values on the integers.
    
    Using $\val$, we have a tropicalization map
    \begin{align*}
    \trop:(\KK^*)^n&\to \RR^n\\
    (x_1,\ldots,x_n)&\mapsto (\val(x_1),\ldots,\val(x_n)).
    \end{align*}
    Given any subvariety $Y\subset (\KK^*)^n$, the closure of its image  $\trop(Y)$ under the tropicalization map is the support of a finite polyhedral complex~\cite[Corollary 3.24]{tropical}.
    
    Let $M$ be the character lattice of $(\KK^*)^n$, and for any $u\in M$, denote by $x^u$ the corresponding regular function on the torus.
    Consider any monomial\footnote{Throughout this paper, we use \emph{monomial} as a synonym for a term of a polynomial, that is, our monomials are allowed to have coefficients.}  $f=c x^u\in \KK[M]$ and any point $\alpha\in\RR^n$. For $P\in (\KK^*)^n$ tropicalizing to $\alpha$, the valuation of $f(P)$ is
    \[
    \val(f(P))=\alpha \cdot u+\val(c).
    \]
    This is independent of $P$. We will refer to this as the valuation of $f$ at $\alpha$, or simply as  $\val(f)$, when $\alpha$ is clear from the context.
    
    In the special case where $Y$ is a hypersurface in $(\KK^*)^n$, its tropicalization is especially simple to describe. Let $\A$ be a finite subset of $M$, and consider a regular function of the form
    \[f=\sum_{u\in \A} c_u x^u,\] where $c_u\in \KK^*$.
    Kapranov's theorem  (see e.g.~\cite[Theorem 3.1.3]{tropical})
    states that 
    \[
    	\trop(V(f))=\{\alpha\in\ \RR^n\ |\ \alpha\cdot u +\val(c_u)\ \textrm{obtains its minimum for at least two}\ u\in\A\}.
    \]
    This set may be given the structure of a polyhedral complex by saying that $\alpha,\alpha'$ are in the relative interior of the same cell if $\alpha\cdot u +\val(c_u)$ and
    $\alpha'\cdot u +\val(c_u)$ obtain their minima for the same $u\in \A$.
    Such a hypersurface (or its tropicalization) is said to be \emph{tropically smooth} if the regular subdivision of $\conv{\A}$ induced by the lower convex hull of \[\{(u,\val(c_u))\}_{u\in\A}\subset \RR^n\times \RR\]
    is a unimodular triangulation.

    \subsection{Complete intersections}\label{sec:ci}
    Given $k$ hypersurfaces $Y_1,\ldots,Y_k\subset (\KK^*)^n$, we say that they (or $\trop(Y_1),\ldots,\trop(Y_k)$) form a \emph{tropical complete intersection}
     if for any $k$-tuple of cells $(E_1,\ldots,E_k)$ with $E_i\subset \trop(Y_i)$, the intersection $E_1\cap\cdots \cap E_k$ is either empty, or has dimension equal to 
     \[
    (1-k)n+\sum_{i=1}^k \dim E_i.
     \]
     Equivalently, 
     \[
    \codim (E_1\cap\cdots \cap E_k )=   \sum_{i=1}^k \codim E_i.
     \]
     
          It follows from Theorem \ref{thm:lifting} below that if this condition holds, then
	  \[\trop(Y_1\cap\ldots\cap Y_k)=\trop(Y_1)\cap\ldots\cap\trop(Y_k),\] and moreover   $Y_1,\ldots,Y_k$ are a complete intersection. However, the converse is not true in general. Nonetheless, we will show in Section \ref{sec:genericity} that  the intersection of sufficiently general tropical hypersurfaces is in fact a complete intersection.
    
    \subsection{Tropicalizing projective varieties}\label{sec:proj}
    We will primarily be interested in tropicalizing projective varieties.
    Let $T\subset \PP^n$ be the open torus given by the non-vanishing of all homogeneous coordinates.
    For $Y\subset \PP^n$, we denote by $\trop(Y)$ the tropicalization of $Y\cap T$.  The character lattice $M$ of $T$ is most naturally thought of as 
    \begin{equation}\label{eqn:M}
    	M=\left\{ w\in \ZZ^{n+1}\ | \sum w_i=0\right\}
    \end{equation}
    with the regular function $x^w$ on $T$ given by $\prod_{i=0}^n x_i^{w_i}$. Dually, the codomain of the tropicaliziation map is $\TP:=\Hom(M,\ZZ)\otimes \RR=\RR^{n+1}/(1,\ldots,1)\cong \RR^n$. 
    
    For $i=0,\ldots, n$, we will denote by $e_i$  the image in $\TP$ of the $i$th standard basis vector in $\RR^{n+1}$. In particular, $\sum_{i=0}^n e_i=0$. 
    We will make frequent use of the following notation: given some subset $J\subseteq \{0,\ldots,n\}$, we set 
    \begin{equation*}
    \langle J \rangle = \vspan \{e_i\ |\ i\in J\}\subseteq\TP,
    \end{equation*}
    the vector space spanned by the vectors $e_i$. 
    In examples throughout this paper, we will use the basis $e_1,\ldots,e_n$ to express elements of $\TP$ as $n$-tuples. In other words, we will often write a point $(\alpha_0,\ldots,\alpha_n)\in\TP$ as $(\alpha_1-\alpha_0,\ldots,\alpha_n-\alpha_0)$.

Throughout the paper, we will consider two examples. The first is a plane curve:    
    \begin{ex}[A curve in $\PP^2$]\label{ex:runningP}
    We will follow this example throughout the paper.     We consider
    \begin{align*}
    X&=(x_0^2x_1 + x_0^2x_2 + x_1^2x_2)\subset \PP^2,
    \end{align*}
    a plane cubic curve.

     Restricting to the torus $T$ and working with regular functions on it, we have
    \begin{align*}
	    X\cap T&=V(1+x_1^{-1}x_2+x_0^{-2}x_1x_2)
\end{align*}
and $\trop(X)$ consists of a vertex $V=(0,0)$ and edges
\begin{align*}
    	E_+=\RR_{\geq 0}\cdot (1,1)\qquad E_-=\RR_{\geq 0} (-1,1)\\
    	E'=\RR_{\geq 0}\cdot (0,-1).
    \end{align*}
    We note here that the edge $E'$ has multiplicity two (see \S \ref{sec:mbasics} and Example \ref{ex:multP}).
    The tropical curve $\trop(X)$ is pictured in Figure \ref{fig:planarExample}.
    \end{ex}

    Our second example is a curve in $\PP^3$:
	    \begin{ex}[A curve in $\PP^3$]\label{ex:running}
    	We will follow this example throughout the paper. We work over the field $\KK$ of Puiseux series in $t$ with coefficients in $\CC$ (see e.g.~\cite[Example 2.1.3]{tropical}). The valuation $\val$ picks out the smallest exponent of $t$.
    We consider
    \begin{align*}
    X_1&=V(x_3^4+x_0^2x_2x_3+x_0x_2^3)\subset \PP^3\\
    X_2&=V(x_1x_3^2+x_1^2x_2+t^3x_0x_3^2)\subset \PP^3\\
    \end{align*}
    The intersection of $X_1$ and $X_2$ consists of two components: a degree $9$ curve $X$, and a non-reduced version of the line $V(x_2,x_3)$. One can compute (e.g. with \texttt{Macaulay2} \cite{M2}) that $X$ has arithmetic genus $8$ and two singular points. Although $X$ is not a complete intersection, its intersection with the torus $T$ is a complete intersection.
    
    Restricting to the torus $T$ and working with regular functions on it, we have
    \begin{align*}
    	X_1\cap T&=V(1+x_0^2x_2x_3^{-3}+x_0x_2^3x_3^{-4})\\
    	X_2\cap T&=V(1+x_1x_2x_3^{-2}+t^3x_0x_1^{-1}).
    \end{align*}
    The hypersurfaces $X_1$ and $X_2$ form a tropical complete intersection, and $\trop(X)$ consists of the vertices
    \[
    	V_1=(0,0,0)\qquad V_2=(3,1,2)\qquad V_3= (3,0,0)
    \]
    and the edges
    \begin{align*}
    &	E_1=\conv \{(0,0,0),(3,1,2)\}\qquad E_2=\conv \{(3,1,2),(3,0,0)\}\\
    &E_3=(3,1,2)+\RR_{\geq 0}\cdot (3,2,4)\qquad
    	E_4=\RR_{\geq 0}\cdot (-1,3,1)\qquad E_5=\RR_{\geq 0} (-2,-4,-3)\\
    &	E_6=(3,0,0)+\RR_{\geq 0}\cdot (0,3,1)\qquad E_{7}=(3,0,0)+\RR_{\geq 0}\cdot (0,-4,-3).
    \end{align*}

    See Figure \ref{fig:surfaces} for depictions of $\trop(X_1)$ and $\trop(X_2)$ with $\trop(X)$ sitting inside.
    See Figure \ref{fig:proj} for the projection of $\trop(X)$ to $\RR^2$ under the map $(\alpha_1,\alpha_2,\alpha_3)\mapsto(\alpha_1-\alpha_3,\alpha_2)$.
    \end{ex}
    
    \begin{figure}
	    \begin{center}
	    \begin{tabular}{c c}
	    \includegraphics[scale=.4]{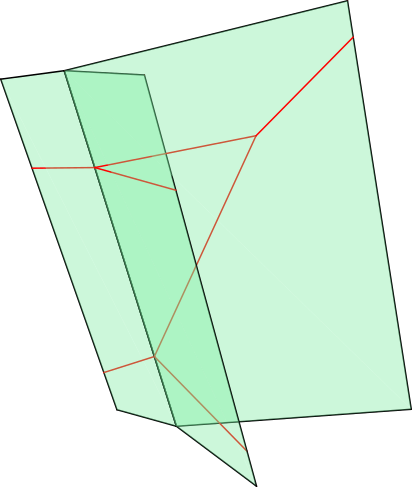} &
	    \includegraphics[scale=.4]{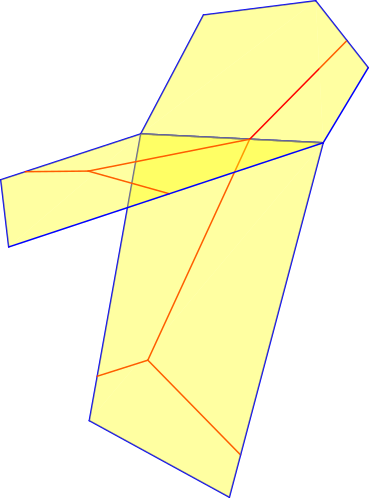}\\
	    $\trop(X_1)$ &$\trop(X_2)$
    \end{tabular}
    \end{center}
\caption{The tropical surfaces $\trop(X_1)$ and $\trop(X_2)$ from Example \ref{ex:running}}.\label{fig:surfaces}
    \end{figure}

    \begin{figure}
    	\begin{tikzpicture}
    \draw (0,0) -- (1,1) -- (3,0);
    \draw (-2,3) -- (0,0) -- (.5,-2);
    \draw (2,3) -- (3,0) -- (4.5,-2);
    \draw (1,1) -- (0,3);
    \draw[fill] (1,1) circle [radius=0.05];
    \draw[fill] (0,0) circle [radius=0.05];
    \draw[fill] (3,0) circle [radius=0.05];
    	\node [left ] at (0,0) {$V_1$};
    	\node [below ] at (1,.9) {$V_2$};
    	\node [right ] at (3,0) {$V_3$};
    	\node [above left] at (.5,.3) {$E_1$};
    	\node [below left] at (2.3,.5) {$E_2$};
    	\node  at (.8,2) {$E_3$};
    	\node  at (-1.7,2) {$E_4$};
    	\node  at (.5,-1) {$E_5$};
    	\node  at (2.6,2) {$E_6$};
    	\node  at (3.5,-1) {$E_7$};
    	\end{tikzpicture}
    	\caption{A projection of $\trop(X)$ from Example \ref{ex:running}}.\label{fig:proj}
    \end{figure}
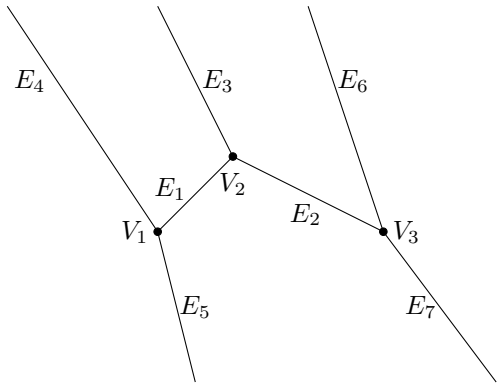

    \subsection{Lowest order parts}
    Let $\Gamma_{\val}\subseteq\RR$ be the image of $\val$. By ~\cite[Lemma 2.1.15]{tropical}, the valuation $\val$ splits in the sense that there is a group homomorphism $\phi:\Gamma_{\val}\to\KK^*$ such that $\val(\phi(\lambda)) = \lambda$ for all $\lambda\in\Gamma_{\val}$. We will fix such a splitting once and for all.
    We let $\FF$ be the residue field of the valuation ring $\KK_{\geq 0}$ by its maximal ideal $\KK_{>0}$. Here $\KK_{\geq 0}$ consists of all elements of $\KK$ of non-negative valuation, and $\KK_{>0}$ those of positive valuation.

    We use the splitting $\phi$ to define the \emph{lowest order part} of some element
    $y\in\KK^*$ as the image in $\FF$ of $y/\phi(\val(y))$. Note that if two elements in $\KK^*$ have the same valuation, then the lowest order part of their sum is the sum of their lowest order parts, unless the lowest order parts sum to zero. Similarly, for any two elements of  $\KK^*$, the lowest order part of their product is the product of their lowest order parts.
    We say that a set of polynomials has \emph{sufficiently general lowest order parts} if the lowest order parts of the coefficients avoid an implicitly specified Zariski closed subset of the space of lowest order parts of the coefficients.
    
    In the special case where $\KK$ is the field of Puiseux series $\CC\{\{t\}\}$, we have $\Gamma_{\val} = \QQ$ and $\FF = \CC$. There is a canonical splitting $\phi:\QQ \to \KK$ sending $\lambda\in\QQ$ to $t^\lambda\in\KK$. With regards to this splitting, the lowest order part of a Puiseux series is just the (complex) coefficient in front of the term with the smallest exponent. 
    
    \subsection{Osserman-Payne lifting}
    To show that certain points in the tropical Grassmannian lift to give tropical tangents, we will make extensive use of the following lifting result, due to B.~Osserman and S.~Payne \cite{lifting}. The form we state here may be found in \cite[Corollary 4.2]{ilten-len}:
    
    \begin{thm}\label{thm:lifting}
    Let $Y_1,\ldots,Y_k$ be hypersurfaces of $(\KK^*)^n$, and 
    \[
    \alpha\in \trop(Y_1)\cap \cdots\cap \trop(Y_k).
    \]
    Assume that in a neighborhood of $\alpha$, the codimension of 
    $\trop(Y_1)\cap \cdots\cap \trop(Y_k)$ in $\RR^n$ is equal to $k$. Then there exists $P\in Y_1\cap \cdots \cap Y_k$ with $\trop(P)=\alpha$.
    \end{thm}
    
    \subsection{Grassmannians}
    We denote by $\Gr(m+1,V)$ the Grassmannian parametrizing $(m+1)$-dimensional subspaces of a vector space $V$.
    For $V=\KK^{n+1}$, we write $\Gr(m+1,n+1):=\Gr(m+1,V)$.
    If $V$ is $(n+1)$-dimensional with dual space $V^*$, there is a canonical isomorphism \[\Gr(m+1,V)\to \Gr(n-m,V^*)\] taking an $(m+1)$-plane $L\subset V$ to its orthogonal complement $L^\perp\subset V^*$.
    
    We may represent any $(m+1)$-dimensional linear subspace $L\subset \KK^{n+1}$ non-uniquely as the row span of an $(m+1)\times (n+1)$ matrix $A_L$. 
    For $J\subset \{0,\ldots,n\}$ of size $m+1$, let $p_J(L)$ be the determinant of the square submatrix of $A_L$ with columns indexed by $J$.
    The Grassmannian $\Gr(m+1,n+1)$ is embedded in $\PP^{{n+1\choose m+1}-1}$ via the \emph{Pl\"ucker embedding}, which sends $L$ to the tuple $(p_J(L))_{J}$. The $p_J$ are called \emph{Pl\"ucker coordinates}.
    
    For any subset $J\subset \{0,\ldots,n\}$, we will denote by $\overline J$ its complement $\{0,\ldots,n\}\setminus J$.
    Let $p_I$ (where $|I|=m+1$) and $p_J^*$ (where $|J|=n-m$) respectively be Pl\"ucker coordinates for $\Gr(m+1,\KK^{n+1})$ and $\Gr(n-m,(\KK^{n+1})^*)$. A straightforward computation with determinants shows that the canonical isomorphism \[\Gr(m+1,\KK^{n+1})\to \Gr(n-m,(\KK^{n+1})^*)\] may be described in Pl\"ucker coordinates via
    \begin{equation}\label{eqn:dualp}
    	p_J^*=(-1)^{\sum J}p_{\overline J},\qquad \sum J:=\sum_{j\in J} j.
    \end{equation}
    
    \subsection{Vanishing of Pl\"ucker coordinates}
    Let $X\subset \PP^n$ be an $m$-dimensional variety. 
    Recall that $\G_X:X\dashrightarrow\Gr(m+1,n+1)$ is the Gauss map sending a smooth point $P\in X$ to the point of $\Gr(m+1,n+1)$ corresponding to $T_P X$.
    We are interested in which Pl\"ucker coordinates vanish on the image  of this map.
    
    \begin{thm}\label{thm:gausszero}
    	Consider a projective variety $X\subset\PP^n$ intersecting the dense torus non-trivially. Let $\G(X)$ be the closure of the image of the Gauss map. Then the $J$th Pl\"ucker coordinate $p_J$ vanishes on $\G(X)$ if and only if for every maximal cell $E$ of $\trop(X)$, $\langle E\rangle \cap \langle \overline J \rangle\neq \{0\}$. In particular, if $X$ is a curve, $p_J$ vanishes on $\G(X)$ if and only if $\langle \trop(X)\rangle \subseteq \langle \overline J \rangle$.
    \end{thm}
    \begin{proof}
    	Let $m$ be the dimension of $X$. Fix some index set $J\subset\{0,\ldots,n\}$ of size $m+1$. Let $\pi:X\dashrightarrow \PP^m$ be the projection onto the coordinates $x_i$ for $i\in J$. The condition $p_J=0$ is equivalent to the differential $d\pi_P:T_P X\to T_{\pi(P)}\PP^m$ being non-injective for every $P$ (where $\pi$ is defined). It follows that $p_J=0$ is equivalent to the condition $\dim(\pi(X))<m$. 
    	
    	We may test this latter condition tropically. Indeed, $\dim(\pi(X))<m$ if and only if $\dim \trop(\pi(X))<\dim \trop(X)$. But $\trop(\pi(X))$ is the same thing as the projection of $\trop (X)$ to $\TP$ sending $e_i$ to $0$  if and only if $i\in \overline J$. Thus, $\dim(\trop(\pi(X)))<\dim (\trop(X))$ if and only if for every $m$-dimensional cell $E$ of $\trop(X)$, $E$ intersects $\langle \overline J \rangle$ non-trivially. This completes the proof of the first claim.
    
    	For the special case of curves, we note that $\langle E \rangle \cap \langle \overline J \rangle \neq \{0\}$ is equivalent to $\langle E \rangle \subseteq \langle \overline J \rangle$ since $\dim \langle E \rangle=1$. The second claim follows.
    \end{proof}
      
    \section{The Tropical Gauss Map}\label{sec:gauss}
    \subsection{Setup}\label{sec:setup}
    We will be considering a curve $X\subset \PP^n$ whose restriction to $(\KK^*)^n$ is the complete intersection of hypersurfaces $X_1,\ldots,X_{n-1}$. Given $\alpha \in \trop(X)$, we wish to determine all $\beta\in\trop(\Gr(2,n+1))$ such that there exists a smooth point $P\in X$ satisfying 
    $\trop(P)=\alpha$ and $\trop(\G_X(P))=\beta$. Here, $\G_X:X\dashrightarrow\Gr(2,n+1)$ is the Gauss map sending a smooth point $P\in X$ to the point of $\Gr(2,n+1)$ corresponding to $T_P X$. We call $\beta$ a \emph{tropical tangent} of $\alpha$.
    
    The tropicalization of $\Gr(2,n+1)$ that we consider will be with respect to its Pl\"ucker embedding. However, it may occur that the image of the Gauss map $\G(X)$ lies outside of the torus, in which case we will effectively be intersecting $\G(X)$ with a lower-dimensional torus. More precisely, we consider $\G(X)$ as a variety inside the smaller-dimensional projective space where we have eliminated all Pl\"ucker coordinates which vanish on $\G(X)$. This is a special case of the extended tropicalization considered in \cite[Section 2]{Payne_AnalytificationTropicalization}. See Theorem \ref{thm:gausszero} for a tropical characterization of when $\G(X)$ is contained in such a smaller projective space.
    
 In order to maintain reasonable control over tropical tangents, we will require some genericity hypotheses.     
 \begin{ass}\label{ass:setup}
	 We will always assume that $X_1,\ldots,X_{n-1}$ form a tropical complete intersection (see \S \ref{sec:ci}). Furthermore, we will assume that for $i=1,\ldots, n-1$ and any co-dimension one cell $E$ of $\trop(X_i)$, the defining equation of $X_i$ has exactly \emph{three} lowest order terms when evaluated along points tropicalizing to the relative interior of $E$.   In particular, the defining equation of $X_i$ has exactly \emph{two} lowest order terms when evaluated along points tropicalizing to the relative interior of a maximal face.
       \end{ass}
    As we are assuming that $X$ is a tropical complete intersection, it follows that $X$ is a trivalent tropical curve.
    This is fulfilled in particular if the $X_i$ are all \emph{tropically smooth} (see \S \ref{sec:basics}), although we do not require that here. 
  
    \begin{rem}
	In order to determine the tropical tangents of $\trop(X)$ at $\alpha\in E\subset \trop(X)$, it suffices to consider hypersurfaces $X_1,\ldots,X_{n-1}$ such that $\trop(X)$ is a tropical complete intersection of the $\trop(X_i)$ only in a neighborhood of $\alpha$ in $\TT^n$.

	For example, consider the planes $X_1=V(x_0+x_1+x_2+x_3)$ and $X_2=V(x_0+t x_1+t^{-1}x_2+x_3)$ in $\PP^3$, and let $X$ be their intersection.
Then $\trop(X_1)$ has a single vertex at $(0,0,0)$ and $\trop(X_2)$ has a single vertex at $(-1,1,0)$. The tropical line $\trop(X)$ is not equal to $\trop(X_1)\cap \trop(X_2)$, since both $\trop(X_1)$ and $\trop(X_2)$ contain the two-dimensional cell 
\[
e_2+	\RR_{\geq 0}\cdot e_1+\RR_{\geq 0}\cdot e_2.
\]
However, $\trop(X_1)$ and $\trop(X_2)$ intersect properly along the ray
\[
	e_2+\RR_{\geq 0}\cdot e_3
\]
and the results of this paper apply to the points $\alpha$ in its relative interior.
		    \end{rem}

    \begin{rem}[Connection to Tropical Elimination Theory]\label{rem:elim}
    As discussed above, we are interested in the tropicalization of $\G(X)$. This can be viewed as a problem in tropical elimination theory. The graph of the morphism $\G_X:X\dashrightarrow  \Gr(2,n+1)$ is a complete intersection (cut out by equations for the $X_i$ along with equations coming from the map). The variety $\G(X)$ is then obtained via a projection.
    
    In \cite[\S4]{elimination}, there is an explicit recipe for the image under a monomial map of the tropicalization of a complete intersection, provided that the complete intersection has \emph{generic coefficients}. However, this does not apply in our setting, since the equations cutting out the graph of $\G_X$ have very special coefficients --- this is the case even if one imposes genericity conditions on the $X_i$. In order to proceed as in \cite{elimination}, one would need to find a suitable resolution of a compactification of the graph of $\G_X$ as discussed in \cite[pg. 560--561]{elimination}.
    
    In the remainder of this article, we will take a different approach to understanding $\trop(\G(X))$, although tropical elimination theory will play a role when we move on to $\trop(\tau(X))$ and $\trop(X^*)$.
    \end{rem}
    
    \subsection{Notation}\label{sec:notation}
    We now fix a cell $E$  in $\trop(X)$. Working in $(\KK^*)^n$, each $X_i$ is cut out by a Laurent polynomial $f_i$ of total degree zero. With $M$ as in \eqref{eqn:M} we may write
    \[
    	f_i=\sum_{w\in \A_i}c_{iw}x^{w}
    \]
    for some finite subset $\A_i\subset M$ and non-zero coefficients $c_{iw}\in\KK$.
    We consider the valuation of the terms of $f_i$ at any $\alpha\in E$. The terms of $f_i$ having minimal valuation depend only on $E$, not on $\alpha$. \emph{After multiplying each $f_i$ by a monomial, we may and will assume that the terms of minimal valuation include the monomial $1$, and all other terms of minimal valuation similarly have valuation zero.}

    If $E$ is an edge, then since $\trop(X)$ is a tropical complete intersection, each $f_i$ only has two terms of minimal valuation. We thus write
    \begin{align}\label{eqn:twoterm}
    	f_i&=1+c_{i\bv_i}x^{\bv_i}
    	+\HOT\qquad i=1,\ldots,n-1,
    \end{align}
    where $\bv_i\in \A_i$, and $c_{i\bv_i}x^{\bv_i}$ also has valuation zero when evaluated at a point tropicalizing to the relative interior of $E$.
    We use the notation $\HOT$ to represent all terms that have higher than minimal valuation for the relative interior of $E$.
    We thus obtain an $(n-1)$-tuple $\bv=(\bv_1,\ldots,\bv_{n-1})\in \prod \A_i$.
    
    If on the other hand $E$ is a vertex, there is one index $i=i_0$ for which $f_i$ has three terms of minimal valuation when evaluated. 
    For $i\neq i_0$ we may write $f_i$ as in \eqref{eqn:twoterm}. For $i=i_0$ we write
    \begin{align}\label{eqn:threeterm}
    	f_{i_0}&=1+c_{i_0\bv_{i_0}}x^{\bv_{i_0}}+c_{i_0\bv_{i_0}'}x^{\bv_{i_0}'}+\HOT.
    \end{align}
    Setting $\bv_i'=\bv_i$ for $i\neq i_0$, we  obtain $(n-1)$-tuples 
    \[\bv=(\bv_1,\ldots,\bv_{n-1})\in \prod \A_i\qquad \bv'=(\bv_1',\ldots \bv_{n-1}')\in \prod \A_i.\]
    
    We will frequently have occasion to consider such $(n-1)$-tuples of elements of $M$. In particular, for $\bw=(\bw_1,\ldots,\bw_{n-1})\in\prod_{i=1}^{n-1}\A_i$, we set
    \[x^\bw:=\prod x^{\bw_i}\qquad c_\bw:=\prod c_{i\bw_i}\qquad a_\bw=c_\bw\cdot x^\bw.\]
    We will denote such $(n-1)$-tuples of elements of $M$ in boldface (e.g.~$\bw$); components of the tuple will be denoted with subscripts (e.g.~$\bw_i$). The $n+1$ components of an element of $M$ will be denoted with a further subscript (e.g.~$\bw_{ij}$ for $\bw_i\in \A_i$, or $w_j$ for $w\in M$).
    We set $\A=\prod \left(\A_i\setminus\{0\}\right)$.
    We will often view an element $\bw\in \A$ as the $(n-1)\times (n+1)$ matrix whose rows are $\bw_1,\ldots,\bw_{n-1}$.

    Consider any $(n-1)\times (n+1)$ matrix $A$ with rows indexed by $1,\ldots,n-1$ and columns indexed by $0,\ldots, n$. For any subset $J\subset \{0,\ldots,n\}$ of size $2$, we denote by $\Delta_J(A)$ the determinant of the $(n-1)\times (n-1)$ submatrix of $A$ with the columns of $J$ removed. 
    We will see that the values of $\Delta_J(\bw)$ for $\bw\in \A$ play a key role in determining the tropical tangents of $X$.
    If $J=\{i,j\}$, we will often write $\Delta_{ij}(\bw)$ instead of $\Delta_J(\bw)$.
    
    \begin{warning}
    	The notation of $f_i$ and $\A$ and $\bv$ we have introduced in this section \emph{depends on the choice of} $E\subset \trop(X)$. Whenever we use such notation, we will have (at least implicitly) fixed some cell $E$ of $\trop(X)$.
    \end{warning}
    
    \begin{ex}[Example \ref{ex:runningP} continued (a curve in $\PP^2$)]\label{ex:runningP2}
    We now compute $\bv$ and when appropriate also $\bv'$ arising from different edges and vertex of the tropical plane curve $\trop(X)$. Note that, due to the assumption that one of the monomials of minimal valuation is always $1$, we may need to multiply the defining equations by a monomial  when considering different faces.

    For the cells $V,E_+,E_-$, we may take 
    \begin{align*}
	    f_1&=1+x_1^{-1}x_2+x_0^{-2}x_1x_2.
    \end{align*}
    We then obtain the following exponent vectors:
    {\scriptsize{
    \[
    \begin{array}{l| l l l }
    &V&E_+&E_-\\
    \hline
    \bv_1 & (0,-1,1) & (0,-1,1) & (-2,1,1)\\
    \bv_1' & (-2,1,1)
    \end{array}.
    \]}}
    
    On the other hand, for the edge $E'$ we may take
    \begin{align*}
	    f_1&=x_1x_2^{-1}+1+x_0^{-2}x_1^2
        \end{align*}
and obtain the exponent vector $\bv_1=(-2,2,0)$.
    \end{ex}
    \begin{ex}[Example \ref{ex:running} continued (a curve in $\PP^3$)]\label{ex:running2}
    We compute the $2$-tuples $\bv$ and when appropriate also $\bv'$ arising from different edges and vertices of the tropical space curve $\trop(X)$. Note that, due to the assumption that one of the monomials of minimal valuation is always $1$, we may need to multiply the defining equations by a monomial  when considering different faces.

    For the cells $V_1,V_3,E_4,E_5,E_6,E_7$, we may take 
    \begin{align*}
    f_1&=1+x_0^2x_2x_3^{-3}+x_0x_2^3x_3^{-4}\\
    f_2&=1+x_1x_2x_3^{-2}+t^3x_0x_1^{-1}
    \end{align*}
    with $i_0=1$ for $V_1,V_3$.
    We then obtain the following exponent vectors:
    {\scriptsize{
    \[
    \begin{array}{l| l l l l l l}
    &V_1&V_3&E_4&E_5&E_6&E_7\\
    \hline
    \bv_1 & (2,0,1,-3) &(2,0,1,-3) & (2,0,1,-3) &(1,0,3,-4) & (2,0,1,-3) &(1,0,3,-4)\\
    \bv_1' & (1,0,3,-4) & (1,0,3,-4)  \\
    \bv_2 & (0,1,1,-2) & (1,-1,0,0) &(0,1,1,-2) & (0,1,1,-2) &(1,-1,0,0)&(1,-1,0,0)\\
    \end{array}.
    \]}}
    
    On the other hand, for the cells $V_2,E_1,E_3$ we may take
    \begin{align*}
    	f_1&=x_0^{-2}x_2^{-1}x_3^3+1+x_0^{-1}x_2^2x_3^{-1}\\
    	f_2&=x_1^{-1}x_2^{-1}x_3^2+1+t^3x_0x_1^{-2}x_2^{-1}x_3^2
    \end{align*}
    with $i_0=2$ for $V_2$.
    We  obtain the following exponent vectors:
    {\scriptsize{
    \[
    \begin{array}{l| l l l }
    &V_2&E_1&E_3\\
    \hline
    \bv_1 & (-1,0,2,-1)& (-1,0,2,-1)& (-1,0,2,-1)\\
    \bv_2 & (0,-1,-1,2) & (0,-1,-1,2) & (1,-2,-1,2)\\
    \bv_2' & (1,-2,-1,2)
    \end{array}.
    \]}}
    
    Finally, for $E_2$ we may take
    \begin{align*}
    f_1&=x_0^{-2}x_2^{-1}x_3^3+1+x_0^{-1}x_2^2x_3^{-1}\\
    f_2&=1+x_1x_2x_3^{-2}+t^3x_0x_1^{-1}
    \end{align*}
    and obtain exponent vectors 
    \[
    \bv_1=(-1,0,2,-1)\qquad \bv_2=(1,-1,0,0).
    \]
    \end{ex}
    \subsection{Pl\"ucker coordinates}\label{sec:plucker}
    To compute the Gauss map, 
    let $J\subset \{0,\ldots n\}$ be a set with $2$ elements. We let $q_J$ be the determinant of the $(n-1)\times (n-1)$ submatrix of the Jacobian matrix for the $f_i$ obtained by removing the $J$ columns, that is, 	
    \[
    	q_J=\Delta_J\left(\frac{\partial f_i}{\partial x_j} \right).
    \]
    Consider a smooth point $P\in X$.
    As $J$ ranges over all subsets of $\{0,\ldots,n\}$ of size $2$, each $q_J$ (evaluated at homogeneous coordinates representing $P$) gives the Pl\"ucker coordinate $p_{\overline J}$ for the point in $\Gr(n-1,n+1)$ corresponding to the orthogonal complement in $(\PP^n)^*$ of $T_P X\subset \PP^n$. 
    By \eqref{eqn:dualp}, the valuations of the Pl\"ucker coordinates $p_{J}(T_P X)$ are given by the $\val(q_J)$ for $q_J$ evaluated at homogeneous coordinates representing $P$. Rescaling these homogeneous coordinates by an element of $\KK^*$ will result in a translation of  $(\val(q_J))_J$ by a multiple of the vector $(1,\ldots,1)$.
    
    Hence, we wish to determine the valuations of the $q_J$. To avoid the above ambiguity arising from different choices of homogeneous coordinates for $P$, it will be convenient to instead work with \[z_J:=q_J\cdot \prod_{j\notin J} x_j.\]
    Since this is homogeneous of degree zero, it is a regular function on $(\KK^*)^n$.
    A straightforward computation shows 
    \begin{equation}\label{eqn:zJ}
    	z_J=\sum_{\bw\in\A} \Delta_J(\bw)c_\bw x^{\bw}
    	=\sum_{\bw\in\A} \Delta_J(\bw)a_\bw.
    \end{equation}
    For any $\alpha$ in the relative interior of $E$, set
    \[
    	\nu_\bw(\alpha)=\val(c_\bw)+\alpha\cdot \left(\sum_{i=1}^{n-1}\bw_i\right).
    \]
    In other words, $\nu_\bw(\alpha)$ is the valuation of $a_\bw=c_\bw x^{\bw}$ at $\alpha$. 
    We then set
    		\begin{equation}\label{eqn:min}
    			\nu_J(\alpha)=\min \{\nu_\bw(\alpha)\ |\ \bw\in\A\ \textrm{and}\ \Delta_J(\bw)\neq 0\}.
    		\end{equation}
    If there are no $\bw\in\A$ with $\Delta_J(\bw)\neq 0$, we set $\nu_J(\alpha)=\infty$.		
    We see that $\nu_J(\alpha)\geq 0$. 
    When $E$ is an edge, inequality holds if and only if
    $\Delta_J(\bv)=0$. When $E$ is a vertex, inequality holds if and only if $\Delta_J(\bv')=\Delta_J(\bv)=0$.
    
    By \eqref{eqn:zJ}, for any point $P\in X$ tropicalizing to $\alpha$ in the relative interior of $E$, we then have
    		\begin{equation}\label{eqn:pz}
    	\val(q_J)=\val(z_J)-\sum_{j\notin J}\alpha_j\geq \nu_J(\alpha)-\sum_{j\notin J}\alpha_j.
    \end{equation}
    In the following, we will analyze what valuations $\val(z_J)\geq \nu_J(\alpha)$ are possible. 
    \begin{ex}[Example \ref{ex:runningP} continued (a curve in $\PP^2$)]\label{ex:nuP}
	    In this example, $\Delta_J(\bv)\neq 0$ except when $J=\{1,2\}$ and $\alpha$ belongs to $V$, or $E_+$, or $J=\{0,1\}$ and $\alpha$ belongs to $E'$.
    	It follows that except in these cases, $\nu_J=0$.
    For the exceptional cases, we record the value $\nu_{J}(\alpha)$ along with the $\bw\in \A$ satisfying $\nu_\bw=\nu_{J}$ and $\Delta_{J}(\bw)\neq 0$:
    	\vspace{1cm}
    
    	\begin{center}	\begin{tabular}{l l l l l}
    			&$J$&$\nu_{J}(\alpha)$&$\bw$ with $\nu_\bw=\nu_{J}$ and $\Delta_{J}(\bw)\neq 0$\\
    \toprule
    $V$ & $\{1,2\}$& $0$ &$((-2,1,1))$\\
    $E_+$& $\{1,2\}$ & $\alpha_1+\alpha_2$& $((-2,1,1))$\\
    $E'$& $\{0,1\}$ &$\alpha_1-\alpha_2$&  $( (0,1,-1) )$\\
    \bottomrule
    	\end{tabular}
    \end{center}
	Here, we are choosing coordinates for $\alpha$ such that $\alpha_0=0$.
    \end{ex}
    
    \begin{ex}[Example \ref{ex:running} continued (a curve in $\PP^3$)]\label{ex:nu}
    	In this example, one computes that $\Delta_J(\bv)\neq 0$ except when $J=\{0,1\}$, and $\alpha$ belongs to $V_3$, $E_2$, $E_6$, or $E_7$. 
    	It follows that except in these cases, $\nu_J=0$.
    For the exceptional cases, we record the value $\nu_{01}(\alpha)$ along with the $\bw\in \A$ satisfying $\nu_\bw=\nu_{01}$ and $\Delta_{01}(\bw)\neq 0$:
    	\vspace{1cm}
    
    	\begin{center}	\begin{tabular}{l l l l}
    			&$\nu_{01}(\alpha)$&$\bw$ with $\nu_\bw=\nu_{01}$ and $\Delta_{01}(\bw)\neq 0$\\
    \toprule
    $V_3$ & $3$ &$((2,0,1,-3),(0,1,1,-2))$ and $((1,0,3,-4),(0,1,1,-2))$\\
    $E_2$& $\alpha_1+3\alpha_2-3\alpha_3$& $((-1,0,2,-1),(0,1,1,-2))$\\
    $E_6$& $\alpha_1+2\alpha_2-5\alpha_3$& $((2,0,1,-3),(0,1,1,-2))$\\
    $E_7$& $\alpha_1+4\alpha_2-6\alpha_3$& $((1,0,3,-4),(0,1,1,-2))$.\\
    \bottomrule
    	\end{tabular}
    \end{center}
	Here, we are choosing coordinates for $\alpha$ such that $\alpha_0=0$.
    \end{ex}
    
    \subsection{Genericity conditions and cancellative pairs}\label{sec:genericity}
  We begin this subsection by verifying that tropical complete intersections are generic (with respect to valuations of coefficients).       
    As in Section \ref{sec:notation}, fix finite sets $\A_i\subset M$, where $i=1,2,\ldots,m$ for some $m$, and let
    \begin{equation}\label{spaceOfValuations}
    \CSP =  \prod_i\RR^{|\A_i|} .
    \end{equation}
The set $\CSP$ is the parameter space for the valuations of the coefficients of polynomials $f_1,\ldots,f_m$ with support $\A_1,\ldots,\A_m$.
       \begin{lemma}\label{lemma:completeIntersections}
   There is a finite number of hyperplanes in $\CSP$ such that for any point  $\gamma\in \CSP$ in the complement of the union of these hyperplanes, and for every collection of polynomials $f_1,\ldots,f_m$ with support $\A_1,\ldots,\A_m$ and coefficients tropicalizing to $\gamma$, $V(f_1),\ldots,V(f_m)\subset (\KK^*)^n$ either have empty intersection or define a tropical complete intersection. \end{lemma}
    
    \begin{proof}
    Let $\Sigma_1,\Sigma_2,\ldots,\Sigma_{m}$ be tropical hypersurfaces that do not form a complete intersection.  If their intersection is not empty then there are faces $F_1,F_2,\ldots, F_{m}$ whose intersection $F$ has dimension  greater than $n(1-m) + \sum_{i=1}^{m} \dim F_i$. 
    For each face $F_i$, the polynomial $f_i$ defining    $\Sigma_i$ contains monomials $c_{i\br_i^j}x^{\br_i^j}$ for $i=1,\ldots,k_i$,  (where $k_i = n-\dim F_i +1$ and each $\br_i\in\A_i$), such that the valuations $\val(c_{i\br^j_i}x^{\br^j_i})_{j=1,\ldots,k_i}$ all coincide for every $x$ tropicalizing to any point in $F$. We therefore have $k_i-1$ equations of the form 
    \[
    \val(c_{i\br^j_i}x^{\br^j_i}) =  \val(c_{i\br^1_i}x^{\br^1_i}),
    \]
    or equivalently
    \[
    \val(x)(\br^j_i - \br^1_i) = c_{i\br^1_i} - c_{i\br^j_i}.
    \]
 The fact that the dimension of $F$ is greater than expected implies that the collection of vectors $\br^j_i - \br^1_i$, as  $i$ and $j$ vary, spans a space of dimension smaller than $k:=\sum (k_i-1)$.  In particular, the set 
        \[
    B = \big((\br^1_i- \br^j_i)\cdot\val(x)\big)_{i,j}
    \]
 as $\val(x)$ varies is not full-dimensional in $\RR^k$. 

  Now  consider the map
    \[
    \phi:\CSP\to\RR^{k},
    \]
    \[
    \phi(c) =  \left(c_{i\br^1_i} - c_{i\br^j_i}\right)_{i,j}.
       \]
Since $\phi$ is surjective onto $\RR^k$, $\phi^{-1}(B)$ is not full-dimensinal in $\CSP$. But $\phi^{-1}(B)$ contains the tropicalization of the coefficients of tropical hypersurfaces whose intersection is not a tropical complete intersection.    Note that $B$ and therefore $\phi^{-1}(B)$ only depended on the exponents and not on the coefficients. Removing these subspaces for all choices of $\br_i^j$ ensures that any resulting intersection of hypersurfaces will be a tropical complete intersection.
       \end{proof}

  In order to completely determine the possible valuations of the $z_J$ we will need to introduce  more terminology and make additional genericity assumptions.    
    Let $\alpha$ be in the relative interior of $E\subset \trop(X)$. Two distinct elements $\br,\br'\in\A$ are a \emph{cancellative pair for $\alpha$} if $\val (a_\br)=\val(a_{\br'})$ at $\alpha$. The \emph{critical locus} of $\trop(X)$ consists of those $\alpha\in \trop(X)$ for which there is a cancellative pair. Note that every vertex of $\trop(X)$ is in the critical locus.
    
    \begin{defn}	\label{def:nonColliding}
    We say that the hypersurfaces $X_i$ have \emph{non-colliding valuations} if 
    \begin{enumerate}
    	\item The critical locus is finite; and
    \item \label{def:nonColliding.condition} Whenever $\alpha$ is in the critical locus and $\br,\br'\in\A$ are a cancellative pair for $\alpha$, any other cancellative pair for $\alpha$ must have the form $\bw,\bw':=\bw+\br'-\br$ for some $\bw\in\A$ satisfying $\bw_i=\br_i$  for all $i$ with $\br_i\neq \br_i'$.\footnote{Note that the critical locus and the property of having non-colliding valuations do not depend on any of the choices we made in \S \ref{sec:notation}.}
    
    \end{enumerate}
    \end{defn}
    	\noindent For $\alpha$ in the critical locus, having non-colliding valuations implies that any $\bw\in \A$ belongs to at most one cancellative pair for $\alpha$.
     If $\alpha$ is a vertex and we have non-colliding valuations, it follows that any cancellative pair $\br,\br'$ must satisfy $\br_1=\bv_1$, $\br_1'=\bv_1'$, and $\br_i=\br_i'$ for $i>1$.
    
    Non-colliding valuations is a purely tropical property in the sense that it only depends on the exponents and on the valuations of the coefficients. 
    As we shall  see in Lemma \ref{lem:vg}, coefficients with generically chosen valuations are non-colliding.

    We will occasionally need a stronger condition:
    \begin{defn}\label{defn:genval} For $\alpha$ in the relative interior of $E\subset \trop(X)$, let $\cS_\alpha$ be the multiset of real numbers
    			\[
    				\cS_\alpha=\{\val(c_{iw}x^{w})\}_{\substack{i=1,\ldots,n-1\\ w\in\A_i\setminus\{0,\bv_i\}}},
    				\]
    where the valuations are taken at $\alpha$. 
    We say that the hypersurfaces $X_i$ have \emph{very general valuations} if 
    	\begin{enumerate}
    		\item The critical locus of $\trop(X)$ is finite; and
    		\item\label{defn:genval2} For any point in the critical locus, $\cS_\alpha$ has exactly one linear dependence over $\QQ$. 
    			\end{enumerate}
    \end{defn}

 As the next lemma shows, the term    very general valuations is justified.

 \begin{lemma}\label{lem:vg}
        If the $X_i$ have very general valuations, then they also have non-colliding valuations.
Moreover, after removing a finite (respectively countable) number of hyperplanes from $\CSP$, any choice of valuations gives rise to hypersurfaces with non-colliding (respectively very general) valuations. 
    \end{lemma}
Before beginning the proof, we introduce one more piece of terminology.
    Given polynomials $f_1,\ldots,f_{n-1}$ defining a tropical complete intersection $\trop(X)$, their \emph{combinatorial type} 
consists of the following data.
\begin{itemize}
\item The abstract graph $\Gamma$ underlying $\trop(X)$; 
\item For each edge or vertex $E$ of $\Gamma$ and for each $i=1,\ldots,n-1$, the set of exponent vectors in  $\A_i$ such that the corresponding monomials of $f_i$ obtain the  minimal valuation along $E$.
\end{itemize}
It is straightforward to see that there are only finitely many possible combinatorial types, and that these combinatorial types induce a polyhedral subdivision of $\CSP$.

\begin{proof}[Proof of Lemma \ref{lem:vg}]
	    Consider any point $\alpha$ in the critical locus.
We first note that any cancellative pair  $\br,\br'$ for $\alpha$ gives rise to a non-trivial linear relation in $S_\alpha$. Indeed, considering valuations 
at $\alpha$, we see that
        \[
    \sum_i\val(c_{i\br_i}x^{\br_i})-\val(c_{i\br'_i}x^{\br'_i}) = 0.
\]
Disregarding the terms coming from the monomials $c_{i\bv_i}x^{\bv_i}$ (which have valuation $0$ at $\alpha$) and the terms with $\br_i=\br_i'$, we obtain a non-trivial linear relation in $S_\alpha$ with all coefficients of the linear combination in $\{-1,0,1\}$.

Suppose now that at $\alpha$, the only non-trivial linear relation in $S_\alpha$ whose coefficients are in $\{-1,0,1\}$ is the above relation coming from $\br,\br'$ (or its negative). 
Let $\bw,\bw'$ be any cancellative pair for $\alpha$.
Then after possibly interchanging $\bw$ and $\bw'$, we have that the non-trivial relation in $S_\alpha$ for $\bw,\bw'$ is the same as that for $\br,\br'$.
This implies that if $\br_i=\br_i'$, then $\bw_i=\bw_i'$, and if $\br_i\neq \br_i'$, we must have $\bw_i=\br_i$ and $\bw_i'=\br_i'$.
We have thus shown in particular that the property of having very general valuations implies having non-colliding valuations.

We next turn to show that we can remove a finite (respectively countable) number of hyperplanes from the space $\CSP$ to achieve non-colliding (respectively very general) valuations. 
    From Lemma \ref{lemma:completeIntersections},  after removing a finite number of hyperplanes from $\CSP$, every tuple of hypersurfaces gives rise to a  tropical complete intersection curve. We will now fix a combinatorial type $\Gamma$ and an edge or vertex $E$ of $\Gamma$. We will show that we can remove a finite number of hyperplanes from $\CSP$ so that the resulting tropical complete intersections of combinatorial type $\Gamma$ have only finitely many critical points along $E$. Furthermore, we will show that we can remove an additional finite (respectively countable) number of hyperplanes so that these critical points satisfy the requirements of non-colliding (respectively very general) valuations. Since there are only finitely many combinatorial types, and each type only has finitely many edges and vertices, the claim of the theorem will follow.

     For any curve $X$ of type $\Gamma$, using the  notation of Section \ref{sec:notation},  there is a tuple  $\bv\in\A$ such that
    \[
    \val(c_{i\bv_i}x^{\bv_i}) = 0
    \]
    for every $i$ and for every $x\in X$ tropicalizing to a point in the relative interior of $E$. That is,
    \[
    \bv_i\cdot\val(x)  = - \val(c_{i\bv_i})
    \]
    for $i=1,\ldots, n-1$.

    Now suppose that $E$ is an edge and $\br$ and $\br'$ are a cancellative pair for a point $x\in X$ tropicalizing to $E$. Then 
    \[
    \sum_i\left(\val(c_{i\br_i}x^{\br_i})-\val(c_{i\br'_i}x^{\br'_i})\right) = 0,
    \]
    namely
    \[
    \val(x)\cdot \sum_i  (\br_i - \br'_i) = \sum_i(\val(c_{i\br'_i})-\val(c_{i\br_i})).
    \]
    If $\sum_i  (\br_i - \br'_i)$ meets $\bv_1,\bv_2,\ldots,\bv_{n-1}$ transversally in $\TP\cong \RR^n$ then there is a single value of $\val(x)$ (and in particular at most a single point on the edge $E$) satisfying these equations. 
    Otherwise, the image $A$ of the map
    \begin{align*}
\RR^n&\to \RR^n\\
\alpha&\mapsto (\bv_1\cdot \alpha,\ldots,\bv_{n-1}\cdot \alpha, \sum_i((\br_i-\br'_i))\cdot \alpha)
\end{align*}
is a subspace of $\RR^n$ of dimension at most $n-1$. Consider the map
    \[
    \CSP\to\RR^n,
    \]
    \[
    c\mapsto \left(c_{1\bv_{1}}, c_{2\bv_{2}},\ldots, c_{(n-1)\bv_{n-1}}, \sum_i(c_{i\br'_i} - c_{i\br_i})\right).
    \]
    The preimage of $A$ in $\CSP$ by the above map contains  the tropicalizations of coefficients of tropical curves that may have more than a single critical point on an edge. Since the map is surjective, the preimage of $A$ in $\CSP$ is  not full-dimensional. Removing these preimages of $A$ for any choice of $\br,\br'$ ensures that we have a finite number of critical points in $E$.

   Fix a point $\alpha\in E$ in the critical locus.
   As above, there is a tuple  $\bv\in\A$ such that
       \[
    \bv_i\cdot \alpha  = - \val(c_{i\bv_i})
    \]
    for $i=1,\ldots, n-1$.
     Furthermore, since $\alpha$ is in the critical locus, we have
    \[
    \val(x)\cdot \sum_i  (\br_i - \br'_i) = \sum_i(\val(c_{i\br'_i})-\val(c_{i\br_i}))
    \]
    for some critical pair $\br,\br'$.
  
    Let us now check Condition \ref{defn:genval2} of Definition \ref{defn:genval}.  Suppose that we have a linear relation in $S_\alpha$ that is independent of the linear relation coming from the critical pair $\br,\br'$. Then for any $x$ with $\trop(x)=\alpha$, there exist rational numbers $\lambda_{iw}$ with 
    \[
    \sum_{i,w} \lambda_{iw} \val(c_{iw} x^w) = 0
    \]
    or equivalently,
    \[
    \val(x)\cdot\sum_{i,w} a_{iw} w  = -\sum_{i,w} a_{iw} \val(c_{iw}).
    \]
Furthermore, since this relation is independent of the one coming from $\br,\br'$, the map 
    \[
    \CSP\to\RR^{n+1},
    \]
    \[
    c\mapsto \left(c_{1\bv_{1}}, c_{2\bv_{2}},\ldots, c_{(n-1)\bv_{n-1}}, \sum_i(c_{i\br'_i} - c_{i\br_i}),\sum_{i,w} \lambda_{iw}c_{iw}  \right).
    \]
    is surjective.
   
    Let  $B$ be the image of the linear map
    \begin{align*}
	    \RR^n&\to\RR^{n+1}\\
	    \alpha&\mapsto (\bv_1\cdot \alpha,\bv_2\cdot \alpha,\ldots,\bv_{n-1}\cdot \alpha,  \sum_i  (\br_i - \br'_i)\cdot \alpha,(\sum_{i,w} a_{iw} w)\cdot \alpha ).
    \end{align*}
    The dimension of $B$ is at most $n$. Since the above map from $\CSP$ to $\RR^{n+1}$ is surjective, the preimage of $B$ in $\CSP$ is not full-dimensional. Furthermore, the preimage of $B$ in $\CSP$ by the map contains the tropicalization of the coefficients of the tropical curves for which 
the given linear relation among the elements of $S_\alpha$ holds. Since there are only countably many possible linear relations over $\QQ$, we may thus remove a countable number of hyperplanes of $\CSP$ so that for the remaining tropical complete intersections, $S_\alpha$ only has a single linear dependence.

Thus, we have shown that we may remove a countable number of hyperplanes from $\CSP$ so that the resulting tropical complete intersections have very general valuations. For the weaker condition of non-colliding valuations, we note that as in the start of the proof, we only have to consider finitely many linear dependencies among the elements of $S_\alpha$, namely, those with coefficients in $\{-1,0,1\}$. Thus, we may remove a finite number of hyperplanes from $\CSP$  so that the resulting tropical complete intersections have non-colliding valuations.
\end{proof}

    \begin{rem}
	    Lemma \ref{lem:vg} implies that in particular, there is a Zariski dense set in the space of coefficients of the polynomials $f_i$ for which the $X_i$ have very general valuations. While being Zariski dense does not, in itself, imply genericity, it can be used to prove various generic properties, e.g. \cite[Corollary 1.3]{ilten-len}.
    
    \end{rem}
    \begin{ex}[Example \ref{ex:runningP} continued (a curve in $\PP^2$)]\label{ex:critP}
    The only element of the critical locus for the tropical plane curve $\trop(X)$ is the vertex $V$.
    \end{ex}

    \begin{ex}[Example \ref{ex:running} continued (a curve in $\PP^3$)]\label{ex:crit}
    	We describe the critical locus for the tropical space curve $\trop(X)$, whose projection is depicted in Figure \ref{fig:crit}.
    Clearly $V_1,V_2,V_3$ are in the critical locus as all vertices are. The only cancellative pairs for $\alpha=V_1$ are 
    \begin{align*}
    	\bv&=((2,0,1,-3),(0,1,1,-2)),\bv'=((1,0,3,-4),(0,1,1,-2)); \textrm{and}\\
    \br&=((2,0,1,-3),(1,-1,0,0)),\br'=((1,0,3,-4),(1,-1,0,0))
    \end{align*}
    and $\br'=\br+\bv'-\bv$ as required in the definition of non-colliding valuations. One similarly checks that $\alpha=V_2,V_3$ satisfy the requirements of non-colliding valuations.
    
    The critical locus contains the following additional points:
    \begin{align*}
    	\left(\frac{9}{8},\frac{3}{8},\frac{6}{8}\right)\in E_1\qquad 
    	\left(3,\frac{3}{8},\frac{6}{8}\right)\in E_2\qquad 
    	\left(\frac{-3}{4},\frac{9}{4},\frac{3}{4}\right)\in E_4\\
    	(-2,-4,-3)\in E_5 \qquad \left(3,\frac{9}{4},\frac{3}{4}\right)\in E_6\qquad (3,-4,-3)\in E_7.
    \end{align*}
    For each such point, there is only a single cancellative pair. Hence, $X_1$ and $X_2$ have non-colliding valuations. In fact, they have very general valuations, since $S_\alpha$ always contains two elements, with at most one of them equal to zero.
    \end{ex}
    
    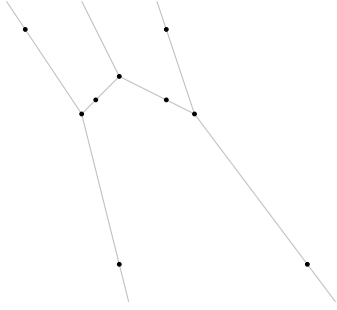
\begin{figure}
    	\begin{tikzpicture}[scale=.5]
    		\draw[lightgray] (0,0) -- (1,1) -- (3,0);
    		\draw[lightgray] (-2,3) -- (0,0) -- (1.25,-5);
    		\draw[lightgray] (2,3) -- (3,0) -- (6.75,-5);
    		\draw[lightgray] (1,1) -- (0,3);
    \draw[fill] (1,1) circle [radius=0.05];
    \draw[fill] (0,0) circle [radius=0.05];
    \draw[fill] (3,0) circle [radius=0.05];
    \draw[fill] (.375,.375) circle [radius=0.05];
    \draw[fill] (2.25,.375) circle [radius=0.05];
    \draw[fill] (-1.5,2.25) circle [radius=0.05];
    \draw[fill] (1,-4) circle [radius=0.05];
    \draw[fill] (2.25,2.25) circle [radius=0.05];
    \draw[fill] (6,-4) circle [radius=0.05];
    	\end{tikzpicture}
    \caption{The critical locus for Example \ref{ex:crit}}\label{fig:crit}
    \end{figure}
    
    \subsection{No cancellation}
    We now return to our study of what values the $\val(z_J)$ can obtain when evaluated at a point $P\in X$ tropicalizing to $\alpha\in E\subset \trop(X)$. In fact, the following proposition shows that in ``generic'' situations, $\val(z_J)=\nu_J$.
    
    \begin{prop}\label{prop:nocancel}
    	Consider $J\subset \{0,\ldots,n\}$ of size $2$.
    	The valuation of $z_J$ for the tropical point $\alpha\in\trop(X)$ is uniquely determined and equal to $\nu_J$
    in the following cases:
    \begin{enumerate}
    	\item	 The point $\alpha$ is not in the critical locus;\label{c1}
    	\item  The $\bw$ achieving the minimum in \eqref{eqn:min} is not part of a cancellative pair $\br,\br'$ with both
    		$\Delta_J(\br)$ and $\Delta_J(\br')$ nonzero;\label{c2}
    	\item The $X_i$ have non-colliding valuations, $\alpha$ is a vertex, and the $\bw$ achieving the minimum in \eqref{eqn:min} is part of a cancellative pair $\br,\br'$ with $\Delta_J(\br)=\Delta_J(\br')$; or\label{c3}
    	\item  The $X_i$ have non-colliding valuations and sufficiently general lowest order parts, $\alpha$ is not a vertex, and the $\bw$ achieving the minimum in \eqref{eqn:min} is part of a cancellative pair $\br,\br'$ such that $\sum_i(\br-\br')_i$ belongs to the vector space spanned by $\{\bv_i\}_{i=1}^{n-1}$.\label{c4}
    \end{enumerate}
    \end{prop}
    \begin{proof}
    We evaluate $z_J$ at some point $P\in X$ tropicalizing to $\alpha$.
    In the cases \ref{c1} and \ref{c2}, the expression for $z_J$ from \eqref{eqn:zJ} has a unique (non-zero) term of minimal valuation, and the claim is immediate. In any other case, we notice that $\val(z_J)$ must be at least the quantity in \eqref{eqn:min}. Suppose that we are in case \ref{c3} but equality does not hold in \eqref{eqn:min}. Since $\alpha$ is a vertex and we have assumed non-colliding valuations, our previous discussion on cancellative pairs implies that $\br_{i_0}=\bv_{i_0}, \br'_{i_0}=\bv'_{i_0}$, and $\br_i=\br'_i$ for $i\neq i_0$. It follows that 
    \[
    	(\Delta_J(\br)a_\br+\Delta_J(\br')a_{\br'})\cdot\frac{1}{\Delta_J(\br)\prod_{i\neq i_0} c_{i\br_i}x^{\br_i}}=c_{i_0\bv_{i_0}}x^{\bv_{i_0}}+c_{i_0\bv_{i_0}'}x^{\bv_{i_0}'}.
    \]
    Therefore when evaluated at $P$, the left hand side of 
    \[
    	f_{i_0}-\frac{1}{\Delta_J(\br)\prod_{i\neq {i_0}} c_{i\br_i}x^{\br_i}}\left(\sum_{\bw\in\A} \Delta_J(\bw)a_\bw-z_J\right)=0
    \]
    has $1$ as its unique term of minimal valuation, contradicting its vanishing.
    Hence, equality must hold in case \ref{c3}.
    
    For the final case, let $\br,\br'$ be the cancellative pair from the hypothesis. By non-colliding valuations, the only terms of $z_J$ of minimal valuation are $\Delta_J(\br)a_\br$ and $\Delta_J(\br')a_\br'$. By case \ref{c2} we may assume that neither $\Delta_J(\br)$ nor $\Delta_J(\br')$ vanish.
    Since neither $\br$ nor $\br'$ can equal $\bv$, it must be the case that there are indices $i,j$ such that $\br_i\neq \bv_i$, $\br_j'\neq \bv_j$, and $(\br_i,i)\neq (\br_j',j)$. 
    
    We use Lemma \ref{lemma:finite} below to see that there are only finitely many possible lowest order parts of $x^{\br}$ and $x^{\br'}$,
    and these are independent of $c_{i\br_i}$ and $c_{j\br_j'}$.
    By choosing the lowest order parts of $c_{i\br_i}$ and $c_{j\br_j'}$ sufficiently generally after fixing all other lowest order parts of coefficients, we see that the lowest order parts of $\Delta_J(\br)a_\br$ and $\Delta_J(\br')a_{\br'}$ cannot cancel, so the valuation of $z_J$ must be as claimed.
    \end{proof}
    
    \begin{rem}
    	In Lemmas \ref{lemma:edge} and \ref{lemma:vert} we will give combinatorial interpretations of some of the criteria of Proposition \ref{prop:nocancel}.
    
    \end{rem}
    
    The following lemma was used in the proof above.
    \begin{lemma}\label{lemma:finite}
    Fix a non-vertex $\alpha \in\trop(X)$.
    	Let $L$ be the intersection of $M$ with the vector space spanned by the vectors $\{\bv_i\}_{i=1}^{n-1}$, and consider any $w\in L$. As $P$ varies over all $P\in X$ tropicalizing to $\alpha$, there are only finitely many possible lowest order parts of $x^w$ evaluated at $P$. Furthermore, these depend only on the lowest order parts of $c_{i\bv_i}$.
    
    \end{lemma}
    \begin{proof}
    Consider any  $P\in X$ tropicalizing to $\alpha$. 
    The inclusion $L\subseteq M$ induces a surjection of tori $\pi:\spec \KK[M]\to \spec \KK[L]$. Taking the lowest order parts of $\pi(P)\in\spec \KK[L]$, we obtain $Q\in \spec \FF[L]$. These lowest order parts are a solution in $\spec \FF[L]$ for
    	\[
    		1+ \widetilde c_{i\bv_i}x^{\bv_i},\qquad i=1,\ldots n-1
    	\]
    	where  $\widetilde c_{i\bv_i}$ is the lowest order part of $c_{i\bv_i}$. Since the $\bv_i$ span $L$, this is a zero-dimensional system and only has finitely many solutions. The claims of the lemma follow. 
    	\end{proof}
    
    	\begin{ex}[Example \ref{ex:runningP} continued (a curve in $\PP^2$)]\label{ex:runningnocancelP}
    	Clearly, Proposition \ref{prop:nocancel} applies to all $\alpha\in \trop(X)$ that are not the vertex $V$.
 We record the resulting valuations of the $q_J$ in Table \ref{table:edgesP}. For this, we use \eqref{eqn:pz}, along with the computations in Example \ref{ex:nuP}. We have highlighted the entries of the table for which $\val(z_J)\neq 0$ by putting them in a box.

 We note that Proposition \ref{prop:nocancel} also applies when $\alpha=V$ and $J=\{1,2\}$ or $J=\{0,1\}$. 
       We will finish computing tropical tangents at $V$ for this example  
       in Example \ref{ex:simultP}.
	\end{ex}

    	\begin{table}	\scriptsize\begin{tabular}{l l l l l }
    			$E$ & $\alpha\in E$ &  $\val(q_{12})$& $\val(q_{02})$&$\val(q_{01})$\\
    			\toprule
			$E_+$ & $s\cdot (1,1)$ & $\fbox{2s}$ & $-s$& $-s$\\ \addlinespace
    			$E_-$ & $s\cdot (-1,1)$ & $0$& $s$& $-s$ \\ \addlinespace
			$E'$ & $s\cdot (0,-1)$ & $0$& $0$ & $\fbox{2s}$\\ \addlinespace
    		    \bottomrule
    	\end{tabular}
    	\vspace{.5cm}
    
    	\caption{Tropical tangents for edges in Example \ref{ex:runningnocancelP}}\label{table:edgesP}
    \end{table}

    	\begin{ex}[Example \ref{ex:running} continued (a curve in $\PP^3$)]\label{ex:runningnocancel}
    Proposition \ref{prop:nocancel} applies to all $\alpha\in \trop(X)$ that are not in the critical locus. However, it also applies in this example to the points of the critical locus that are not vertices. Indeed, by our computations in Example \ref{ex:nu}, for any point $\alpha$ in an edge, the $\bw$ achieving the minimum in \eqref{eqn:min} is not part of a cancellative pair. Hence, case \ref{c2} of Proposition \ref{prop:nocancel} applies. 
    
    We record the resulting valuations of the $q_J$ in Table \ref{table:edges}. For this, we use \eqref{eqn:pz}, along with the computations in Example \ref{ex:nu}. We have highlighted the entries of the table for which $\val(z_J)\neq 0$ by putting them in a box.
    We will finish computing tropical tangents for this example by determining the tropical tangents at $V_1,V_2,V_3$ 
    in Example \ref{ex:simult}. We note here that Proposition \ref{prop:nocancel} also applies to $z_{01}$ for $V_2$, since at this point,
    $\Delta_{01}(\bv)=\Delta_{01}(\bv')$.
    
    \end{ex}
    	\begin{table}	\scriptsize\begin{tabular}{l l l l l l l l}
    			$E$ & $\alpha\in E$ & $\val(q_{23})$& $\val(q_{13})$ &$\val(q_{12})$& $\val(q_{03})$ &$\val(q_{02})$&$\val(q_{01})$\\
    			\toprule
    			$E_1$ & {\makecell[l]{$s\cdot (3,1,2)$,\\ $0\leq s \leq 1$}} & $-3s$ & $-s$ & $-2s$ & $-4s$ & $-5s$ & $-3s$\\ \addlinespace
    			$E_2$ & {\makecell[l]{$(3,0,0)+s\cdot (0,1,2)$,\\ $0\leq s \leq 1$}} & $-3$ & $-s$ & $-2s$ & $-3-s$ & $-3-2s$ & \fbox{$3-6s$}\\ \addlinespace
    			$E_3$ & {\makecell[l]{$(3,1,2)+s\cdot (3,2,4)$,\\ $0\leq s$}} & $-3-3s$ & $-1-2s$ & $-2-4s$ & $-4-5s$ & $-5-6s$ & $-3-6s$\\ \addlinespace
    			$E_4$ & {\makecell[l]{$s\cdot (-1,3,1)$,\\ $0\leq s$}} & $s$ & $-3s$ & $-s$ & $-2s$ & $0$ & $-4s$\\ \addlinespace
    			$E_5$ & {\makecell[l]{$s\cdot (-2,-4,-3)$,\\ $0\leq s$}} & $2s$ & $4s$ & $3s$ & $6s$ & $5s$ & $7s$\\ \addlinespace
    			$E_6$ & {\makecell[l]{$(3,0,0)+s\cdot (0,3,1)$,\\ $0\leq s$}} & $-3$ & $-3s$ & $-s$ & $-3-3s$ & $-3-s$ & \fbox{$3-3s$}\\ \addlinespace
    			$E_7$ & {\makecell[l]{$(3,0,0)+s\cdot (0,-4,-3)$,\\ $0\leq s$}} & $-3$ & $4s$ & $3s$ & $-3+4s$ & $-3+3s$ & \fbox{$3+9s$}\\ \addlinespace
    \bottomrule
    	\end{tabular}
    	\vspace{.5cm}
    
    	\caption{Tropical tangents for edges in Example \ref{ex:runningnocancel}}\label{table:edges}
    \end{table}
    
    \subsection{Cancellation}
    	We continue our study of $\val(z_J)$, and deal with the cases not covered by Proposition \ref{prop:nocancel}:
    
    	\begin{prop}\label{prop:cancel}
    		Let $J\subset \{0,\ldots,n\}$ be of size $2$ and consider $\alpha$ in the critical locus of $\trop(X)$.
    		Assume that the $X_i$ have non-colliding valuations, and the $\bw$ achieving the minimum in \eqref{eqn:min} is part of a cancellative pair $\br,\br'$ with both $\Delta_J(\br)$ and $\Delta_J(\br')$ nonzero. Then for this $\alpha$, $\val(z_J)$ can be any quantity greater than or equal to $\nu_J$ if one of the following holds:
    	 \begin{enumerate}
    		 \item If $\alpha$ is a vertex, and $\Delta_J(\br)\neq \Delta_J(\br')$; or
    		 \item If $\alpha$ is not a vertex, and $\sum_i(\br-\br')_i$ does not belong to the vector space spanned by $\{\bv_i\}_{i=1}^{n-1}$.
    	 \end{enumerate}
    \end{prop}
    \begin{proof}
    	In both cases, we will use Osserman-Payne lifting (Theorem \ref{thm:lifting})  to show that for any $\lambda > \nu_J(\alpha)$, there exists $(P,z_J)\in X\times \KK^*$  such that
    	\eqref{eqn:zJ}
    	is satisfied and $\trop(P,z_J)=(\alpha,\lambda)$.
    	In both cases,  $\val(z_J)=\lambda>\nu_J$ implies that the terms of minimal valuation (at $\alpha$) of the right-hand-side of \eqref{eqn:zJ} are $\Delta_J(\br)a_\br+\Delta_J(\br)a_{\br'}$.
    	
    For the first case, first fix any $z_J$ with $\val(z_J)>\nu_J$. We define 
    \[
    	f_{i_0}':=f_{i_0}-\frac{1}{\Delta_J(\br')\prod_{i\neq i_0} c_{i\br_i'}x^{\br_i'}}\left(\sum_{\bw\in\A} \Delta_J(\bw)a_\bw-z_J\right).
    \]
    The terms of minimal valuation (at $\alpha$) of $f_{i_0}'$ are $1$ and some non-zero multiple of $c_{i_0\bv_{i_0}}x^{\bv_{i_0}}$. For this, we are using that
    $\br_{i_0}=\bv_{i_0}$,
    $\br'_{i_0}=\bv'_{i_0}$,
    and
    $\br_i=\br_i'$ for $i\neq i_0$.
    We may now apply  Osserman-Payne lifting to the hypersurfaces in $(\KK^*)^n$ cut out by \eqref{eqn:zJ} and $f_1,\ldots,f_{i_0-1},f_{i_0}',f_{i_0+1},\ldots,f_{n-1}=0$.
    Indeed, near 
    $\alpha$ their tropicalizations are hyperplanes orthogonal to $\bv_1-\bv_1', \bv_1,\ldots,\bv_n$, and thus intersect properly at $\alpha$.
    
    To see that we can also achieve $\val(z_J)=\nu_J$,
    let 
    \[
    	f_{i_0}'':=f_{i_0}-\frac{1}{\Delta_J(\br')\prod_{i\neq i_0} c_{i\br_i'}x^{\br_i'}}\left(
    	m\cdot \Delta_J(\br)a_\br+\Delta_J(\br')a_{\br'}\right).
    \]
    for any $m\in \ZZ$ with $m\neq 1$.
    Consider the hypersurfaces in $(\KK^*)^n$ cut out by $f_1,\ldots,f_{i_0-1},f_{i_0}'',f_{i_0+1},\ldots,f_{n-1}=0$ along with $m\cdot \Delta_J(\br)a_\br+\Delta_J(\br')a_{\br'}=0$. Osserman-Payne lifting also applies here, and we obtain $P\in \trop(X)$ such that the monomials 
    $\Delta_J(\br)a_\br,\Delta_J(\br')a_{\br'}$ will no longer cancel in \eqref{eqn:zJ}, giving $\val(z_J)=\nu_J$.

    In the second case, when $\lambda>\nu_J$ we may fix $z_J$ with $\val(z_J)>\nu_J$ and directly apply Osserman-Payne lifting to the hypersurfaces in $(\KK^*)^n$ cut out by \eqref{eqn:zJ} and $f_1,\ldots,f_{n-1}=0$. Indeed, near 
    $\alpha$ their tropicalizations are hyperplanes orthogonal to $\sum_i(\br-\br')_i, \bv_1,\ldots,\bv_n$, and thus intersect properly at $\alpha$.
    The argument for the claim when $\val(z_J)=\nu_J$ is similar to the first case.
    \end{proof}
    
    \begin{rem}\label{rem:cancel}
    	We may drop the assumption of non-colliding valuations from Proposition \ref{prop:cancel} if we instead assume that $\alpha$ is a vertex and $\Delta_J(\bv)$, $\Delta_J(\bv')$, and $\Delta_J(\bv)-\Delta_J(\bv')$ are all non-zero.
    \end{rem}
    
    If we assume that $X_1,\ldots,X_{n-1}$ have non-colliding valuations and sufficiently general lowest order parts, for every $\alpha\in \trop(X)$ we may apply either Proposition \ref{prop:nocancel} or \ref{prop:cancel} to determine all possible values of $\val(z_J)$ for $P\in X$ tropicalizing to $\alpha$. However, we need to understand which values of $\val(z_J)$ are \emph{simultaneously} possible as $J$ ranges over all subsets of $\{0,\ldots, n\}$ of size $2$. For those $\alpha$ and $J$ covered by Proposition \ref{prop:nocancel}, there is nothing to do, since $\val(z_J)$ is uniquely determined. We deal with the cases of Proposition \ref{prop:cancel} in the following subsection.
    	\begin{ex}[Example \ref{ex:runningP} continued (a curve in $\PP^2$)]\label{ex:simultP}
   It remains to compute the tropical tangents at the vertex $V$. We have already seen that for $J=\{1,2\}$ or $J=\{0,1\}$, $\val(z_J)=\nu_J=0$. For $J=\{0,2\}$ we may apply Proposition \ref{prop:cancel} to obtain that $\val(z_J)$ takes any value greater than or equal to zero.

   Hence, the tropical tangents at $V$ have Pl\"ucker coordinates whose valuations have the form \[\val(q_{12})=0\qquad \val(q_{02})=s \qquad \val(q_{12})=0\]
   for any $s\geq 0$.
	\end{ex}

    \subsection{Simultaneous cancellation}\label{sec:simult}
    We introduce some notation for the remainder of this section. We will always assume that $\alpha$ is in the critical locus of $\trop (X)$, and that the $X_i$ either have non-colliding valuations or that the assumptions of Remark \ref{rem:cancel} hold. The point $P\in \trop(X)$  will always tropicalize to $\alpha$.

     We call indices $J$ for which one of the conditions of Proposition \ref{prop:cancel} is satisfied \emph{cancellative}. For any cancellative index $J$ , we let $\br_J,\br_J'$ be the associated cancellative pair. By the assumption on non-colliding valuations, after possibly interchanging $\br_J$ and $\br'_J$, we may assume that $\br_J-\br_J' =  \br_I-\br_I'$ for any cancellative indices $I$ and $J$. 
    \begin{prop}\label{prop:delta}
    	Let $I,J$ be cancellative indices, and assume that $\val(z_I)>\nu_I$ at $P\in\trop(X)$. Then $\val(z_J)>\nu_J$ at $P$ if and only if
    	\begin{equation}\label{eqn:delta}
    		\frac{\Delta_J(\br_J)}{\Delta_J(\br_J')}=
    		\frac{\Delta_I(\br_I)}{\Delta_I(\br_I')}.
    	\end{equation}
    \end{prop}
    \begin{proof}
    	First, we may modify \eqref{eqn:zJ} by adding a multiple of the similar expression for $z_I$ to obtain 
    \begin{equation*}
    	z_J=\sum_{\bw\in\A} \Delta_J(\bw)a_\bw
    	=\sum_{\bw\in\A} \Delta_J(\bw)a_\bw+\frac{\Delta_J(\br_J)a_{\br_J}}{\Delta_I(\br_I)a_{\br_I}}
    	\left(z_I-\sum_{\bw\in\A} \Delta_I(\bw)a_\bw\right).
    \end{equation*}
    The minimal valuation terms in the first expression for $z_J$ are exactly
    $\Delta_J(\br_J)a_{\br_J}+\Delta_J(\br_J')a_{\br_J'}$, each of which has valuation $\nu_J$. We will show that both of them are cancelled by terms from the expression that we added on the right precisely when condition \eqref{eqn:delta} holds. 
    
    Since $\val(z_I)>\nu_I$, the minimal valuation terms of the expression that we have added on the right are
    \[
    -\frac{\Delta_J(\br_J)a_{\br_J}}{\Delta_I(\br_I)a_{\br_I}}
    	\left(
    	\Delta_I(\br_I)a_{\br_I}+\Delta_I(\br_I')a_{\br_I'}
    	\right)
    =-\Delta_J(\br_J)a_{\br_J}-\frac{\Delta_J(\br_J)a_{\br_J}}{\Delta_I(\br_I)a_{\br_I}}
    	\Delta_I(\br_I')a_{\br_I'}.
    	\]
     Thus, the term $\Delta_J(\br_J)a_{\br_J}$ is cancelled. Now, let $K$ be the set of indices $i$ for which $(\br_I)_i\neq (\br_I')_i$.	The assumption on non-colliding valuations implies that  
    \[
    	\frac{c_{\br_J}c_{\br'_I}}{c_{\br_I}}
    	=\prod_i \frac{c_{i(\br_J)_i}c_{i(\br_I')_i}}{c_{i(\br_I)_i}}
    	=\prod_{i\in K} c_{i(\br_I')_i}
    	\prod_{i\notin K} c_{i(\br_J)_i}
    	=\prod_{i\in K} c_{i(\br_J')_i}
    	\prod_{i\notin K} c_{i(\br_J')_i}
    	=c_{\br_J'},
    \]
    from which it follows that $a_{\br_J}a_{\br_I'}/a_{\br_I}=a_{\br_J'}$. Therefore, the term $\Delta_J(\br_J')a_{\br_J'}$ is cancelled as well if and only if 
    \begin{equation}
    		\frac{\Delta_J(\br_J)}{\Delta_J(\br_J')}=
    		\frac{\Delta_I(\br_I)}{\Delta_I(\br_I')}.
    	\end{equation}
    The claim follows.
    \end{proof}

    We have thus determined for which collections of indices $J$ one may simultaneously have $\val(z_J)>\nu_J$. It remains to determine the exact values. For this, we will consider the ring 
    \[R(\A):=\QQ[y_\bw^{\pm 1}\ |\ \bw\in\A] \]
    of Laurent polynomials in variables $y_\bw$, where $\bw\in\A$. 
    We then have a natural ring homomorphism
    \begin{align*}
    	\phi:R(\A)	&\to \KK[M]\\
    	y_\bw&\mapsto c_\bw\cdot x^\bw \qquad \bw\in\A.
    \end{align*}
    We note that in light of \eqref{eqn:zJ}, we have 
    \[
    	z_J=\phi\left( \sum_{\bw\in\A} \Delta_J(\bw)y_\bw)\right).
    \]
    
    For any Laurent monomial $\xi\in R(\A)$ and point $\alpha\in \TP$, we may consider the valuation 
    $\val(\phi(\xi))$ of $\phi(\xi)$ at $\alpha$,
    since $\phi$ maps monomials to monomials. By a slight abuse of notation, we will refer to this simply as the valuation of $\xi$ at $\alpha$.

    \begin{lemma}\label{lemma:mon}
    	Assume that the $X_i$ have very general valuations, and let $\br,\br'$ be a cancellative pair for a point $\alpha\in X$.
    Let $\xi\in R(\A)$ be a Laurent monomial. Then for any other Laurent monomial $\xi'\in R(\A)$ with $\deg \xi'=\deg \xi$ and $\val(\phi(\xi'))=\val(\phi(\xi'))$ at $\alpha$, there exists $\gamma\in\QQ$ and $j\in\ZZ$ such that
    	\[
    \phi(\xi')=\gamma\cdot a_\br^ja_{\br'}^{-j}\cdot \phi(\xi).
    	\]
    \end{lemma}
    \begin{proof}
    	Applying $\phi$ we may write
    \begin{align*}
    	&\phi(\xi)=\lambda\cdot \prod_{\substack{i=1,\ldots,n-1\\w\in\A_i\setminus\{0\}}} (c_{iw}x^w)^{m_{iw}}\qquad \lambda\in \QQ,\ m_{iw}\in\ZZ;\\
    	&\phi(\xi')=\lambda'\cdot \prod_{\substack{i=1,\ldots,n-1\\w\in\A_i\setminus\{0\}}} (c_{iw}x^w)^{m_{iw}'}\qquad \lambda'\in \QQ,\ m_{iw}'\in\ZZ,
    \end{align*}
    noting that for each $i$,
    \begin{equation}\label{eqn:degrel}
    	\sum_{w\in \A_i\setminus\{0\}} m_{iw}=\sum_{w\in \A_i\setminus\{0\}} m_{iw}'.
    \end{equation}
    
    Since $\val(\phi(\xi'))=\val(\phi(\xi'))$ 
    we obtain 
    \[
    	\sum_{i,w} (m_{iw}-m'_{iw})\cdot \val(c_{iw}x^{w})=
    	\sum_{i,w\neq \bv_i} (m_{iw}-m'_{iw})\cdot \val(c_{iw}x^{w})=0
    \]
    On the other hand, since $\br,\br'$ is a cancellative pair we have the relation 
    \begin{equation}\label{eqn:primrel}
    		\sum_{i} \val(c_{i\br_i}x^{\br_i})-\val(c_{i\br_i'}x^{\br_i'})=0.
    	\end{equation}
    Omitting the terms of the form $\val(c_{i\bv_i}x^{\bv_i})$ from  \eqref{eqn:primrel}, we obtain a non-trivial relation among the elements of $S_\alpha$ (cf. Definition \ref{defn:genval}). 
    
    Since the $X_i$ have very general valuations, any integral relation among the elements of $S_\alpha$ must be an integral multiple of the relation coming from \eqref{eqn:primrel}.
    In particular, there exists $j\in \ZZ$ such that for any $i=1,\ldots,n-1$ and any $w\in \A_i\setminus \{0,\bv_i\}$,
    \[
    	m_{iw}-m_{iw}'=\begin{cases}
    		-j &w=\br_i\ \textrm{and}\ w\neq \br_i'\\
    		j &w\neq \br_i\ \textrm{and}\ w= \br_i'\\
    		0 &\textrm{else}
    	\end{cases}
    \]
    Using \eqref{eqn:degrel} we may extend the above to also include $w=\bv_i$ and conclude 
    that 
    \[
    \phi(\xi')=\gamma\cdot a_\br^ja_{\br'}^{-j}\cdot \phi(\xi)
    \]
    as desired.
    
    \end{proof}
    \begin{lemma}\label{lemma:cancel}
    	Assume that the $X_i$ have very general valuations. Let $\br,\br'$ be a cancellative pair for a point $\alpha$, and $I\subset \{0,\ldots,n\}$ a subset of size $2$ such that $\Delta_I(\br)\neq 0$.
    Consider 
    any homogeneous Laurent polynomial $f\in R(\A)$ of degree $d$, and denote by $\nu$ the minimal valuation of the monomials of $f$ at $\alpha$. 
    Then we may algorithmically construct
    \[
    g,h\in R(\A)
    \]
    such that 
    \begin{equation}
    	\phi(f)=\phi\big(h+(\Delta_I(\br)y_\br+\Delta_I(\br')y_{\br'})\cdot g\big)
    \end{equation}
    and
    \begin{enumerate}
    	\item\label{item:g} the monomials of $g$ all have degree $d-1$ and valuation $\nu-\nu_{\br}$ at $\alpha$; and
    	\item the monomials of $h$ all have degree $d$ and valuation at least $\nu$ at $\alpha$, with at most one term having valuation $\nu$. Furthermore, the valuations of the monomials of $h$ are among the valuations of the monomials of $f$.
    \end{enumerate}
    Furthermore, this construction does not depend on the coefficients $c_\bw$ but only their valuations $\val(c_\bw)$.
    \end{lemma}
    
    \begin{proof}
    	Fix an arbitrary monomial $\xi\in R(\A)$ of valuation $\nu$ at $\alpha$. For any monomial $\xi'$ of $f$ with valuation $\nu$, we may apply Lemma \ref{lemma:mon}.
    We refer to the exponent $j$ appearing in the expression for $\phi(\xi')$ in the lemma as the $\br$-exponent of $\xi'$.

    We first claim that there exist $g,h^* \in R(\A)$ fulfilling
    \[
    	f=h^*+(\Delta_I(\br)y_\br+\Delta_I(\br')y_{\br'})\cdot g
    \]
    such that $g$ satisfies item \ref{item:g} in the statement of the lemma, the monomials of $h^*$ all have degree $d$ and valuation at least $\nu$ at $\alpha$, the monomials of valuation $\nu$ all have the same $\br$-exponent, and the valuations of the monomials of $h^*$ are among the valuations of the monomials of $f$.
    
    To prove this, we induct on the difference $\delta$ between the maximal and minimal $\br$-exponents of the valuation-$\nu$ monomials of $f$. When $\delta=0$, the claim is trivial.
    
    For the induction step, 
    set 
    \[g''=\frac{\sum_\eta \eta}{\Delta_I(\br)y_\br} 
    \]
    where the sum is taken over all valuation-$\nu$ monomials $\eta$ of $f$ with maximal $\br$-exponent.
    Setting 
    \[
    f'=f-(\Delta_I(\br)y_\br+\Delta_I(\br')y_{\br'})g'',
    \]
    the difference between the maximal and minimal $\br$-exponents of the valuation-$\nu$ monomials of $f'$ is strictly smaller than $\delta$. By the induction hypothesis, there thus exist $g',h^*\in R(\A)$ fufilling 
    \[
    	f'=h^*+(\Delta_I(\br)y_\br+\Delta_I(\br')y_{\br'})\cdot g'
    \]
    along with the other conditions of the claim.
    We may then take $g=g'+g''$  to obtain the desired expression for $f$.
    
    Having proven the above claim, we now finish the proof of the lemma. If $h^*$ has no valuation-$\nu$ monomials, we may simply take $h=h^*$. Otherwise, let $\eta_0,\ldots,\eta_k$ be the valuation-$\nu$ monomials of $h^*$. Since they all have the same $\br$-exponent, there exist rational numbers $\lambda_1,\ldots,\lambda_k$ such that 
    \[
    \phi(\eta_i)=\lambda_i\cdot \phi(\eta_0)\qquad i= 1,\ldots,k.
    \]
    We then set 
    \[
    	h=h^*+\sum_{i=1}^k (\lambda_i \eta_0-\eta_i).
    \]
    It follows from construction that $g$ and $h$ have the desired properties and the claim of the lemma follows.
    \end{proof}

    We return to the problem of determining the exact value of $\val(z_J)$ given $\val(z_I)$ for cancellative indices $I$ and $J$. 
    \begin{alg}\label{alg:zalg}
    	Assume $X_1,\ldots,X_{n-1}$ have very general valuations.
    We will inductively construct sequences $(g_i)$, $(\zeta_i)$ of homogeneous elements of $R(\A)$ such that 
    \[
    z_J=\phi(\zeta_i)+\phi(g_i)z_I.
    \]
    We begin with
    \[
    	\zeta_0=\sum_{w\in\A} \Delta_J(\bw)y_\bw \qquad g_0=0.
    \]
    
    The algorithm \emph{terminates} if either 
    $\zeta_i$ has a unique monomial of minimal valuation, or the minimal valuation of a monomial of $\zeta_i$ is at least $\nu_J-\nu_I+\val(z_I)$.
    Otherwise, we apply Lemma \ref{lemma:cancel} to $f=\zeta_i$ with the cancellative pair $\br_I,\br_I'$, obtaining $g$ and $h$.
    We use these to define the next terms in the sequences:
    \begin{align*}
    	g_{i+1}&=g_i+g\\
    	\zeta_{i+1}&=h-g\cdot \sum_{\bw\in\A\setminus\{\br_I,\br_I'\}} \Delta_I(\bw)y_\bw.
    \end{align*}
    \end{alg}
    \begin{prop}\label{prop:algorithm}
    Assume $X_1,\ldots,X_{n-1}$ have very general valuations.
    Let $I$ and $J$ be cancellative indices such that \eqref{eqn:delta} is satisfied. Consider some $P\in X$ tropicalizing to $\alpha$ such that 
    $\val(z_I)>\nu_I$ at $P$.
    Then
    \begin{enumerate}
    	\item Algorithm \ref{alg:zalg} terminates;
    	\item For every $i$, we have 
    		\begin{equation}\label{eqn:zeta}
    z_J=\phi(\zeta_i)+\phi(g_i)z_I.
    		\end{equation}
    \end{enumerate}
    In particular, either the valuation $\val(z_J)$ of $z_J$ at $P$ is at least $\nu_J-\nu_I+\val(z_I)$, or there is some $i$ such that $\zeta_i$ has a unique monomial $\xi$ of minimal valuation at $\alpha$, and 
    \[
    \val(z_J)=\val(\xi).
    \]
    \end{prop}
    \begin{proof}
    	We first show by induction that \eqref{eqn:zeta} holds. For $i=0$ the claim is clear. For the induction step we have
    	\begin{align*}
    		\phi(\zeta_{i+1})+\phi(g_{i+1})z_I=
    		\phi(h)+\phi(g)\left(z_I- \sum_{\bw\in\A\setminus\{\br_I,\br_I'\}} \Delta_I(\bw)a_\bw\right)+\phi(g_i)z_I
    	\end{align*}
    	with $g,h$ as in Algorithm \ref{alg:zalg}.
    	By \eqref{eqn:zJ} we then have
    	\begin{align*}
    	\phi(\zeta_{i+1})+\phi(g_{i+1})z_I=
    		\phi(h)+\phi(g)\left(\sum_{\bw\in\{\br_I,\br_I'\}} \Delta_I(\bw)a_\bw\right)+\phi(g_i)z_I\\
    		=\phi(\zeta_i)+\phi(g_i)z_I=z_J.
    	\end{align*}
    
    	We next show that the algorithm terminates.
     Set
    	\[
    		\epsilon=\min\{\nu_\bw-\nu_{\bw'}\ | \bw,\bw'\in \A\ \textrm{and}\ \nu_\bw>\nu_{\bw'}\}.
    	\]
    	Let $\nu_i$ be the minimal valuation of a monomial of $\zeta_i$.
    	The valuation of any monomial of $\zeta_{i+1}$ is either among the valuations of the monomials of $\zeta_i$ (by Lemma \ref{lemma:cancel}) or is of the form $\nu_i+\nu_{\bw}-\nu_{I}$ for $\bw\in\A\setminus\{\br_I,\br_I'\}$ satisfying $\nu_{\bw}-\nu_{I}>0$.
    	Then since $\zeta_i$ only has finitely many monomials, after a finite number of steps $k$ we must have 
    	\[\nu_{i+k}\geq \min\{\nu_i+\epsilon,\nu_J-\nu_I+\val(z_I)\}.\]
    	It follows that the algorithm must terminate.

    By the construction of the $\zeta_i$ and $g_i$, one may readily show by induction on $i$ that $\val(\phi(g_i))\geq \nu_J-\nu_I$.
    	The remaining claims of them proposition then follow from the termination criterion, \eqref{eqn:zeta}, and the inequality 
    	\[
    \val(\phi(g_i) z_I)=\val(\phi(g_i))+ \val(z_I)\geq \nu_J-\nu_I+\val(z_I).
    	\]
    
    \end{proof}

    \begin{thm}\label{thm:alltangents}
	    Let $X\subset \PP^n$ be a curve satisfying the hypotheses of \S\ref{sec:setup} and Assumption \ref{ass:setup}. Assume further that $X_1,\ldots,X_{n-1}$ have very general valuations and sufficiently general lowest order parts. Then for any $\alpha\in \trop(X)$, we may determine all tropical tangents to $\alpha$ with Propositions \ref{prop:nocancel}, \ref{prop:cancel}, and \ref{prop:algorithm} using only the ``tropical'' data of $\{\A_i\}_{i=1}^{n-1}$ and $\{\val(c_{iw})\}_{w\in \A_i}$.
    \end{thm}
    \begin{proof}
    	To determine the tropical tangents, it suffices to determine the possible tuples $(\val(q_J))_{J}$
    	for $P\in X$ tropicalizing to $\alpha$.	By \eqref{eqn:pz}, it suffices to determine $(\val(z_J))_{J}$ instead.
    
    	Note that the quantities $\nu_J$ are determined by the tropical data mentioned in the statement of the theorem. If $J$ is an index set satisfying one of the conditions of Proposition \ref{prop:nocancel}, we have $\val(z_J)=\nu_J$. For any of the remaining indices $J$, by Proposition \ref{prop:cancel} it is possible to obtain $\val(z_J)=\nu_J$. In fact, slightly adapting the proof of  \ref{prop:cancel}. we see that it is possible to obtain $\val(z_J)=\nu_J$ simultaneously for all $J$. 
    
    	Suppose instead that $\val(z_I)>\nu_I$ for some index set $I$. Then Proposition \ref{prop:delta} tells us exactly for which index sets $J$ we have $\val(z_J)>\nu_J$. Fixing the value of $\val(z_I)$ (which by Proposition \ref{prop:cancel} may be any quantity larger than $\nu_I$), we may now apply Proposition \ref{prop:algorithm} to either determine $\val(z_J)$, or conclude that $\val(z_J)-\nu_J\geq \val(z_I)-\nu_I$. If we are in the latter case, we may apply Proposition \ref{prop:algorithm} with the roles of $I$ and $J$ reversed to either conclude that $\val(z_I)-\nu_I=\val(z_J)-\nu_J$, or determine the value of $\val(z_I)$ based on the value of $\val(z_J)$.
    	\end{proof}
   
	\begin{rem}\label{rem:GX}
		We may use Theorem \ref{thm:alltangents} to obtain a complete description of $\trop(\G(X))$. Indeed, the curve $X$ intersects the coordinate hyperplanes of $\PP^n$ in only finitely many points, so there are only finitely many points of $\trop(\G(X))$ that are not a tropical tangent for some $\alpha\in\trop(X)$. Hence, $\trop(\G(X))$ is the closure of all tropical tangents computed via Theorem \ref{thm:alltangents}.
	\end{rem}

    	\begin{ex}[Example \ref{ex:running} continued (a curve in $\PP^3$)]\label{ex:simult}
    It remains to compute the tropical tangents at $V_1$, $V_2$, $V_3$. 
    We first apply Proposition \ref{prop:delta} to compute that simultaneous cancellation is possible for the following collections of indices:
    	\vspace{.3cm}
    	\begin{center}	\begin{tabular}{l l }
    Vertex & Collections of Index Sets\\
    			\toprule
    $V_1$& $12,13,23$\\
    $V_2$& $02,03,23$\\
    $V_3$& $03,13$ and $02,12$\\
    			\bottomrule
    	\end{tabular}
    \end{center}

    We now apply Algorithm \ref{alg:zalg}
    We begin with $V_1$.
Here we have exponent vectors
\begin{align*}
	\bv_1=(2,0,1,-3),\bv_1'=(1,0,3,-4)\in\A_1\setminus\{0\}\\
	\bv_2=(0,1,1,-2),\bw=(1,-1,0,0)\in\A_2\setminus\{0\}.
\end{align*}
Taking  $I=\{2,3\}$ and  $J=\{1,3\}$, we have 
    \begin{align*}
	    \zeta_0=\underline{2y_{(\bv_1,\bv_2)}} +
	    \underline{1y_{(\bv_1',\bv_2)}} + 
	    (-2)y_{(\bv_1,\bw)}  
	    (-3)y_{(\bv_1',\bw)}
    \end{align*}
    and $\br_I=(\bv_1,\bv_2),\br'_I=(\bv_1',\bv_2)$.
    Terms of minimal valuation are underlined.
    Applying Lemma \ref{lemma:cancel} to $\zeta_0$, we obtain 
    \begin{align*}
	    g=1\qquad h=\underline{(-2)y_{(\bv_1,\bw)}}+\underline{(-3)y_{(\bv_1',\bw)}}\\
	    g_1=1\qquad \zeta_1=\underline{1y_{(\bv_1,\bw)}}+\underline{(-2)y_{(\bv_1',\bw)}}.
    \end{align*}
    Applying Lemma \ref{lemma:cancel} to $\zeta_1$, we obtain 
    \begin{align*}
	    g=\frac{y_{(\bv_1,\bw)}}{2y_{(\bv_1,\bv_2)}}\qquad h=\underline{(-5/2){y_{(\bv_1',\bw)}}}\\
	    g_2=\underline{1}+\frac{y_{(\bv_1,\bw)}}{2y_{(\bv_1,\bv_2)}}
\qquad
\zeta_2=\underline{(-5/2){y_{(\bv_1',\bw)}}}+\HOT.
    \end{align*}
Since 
\[
	z_{13}=\phi(\zeta_2)+\phi(g_2)z_{23}=\phi({(-5/2){y_{(\bv_1',\bw)}}}+\HOT)+\phi\left(1+\frac{y_{(\bv_1,\bw)}}{2y_{(\bv_1,\bv_2)}}\right)z_{23}
\]
and the valuation of ${(-5/2){y_{(\bv_1',\bw)}}}$ is three, we conclude
that
    the minimum of $\val(z_{23}),\val(z_{13}), 3$ is obtained at least twice.
    Similar computations show that the minimum of $\val(z_{23}),\val(z_{12}), 3$ and $\val(z_{13}),\val(z_{12}), 3$ are both obtained twice.

    We next consider $V_2$ where we have
exponent vectors
\begin{align*}
	\bv_1=(-1,0,2,-1),\bw=(-2,0,-1,3)\in\A_1\setminus\{0\}\\
	\bv_2=(0,-1,-1,2),\bv_2'=(1,-2,-1,2)\in\A_2\setminus\{0\}.
\end{align*}
Take $I=\{2,3\}$, $J=\{0,3\}$. Then 
 \begin{align*}
	    \zeta_0=\underline{2y_{(\bv_1,\bv_2)}} +
	    \underline{4y_{(\bv_1,\bv_2')}} + 
	    (-1)y_{(\bw,\bv_2)}+  
	    (-2)y_{(\bw,\bv_2')}
    \end{align*}
    and $\br_I=(\bv_1,\bv_2),\br'_I=(\bv_1,\bv_2')$.
 Applying Lemma \ref{lemma:cancel} to $\zeta_0$, we obtain 
    \begin{align*}
	    g=2\qquad h=\underline{(-1)y_{(\bw,\bv_2)}}+\underline{(-2)y_{(\bw,\bv_2')}}\\
	    g_1=2\qquad \zeta_1=\underline{(-5)y_{(\bw,\bv_2)}}+\underline{(-10)y_{(\bw,\bv_2')}}.
    \end{align*}
Applying Lemma \ref{lemma:cancel} to $\zeta_1$, we obtain 
    \begin{align*}
	    g=\frac{-5y_{(\bw,\bv_2)}}{y_{(\bv_1,\bv_2)}}\qquad h=0\\
	    g_2=2+\frac{-5y_{(\bw,\bv_2)}}{y_{(\bv_1,\bv_2)}}
	    \qquad \zeta_2=\frac{-5y_{(\bw,\bv_2)}}{y_{(\bv_1,\bv_2)}}\left(\underline{2y_{(\bw,\bv_2)}}+\underline{4y_{(\bw,\bv_2')}}\right)
    \end{align*}
    Unlike in the case of $V_1$, there are still two terms of minimal valuation in $\zeta_2$.  In fact, we may continue this process indefinitely, since after applying $\phi$, the two terms appearing in $\zeta_i$ will always be a multiple of $\phi(\Delta_I(\br_I)y_{\br_I}+\Delta_I(\br_I')y_{\br_I'})$. We thus obtain $\val(z_{03})=\val(z_{23})$, and a similar computation shows $\val(z_{02})=\val(z_{23})$.

    We finally consider $V_3$. 
    Computations similar to those for $V_1$ show that the minimum of $\val(z_{13}),\val(z_{03}), 3$ and of $\val(z_{12}),\val(z_{02}), 3$ are both obtained twice.
    
    We summarize the  tropical tangents of $V_1,V_2,V_3$ in Table \ref{table:vertices}. Combined with Table \ref{table:edges}, this gives a description of $\trop(\G(X))$. This tropical curve has $6$ vertices, $5$ bounded edges, and $20$ unbounded edges. A random projection to the plane is pictured in Figure \ref{fig:gauss}.
    \end{ex}
    
    	\begin{table}	\begin{tabular}{l l l l l l l}
    			 &  $\val(q_{23})$& $\val(q_{13})$ &$\val(q_{12})$& $\val(q_{03})$ &$\val(q_{02})$&$\val(q_{01})$\\
    			\toprule
    			$V_1$& $s_{23}$ & $s_{13}$& $s_{12}$ & $s_{03}$ & $s_{02}$ & $s_{01}$\\ \addlinespace
    			\multicolumn{7}{l}{$s_{ij} \geq 0$, all but one of $s_{23},s_{03},s_{02},s_{01}$ vanish,}\\
    			\multicolumn{7}{l}{and $s_{23}=s_{13}=s_{12}\leq 3$ or $s_{23},s_{13},s_{12}\geq 3$ with at least two equal.}\\
    \midrule
    \\
    			 &  $\val(q_{23})$& $\val(q_{13})$ &$\val(q_{12})$& $\val(q_{03})$ &$\val(q_{02})$&$\val(q_{01})$\\
    			\toprule
    			$V_2$& $s_{23}-3$ & $s_{13}-1$& $s_{12}-2$ & $s_{23}-4$ & $s_{23}-5$ & $-3$\\ \addlinespace
    			\multicolumn{7}{l}{$s_{ij} \geq 0$, all but one of $s_{23},s_{13},s_{12}$ vanish.}\\
    \midrule
    \\
    
    			 &  $\val(q_{23})$& $\val(q_{13})$ &$\val(q_{12})$& $\val(q_{03})$ &$\val(q_{02})$&$\val(q_{01})$\\
    			\toprule
    			$V_3$& $s_{23}-3$ & $s_{13}$& $s_{12}$ & $s_{03}-3$ & $s_{02}-3$ & $s_{01}+3$\\ \addlinespace
    			\multicolumn{7}{l}{$s_{ij} \geq 0$, all but one of $s_{23},s_{13},s_{12},s_{01}$ vanish,}\\
    			\multicolumn{7}{l}{and $s_{13}=s_{03}\leq 3$ or $s_{13},s_{03}\geq 3$ with at least one equal,}\\
    			\multicolumn{7}{l}{and $s_{12}=s_{02}\leq 3$ or $s_{12},s_{02}\geq 3$ with at least one equal.}\\
    
    			\bottomrule
    	\end{tabular}
    	\vspace{.5cm}
    
    	\caption{Tropical tangents for vertices in Example \ref{ex:simult}}\label{table:vertices}
    \end{table}
    
    \begin{figure}
    	\begin{tikzpicture}[scale=.35]
    \draw[fill] (0,0) circle [radius=0.2];
    \draw[fill] (-2.13185,1.73585) circle [radius=0.2];
    \draw[fill] (6.46031,2.88212) circle [radius=0.2];
    \draw[fill] (-7.57524,-1.42961) circle [radius=0.2];
    \draw[fill] (10.3577,6.4227) circle [radius=0.2];
    \draw[fill] (8.75494,8.06678) circle [radius=0.2];
    \draw (0,0) -- (-8.36171,-5.85439);
    \draw (0,0) -- (10.4936,4.11854);
    \draw (0,0) -- (-2.13185,1.73585);
    \draw (-2.13185,1.73585) -- (-9.2694,1.0895);
    \draw (-2.13185,1.73585) -- (6.46031,2.88212);
    \draw (6.46031,2.88212) -- (.252045,-2.01156);
    \draw (6.46031,2.88212) -- (21.2607,8.92207);
    \draw (0,0) -- (-7.57524,-1.42961);
    \draw (6.46031,2.88212) -- (10.3577,6.4227);
    \draw (6.46031,2.88212) -- (8.75494,8.06678);
    \draw (0,0) -- (1.36328,1.46338);
    \draw (0,0) -- (1.94388,.470475);
    \draw (0,0) -- (1.74301,-.98078);
    \draw (-2.13185,1.73585) -- (-1.96537,3.72891);
    \draw (-2.13185,1.73585) -- (-1.47746,3.62577);
    \draw (6.46031,2.88212) -- (.589291,-1.95393);
    \draw (6.46031,2.88212) -- (8.20332,1.90134);
    \draw (-7.57524,-1.42961) -- (-13.4463,-6.26566);
    \draw (-7.57524,-1.42961) -- (-7.40877,.56345);
    \draw (-7.57524,-1.42961) -- (-6.92085,.460304);
    \draw (-2.13185,1.73585) -- (-3.41378,.284756);
    \draw (10.3577,6.4227) -- (11.0121,8.31262);
    \draw (10.3577,6.4227) -- (12.3016,6.89318);
    \draw (8.75494,8.06678) -- (8.92141,10.0598);
    \draw (8.75494,8.06678) -- (10.1182,9.53015);
    
    	\end{tikzpicture}
    	\caption{A random projection of $\trop(\G(X))$ from Example \ref{ex:simult}}.\label{fig:gauss}
    \end{figure}
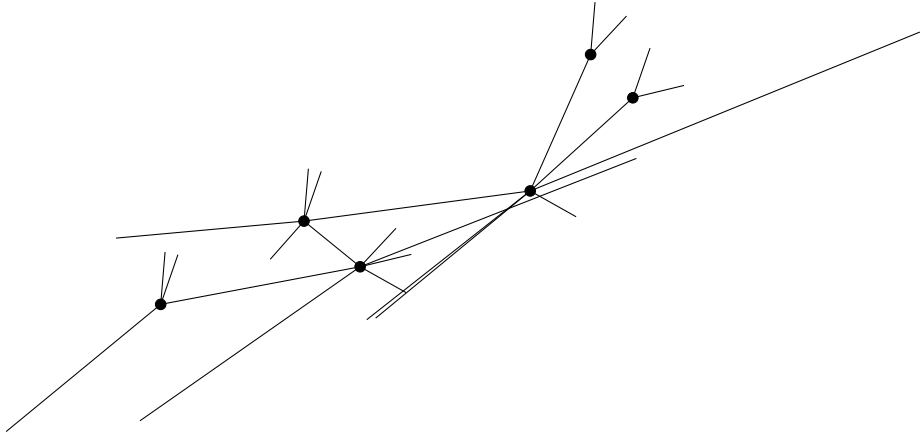

    \section{Combinatorial Conditions}\label{sec:comb}
    \subsection{Vanishing determinants}
    From Equation \ref{eqn:min} and Propositions \ref{prop:nocancel} and \ref{prop:cancel}, we have seen that it is important to determine when $\Delta_J(\bw)=0$ for some $\bw\in\A$. Of primary interest are the cases $\bw=\bv$, and $\bw=\bv'$ (in the case of a vertex).
    
    \begin{lemma}\label{lemma:edge}
    	Suppose $\alpha \in\trop(X)$ is in the relative interior of an edge $E$ and consider $J\subset \{0,\ldots, n\}$ with $|J|=2$.  Then $\Delta_J(\bv)=0$ if and only if $\langle E\rangle \subseteq \langle \overline J\rangle$.
    \end{lemma}
    \begin{proof}
    $\Delta_J(\bv)=0$ if and only if  $\bv$ becomes singular after removing the columns in $J$, that is, the kernel of $\bv$ contains a non-trivial element whose $J$-coordinates are zero. On the other hand $\langle E \rangle$ is the image of the kernel of $\bv$ in $\TP$. The claim follows.
    \end{proof}
    
    Suppose instead that $\alpha\in\trop(X)$ is a vertex.
    Let $\Lambda$ be the affine tangent space to $\trop(X)$ at $\alpha$, that is, the affine span of a neighborhood of $\alpha$ in $\trop(X)$.
    
    \begin{lemma}\label{lemma:vert}
    	Let $\alpha \in\trop(X)$ be a vertex and consider $J\subset \{0,\ldots, n\}$ with $|J|=2$. \begin{enumerate}
    		\item $\Delta_J(\bv)$, $\Delta_J(\bv')$, or $\Delta_J(\bv)-\Delta_J(\bv')$ vanishes  if and only if $\langle E\rangle \subseteq \langle \overline J\rangle$ for some edge $E$ adjacent to $\alpha$.
    		\item $\Delta_J(\bv)=\Delta_J(\bv')=0$ if and only if $\langle \Lambda\rangle \subseteq \langle \overline J \rangle$. 
    	\end{enumerate}
    \end{lemma}
    \begin{proof}
    	The claims follow by applying Lemma \ref{lemma:edge} to the edges adjacent to $\alpha$, and noting that $\Delta_J(\bv)-\Delta_J(\bv')=\Delta_J(\bv-\bv')$.
    \end{proof}
    
    The second condition of Lemma \ref{lemma:vert} holds if and only if $\nu_J>0$. On the other hand, assuming that the second condition does not hold, by Propositions \ref{prop:nocancel} and \ref{prop:cancel} the first condition holds if and only if $\val(z_J)$ is necessarily zero. 
    
    \begin{ex}[Example \ref{ex:runningP} continued (a curve in $\PP^2$)]\label{ex:combP1}
    	We may use Lemma \ref{lemma:edge} in conjunction with Proposition \ref{prop:nocancel} to see that for $\alpha$ in the relative interior of $E_+$, $E_-$, or $E'$, $\val(z_J)>0$ 
	if and only if $\alpha\in E_+$ and $J=\{1,2\}$, or $\alpha\in E'$ and $J=\{0,1\}$. Indeed, $\langle E_+\rangle =\langle e_0\rangle$ and $\langle E'\rangle =\langle e_2\rangle$.

	We now consider the vertex $\alpha=V$ and apply Lemma \ref{lemma:vert} with $J=\{1,2\}$ or $J=\{0,1\}$.
	It follows from Proposition \ref{prop:nocancel} that $\val(z_{01})=\val(z_{12})=0$.
    \end{ex}
    
    \begin{ex}[Example \ref{ex:running} continued (a curve in $\PP^3$)]\label{ex:comb1}
    	We may use Lemma \ref{lemma:edge} in conjunction with Proposition \ref{prop:nocancel} to see that $\val(z_J)=0$ for all $\alpha$ in the relative interior of $E_1$, $E_3$, $E_4$, and $E_5$. On the other hand, since $\langle E_2\rangle$, $\langle E_6 \rangle$, and $\langle E_7 \rangle$ are all contained in $\langle e_2,e_3\rangle$, we see that for points $\alpha$ in these edges, $\val(z_{01})>0$.
    
    	We now consider the vertex $\alpha=V_2$ and apply Lemma \ref{lemma:vert} with $J=\{0,1\}$. Since the edge $E_2$ satisfies $\langle E_2 \rangle \subseteq \langle \overline J\rangle$ and the affine tangent space $\Lambda$ does not satisfy $\langle \Lambda \rangle \subseteq \langle \overline J\rangle$, we see that $\val(z_{01})=0$.
    \end{ex}
    
    \subsection{Simultaneous cancellation}
    The goal of this section is to give a more combinatorial interpretation of the conditions of Proposition~\ref{prop:delta}.  In what follows, we use $\mathds{1}$ to denote the all $1$s vector. 
    
    
    For a pair of $(n-1)\times(n+1)$ matrices  $A$ and $A'$, we set 
    \[\Theta_{ij,k\ell} (A,A') := \Delta_{ij}(A)\Delta_{k\ell}(A') - \Delta_{k\ell}(A)\Delta_{ij}(A').\]
    We omit the matrices when they are known from context, and simply write $\Theta_{ij,kl}$. 
    Note that $\Theta_{ij,k\ell} (A,A')$ is not symmetric in the indices, namely $\Theta_{ij,k\ell} (A,A') = - \Theta_{k\ell, ij} (A,A')$.
    
    \begin{lemma}\label{lemma:sc}
	    Consider full rank $(n-1)\times (n+1)$ matrices $A$ and $A'$ whose kernels contain $\mathds{1}$. We let $F$ and $F'$ be the images of $\row(A)^{\perp}$ and $\row(A')^{\perp}$ in $\TP=\RR^{n+1}/\mathds{1}$.
    \begin{enumerate}
    	\item If $\dim\langle F,F'\rangle=1$ then $\Theta_{ij,k\ell} = 0$ for any indices $i,j,k,\ell$.
    
    \item  If $\dim\langle F,F'\rangle=2$ and $I=\{i,j\}, J=\{k,j\}$ for distinct indices $i,j,k$, then  $\Theta_{ij,kj}(A,A') = 0$ if and only if 
    \[
    	\langle F,F'\rangle \cap \langle \overline{\{i,j,k\}}\rangle\neq \{0\}.
    \]
    
    \item If $\dim\langle F,F'\rangle=2$ and $I=\{i,j\}, J=\{k,\ell\}$ for distinct indices $i,j,k,\ell$, then $\Theta_{ij,k\ell}(A,A') = 0$ if and only if 
    \[
    \langle F,F'\rangle\cap \langle e_i+e_j,e_k+e_\ell, e_m|\ m\neq i,j,k,\ell\rangle\neq \{0\}.
    \]
    \end{enumerate}
    \end{lemma}
    
    \begin{proof} 
    If $\dim\langle F,F'\rangle=1$ then $\row(A)=\row(A')$. Since row operations do not affect the vanishing of $\Theta_{ij,k\ell}$, we may assume that $A=A'$. But then clearly $\Theta_{ij,k\ell}=0$. 
    
    We now assume that $\dim\langle F,F'\rangle=2$. Fix indices $i,j,k,\ell$ with $i,j,k$ necessarily distinct, and either $\ell=j$ or $\ell$ distinct. Let
    \[
	    V=\langle e_i-e_j,e_k-e_\ell\rangle^\perp \subseteq \RR^{n+1}.
    \]
If $\ell=j$, the image of $V$ in $\TT$ is $\langle \overline{\{i,j,k\}}\rangle$; if $\ell\neq j$, then the image of $V$ in $\TT$ is $\langle e_i+e_j,e_k+e_\ell, e_m|\ m\neq i,j,k,\ell\rangle$.

Let $B$ and $B'$ be the matrices obtained from $A$ and $A'$ by adding the row $e_i$ at the top. 
Then $(\row B)^\perp$ is one-dimensional, does not contain $\mathds{1}$, and its image in $\TT^n$ is $F$. A similar statement holds for $B'$ and $F'$. Since $V$ does contain $\mathds{1}$, it follows that the space $\langle F,F'\rangle$ and the image of $V$ in $\TT^n$ have non-trivial intersection if and only if 
$(\row B)^\perp+(\row B')^\perp$ has non-trivial intersection with $V$ in $\RR^{n+1}$.

By our assumptions, $(\row B)^\perp$ and $(\row B')^\perp$ are linearly independent, and $(\row B)^\perp$, $(\row B')^\perp$, and $V$ have complimentary dimensions. Hence,
$(\row B)^\perp+(\row B')^\perp$ has non-trivial intersection with $V$ if and only if
\[
	(\row B)^\perp + (\row B')^\perp +  V=\RR^{n+1}. 
\]

Let $\bigwedge B\in \bigwedge^{n} \RR^{n+1}$ be the exterior product of the vectors in $\RR^{n+1}$ that are the rows of $B$. Let  $*$ be the Hodge $*$-operator \[*: \bigwedge^m \RR^{n+1}\to \bigwedge^{n+1-m} \RR^{n+1},\] that is,
for $\omega,\omega'\in \bigwedge^m \RR^{n+1}$, 
\[
\omega \wedge (\omega')^*=\langle \omega, \omega' \rangle \cdot e_0\wedge \cdots \wedge e_n.
\]
We will show that
\begin{equation}\label{eqn:theta}
	(\bigwedge B)^* \wedge (\bigwedge B')^* \wedge ((e_i-e_j)\wedge (e_k-e_\ell))^*=\pm \Theta_{ij,kl} e_0\wedge \cdots \wedge e_n.
\end{equation}
The second and third claims of the lemma will follow, since the linear subspaces of $\RR^{n+1}$ corresponding to the forms 
$(\bigwedge B)^*$, $(\bigwedge B')^*$, and $((e_i-e_j)\wedge (e_k-e_\ell))^*$ are exactly
$\row(B)^\perp$, $\row(B')^\perp$ and $V$.

To show 
\eqref{eqn:theta}, we use the properties of the $*$-operator and Laplace expansion in the first row of $B,B'$ to note that
\begin{align*}
	&(\bigwedge B)^*=\sum_m (-1)^{n-m} \Delta_m(B) e_m=\sum_{m\neq i} (-1)^{n-m} \Delta_{im}(A)e_m\\
	&(\bigwedge B')^*=\sum_m (-1)^{n-m} \Delta_m(B') e_m=\sum_{m\neq i} (-1)^{n-m} \Delta_{im}(A')e_m.
\end{align*}
Here $\Delta_m$ denotes the determinant of the submatrix obtained by deleting the $m$th column.
Since the $*$-operator is linear, the coefficient of $e_0\wedge\ldots \wedge e_n$ in the left hand side of \eqref{eqn:theta} will be the sum of the coefficients in $(\bigwedge B)^* \wedge (\bigwedge B')^*$ of 
\[
e_i\wedge e_k,\quad -e_i\wedge e_\ell, \quad -e_j \wedge e_k, \quad e_j\wedge e_\ell.
\]
By the above, these are
\begin{align*}
&e_i\wedge e_k: &0\\
&-e_i\wedge e_\ell: &0\\
&-e_j\wedge e_k: &-(-1)^{j+k}(\Delta_{ij}(A)\cdot\Delta_{ik}(A')-\Delta_{ik}(A)\cdot\Delta_{ij}(A'))\\
&e_j\wedge e_\ell:&(-1)^{j+\ell}(\Delta_{ij}(A)\cdot\Delta_{i\ell}(A')-\Delta_{i\ell}(A)\cdot\Delta_{ij}(A')) 
\end{align*}
and so the sum is
\begin{equation}\label{eqn:deltalong}
\begin{aligned}
&(-1)^j\Delta_{ij}(A)\Big((-1)^\ell\Delta_{i\ell}(A')-(-1)^k\Delta_{ik}(A')\Big)\\
&-(-1)^j\Big((-1)^\ell\Delta_{i\ell}(A)-(-1)^k\Delta_{ik}(A)\Big)\Delta_{ij}(A').
\end{aligned}
\end{equation}

By permuting the columns of $A$ and $A'$, we may assume without loss of generality that $i<k<\ell$, since this will only change the sign of $\Theta_{ij,k\ell}$.
Since $\mathds{1}$ is in the kernel of $A$ and $A'$, the multilinearity of the determinant implies that 
\[
	(-1)^{(i+k+\ell)}\Delta_{k\ell}(A)=(-1)^\ell \Delta_{i\ell}(A)-(-1)^k \Delta_{ik}(A)
\]
with a similar expression for $A'$. Substituting these expressions into \eqref{eqn:deltalong} yields $\pm \Theta_{ij,kl}$, completing the proof.
    \end{proof}

    Let $X$ be as in \S\ref{sec:setup} and Assumption \ref{ass:setup}.
    We now apply Lemma \ref{lemma:sc} to study simultaneous cancellation for tropical tangents at a vertex $\alpha\in \trop(X)$. Recall that the \emph{affine tangent space} of $\trop(X)$ at $\alpha$ is the affine span of a neighborhood of $\trop(X)$ near $\alpha$.
    \begin{prop}\label{prop:simultv}
    Let $\alpha\in \trop(X)$ be a vertex and let $\Lambda$ be the affine tangent space at $\alpha$. Let $I$, $J$ be sets of indices with $|I|=|J|=2$, such that for any edge $E$ adjacent to $\alpha$, $\langle E \rangle \not\subseteq \langle \overline I \rangle$ and $\langle E \rangle \not\subseteq \langle \overline J \rangle$. Then it is possible  to simultaneously have $\val(z_I)>\nu_I$ and $\val(z_J)>\nu_J$ if and only if:
    \begin{enumerate}
    	\item $I\cap J\neq \emptyset$, and 
    \[
    	\langle \Lambda \rangle \cap \langle \overline{I\cup J}\rangle\neq \{0\};\qquad\textrm{or}
    \]
    \item  $I\cap J=\emptyset$, and 
    \[
    	\langle \Lambda\rangle\cap \left(\langle \sum_{i\in I} e_i,\sum_{j\in J} e_j\rangle+\langle \overline{I\cup J}\rangle\right)\neq \{0\}.
    \]
    \end{enumerate}
    \end{prop}
    \begin{proof}
    	By Lemma \ref{lemma:vert}, we obtain cancellative pairs $\br_I=\br_J=\bv$ and $\br_I'=\br_J'=\bv'$. Taking $A=\br$ and $A'=\br'$, we obtain $\langle \Lambda \rangle=\langle F,F'\rangle$. The claim now follows from Lemma \ref{lemma:sc} and Proposition \ref{prop:delta}.
    \end{proof}
    
    \begin{ex}[Example \ref{ex:running} continued]\label{ex:comb2}
    	We may use Proposition \ref{prop:simultv} to revisit the  simultaneous cancellation analysis at the vertices $V_1$, $V_2$, and $V_3$, see also Example \ref{ex:simult}.
    At the vertex $V_1$,   the lineality space $\langle \Lambda \rangle $ of the affine tangent space  is spanned by $(-1,3,1)$ and $(-2,-4,-3)$. 
    This contains $e_0$, but not $e_i$ for $i\neq 0$. Furthermore, it does not contain $e_i+e_j$ for any $0\leq i < j \leq 3$. It follows that the only index sets for which simultaneous cancellation occurs are $I,J$ with $I,J\subset \{1,2,3\}$.
    
    At the vertex $V_2$, $\langle \Lambda \rangle $ is spanned by $(3,2,4)$ and $(0,-1,-2)$. 
    This contains $e_1$, but not $e_i$ for $i\neq 1$, and not $e_i+e_j$ for any $0\leq i < j \leq 3$. It follows that the only index sets for which simultaneous cancellation occurs are $I,J$ with $I,J\subset \{0,2,3\}$.
    
    At the vertex $V_3$, the space $\langle \Lambda \rangle$ is spanned by $(0,1,2)$ and $(0,3,1)$, which is just $\langle  e_2,e_3\rangle$. We will not be able to use Proposition \ref{prop:simultv} to say anything about the index $I={0,1}$. However, we see that for other index sets, we only have simultaneous cancellation for $\{0,2\},\{1,2\}$ and $\{0,3\},\{1,3\}$.
    \end{ex}

    \subsection{The generic case}
    We may use the results of the previous two sections to describe the generic behaviour of the tropical Gauss map. Let $X\subset \PP^n$ be a curve satisfying the hypotheses of \S \ref{sec:setup} and Assumption \ref{ass:setup}. 
    
    \begin{thm}\label{thm:genericrestate}
    Fix a point $\alpha\in\trop(X)$, and let $\Lambda\subset \TP$ denote the affine tangent space to $\trop(X)$ at $\alpha$. Assume that $\langle \Lambda \rangle$ is not contained in any  $\langle J \rangle $ with $|J|=n-1$.
    \begin{enumerate}
    	\item If $\alpha$ is in the relative interior of an edge of $\trop(X)$, then there is a unique tropical tangent $\beta$ to $\alpha$, and its tropical Pl\"ucker coordinates are
    \[
    	\beta_I=\sum_{j\notin I} -\alpha_j\qquad I\subset \{0,\ldots,n\},\ |I|=2.
    \]
    
    \item Suppose that $\alpha$ is a vertex of $\trop(X)$, and $\langle \Lambda \rangle \cap (\langle e_i+e_j\rangle + \langle J'\rangle)=\{0\}$ for all $i,j$ and $J'$ with $|J'|=n-3$. Then $\beta$ is a tropical tangent to $\alpha$ if and only if 
    \[
    	\beta_I=\lambda_I+\sum_{j\notin I} -\alpha_j\qquad I\subset \{0,\ldots,n\},\ |I|=2
    \]
    with all $\lambda_I\geq 0$, at most one $\lambda_I\neq 0$, and $\lambda_I=0$ if $\langle E \rangle \subset \langle \overline I \rangle$ for an edge $E$ of $\trop(X)$ adjacent to $\alpha$.
    \end{enumerate}
    \end{thm}
    \begin{proof} 
    	We first suppose that $\alpha$ is in the relative interior of an edge $E$, in which case $\langle \Lambda\rangle=\langle E \rangle$. Lemma \ref{lemma:edge} and  the first two parts of Proposition \ref{prop:nocancel} yield the claim.
    
    	Suppose instead that $\alpha$ is a vertex. Using Lemma \ref{lemma:vert}, the assumption that $\langle \Lambda \rangle$ is not contained in any  $\langle I \rangle $ implies that at least one of $\Delta_I(\bv)$ or $\Delta_I(\bv')$ is non-zero. From the definition of $\nu_I$ (see~\ref{eqn:min}), we then have $\nu_I=0$ for all $I$.
    	Furthermore, one of $\Delta_I(\bv), \Delta_I(\bv'), \Delta_I(\bv)-\Delta_I(\bv')$ is 0 if and only if $\langle E \rangle \subset \langle \overline I \rangle$ for an edge $E$ of $\trop(X)$ adjacent to $\alpha$.		 When that is the case, the second part of Proposition~\ref{prop:nocancel} implies that $\val(z_I) = \nu_I=0$.
    
    Conversely, when none of $\Delta_I(\bv), \Delta_I(\bv'), \Delta_I(\bv)-\Delta_I(\bv')$ equals to $0$, we may use Proposition~\ref{prop:cancel} to conclude that  $\lambda_I$ can be any quantity greater than or equal to $0$ (note that, while we have not assumed non-colliding valuations, Proposition~\ref{prop:cancel}  still applies due to Remark~\ref{rem:cancel}). Finally,  Proposition~\ref{prop:simultv} guarantees that there is at most a single non-zero $\lambda_I$, and the claim follows. 
    \end{proof}
    
    In the setting of Theorem \ref{thm:genericrestate}, it is possible to give a more ``geometric'' description of the tropical tangents at a point $\alpha\in\trop(X)$. For any $i\neq j\in \{0,\ldots,n\}$ and any $\lambda\geq 0$, consider the tropical line
    \begin{align*}
    	\LL_{ij}(\lambda)= \conv\{0,\lambda(e_i+e_j)\}\cup \big(\lambda(e_i+e_j)+\RR_{\geq 0}\cdot e_i\big)\cup \big(\lambda(e_i+e_j)+\RR_{\geq 0}\cdot e_j\big)\\
    	\cup\bigcup_{k\neq i,j}\RR_{\geq 0}\cdot e_k.
    \end{align*}
    This is the unique tropical line with vertices $0$ and $\lambda(e_i+e_j)$, and unbounded directions $e_0,\ldots,e_n$. We set $\LL:=\LL_{ij}(0)$ for any $i,j$.
    
    \begin{cor}\label{cor:geom}
    Fix a point $\alpha\in\trop(X)$, and let $\Lambda\subset \TP$ denote the affine tangent space to $\trop(X)$ at $\alpha$. Assume that $\langle \Lambda \rangle$ is not contained in any  $\langle J \rangle $ with $|J|=n-1$.
    \begin{enumerate}
    	\item
    If $\alpha$ is in the relative interior of an edge of $\trop(X)$, then the unique tropical tangent at $\alpha$ is the tropical line whose vertex is at $\alpha$, namely $\alpha+\LL$.
    
    \item Suppose that $\alpha$ is a vertex of $\trop(X)$, and $\langle \Lambda \rangle \cap (\langle e_i+e_j\rangle + \langle J'\rangle)=\{0\}$ for all $i,j$ and $J'$ with $|J'|=n-3$. Then 
    	the tropical tangents at $\alpha$ are exactly the tropical lines $\alpha+\LL_{ij}(\lambda)$ for $\lambda\geq 0$ and those $i,j$ such that 
    	the affine tangent space $\Lambda'$ of $\alpha+\LL_{ij}(\lambda)$ at $\alpha$ intersects $\trop(X)$ dimensionally transversally at $\alpha$, that is,
    	$\langle E \rangle \not\subseteq \langle \Lambda'\rangle $ for every edge $E$ of $\trop(X)$ adjacent to $\alpha$.
    \end{enumerate}
    \end{cor}
    \begin{proof}
    	This follows directly from Theorem \ref{thm:genericrestate} by translating between tropical Pl\"ucker coordinates and the corresponding tropical line in $\TP$. This translation can be done using the equations from Theorem \ref{thm:dual} below. Details are left to the reader.
    \end{proof}
    
    In Figure \ref{fig:geom}, we illustrate the two cases of Corollary \ref{cor:geom}. Each grey part represents a local picture of a tropical curve, whereas the black part represents a tropical tangent at $\alpha$.     The picture on the left depicts the first case of Corollary \ref{cor:geom}, namely, $\alpha$ is in the relative interior of an edge and admits  a unique tropical tangent. The picture on the right depicts the second case of the corollary, namely, $\alpha$ is a vertex, and there is a family of tangents coming from the choice of $(i,j)$ and a parameter $\lambda$.

    In the more general situation where the genericity hypotheses of Corollary \ref{cor:geom} are not met, we do not know of a satisfying geometric description of the tropical tangents. One obstruction is that in the general situation, the set of tropical tangents is not determined solely by the local structure of $\trop(X)$ near $\alpha$. However, in the case of tropically smooth plane curves, there is a satisfying geometric description; see \cite{ilten-len}.
    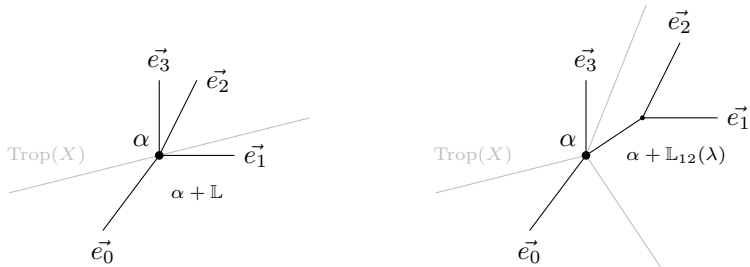
\begin{figure}
    	\begin{tikzpicture}[scale=.5]
    		\draw[lightgray] (-4,-1) -- (4,1);
    		\draw (0,2) -- (0,0) -- (2,0); 
    		\draw (1,2) -- (0,0) -- (-1.5,-2); 
    \draw[fill] (0,0) circle [radius=0.1];
    	\node [above left] at (0,0) {$\alpha$};
    	\node[lightgray] at (-3,0) {\scriptsize$\trop(X)$};
    	\node [right] at (2,0) {$\vec{e_1}$};
    	\node [above] at (0,2) {$\vec{e_3}$};
    	\node [right] at (1,2) {$\vec{e_2}$};
    	\node [below] at (-1.5,-2) {$\vec{e_0}$};
    	\node at (1,-1) {\scriptsize$\alpha+\LL$};
    	\end{tikzpicture}
    	\qquad\qquad
    	\begin{tikzpicture}[scale=.5]
    \draw (0,2) -- (0,0) -- (-1.5,-2);
    \draw (0,0) -- (1.5,1);
    \draw (3.5,1) -- (1.5,1) -- (2.5,3);
    \draw [lightgray] (-4,-1) -- (0,0) -- (1.6,4);
    \draw [lightgray] (0,0) -- (2,-3);
    		\draw[fill] (0,0) circle [radius=0.1];
    		\draw[fill] (1.5,1) circle [radius=0.05];
    	\node [above left] at (0,0) {$\alpha$};
    	\node[lightgray] at (-3,0) {\scriptsize$\trop(X)$};
    	\node [right] at (3.5,1) {$\vec{e_1}$};
    	\node [above] at (2.5,3) {$\vec{e_2}$};
    	\node [above] at (0,2) {$\vec{e_3}$};
    	\node [below] at (-1.5,-2) {$\vec{e_0}$};
    	\node at (2.4,0) {\scriptsize$\alpha+\LL_{12}(\lambda)$};
    	\end{tikzpicture}
    	\caption{Tropical tangents (in black) for tropical curves (in grey) in $\TT^3$}\label{fig:geom}
    \end{figure}

    \section{Tropical Tangential and Dual Varieties}\label{sec:bundle}
    \subsection{Projections of tautological bundles}\label{sec:bundlep}
    The Grassmannian $\Gr(m+1,V)$ comes equipped with two tautological vector bundles. The \emph{tautological sub-bundle} $\cS$ is the sub-bundle of the trivial bundle $\Gr(m+1,V)\times V$ whose fiber at $P\in\Gr(m+1,V)$ is the subspace of $V$ corresponding to $P$. The \emph{tautological quotient bundle} $\cQ$ is the quotient of the trivial bundle $\Gr(m+1,V)\times V$ whose fiber at $P\in\Gr(m+1,V)$ is the quotient of $V$ by the subspace of $V$ corresponding to $P$.
    
    Consider an $m$-dimensional variety $X\subset \PP^n$ (e.g. a curve as in \S\ref{sec:setup}). We restrict the tautological sub-bundle $\cS$ of $\Gr(m+1,n+1)$ to the closure of the image of the Gauss map $\G(X)$ and projectivize its fibers to obtain a projective bundle
    \[
    	\PP\left(\cS_{|\G(X})\right)\subset \G(X)\times \PP^n.
    \]
    The dual $\cQ^*$ of the quotient bundle is a sub-bundle of the trivial bundle $\Gr(m+1,n+1)\times (\KK^{n+1})^*$. We may similarly restrict this to $\G(X)$ and projectivize to obtain a projective bundle
    \[
    	\PP\left(\cQ^*_{|\G(X})\right)\subset \G(X)\times (\PP^n)^*.
    \]
    
    \begin{lemma}Let $X\subset \PP^n$ be a projective variety.
    The tangential variety  $\tau(X)$ is the image of $\PP(\cS_{|\G(X)})$ under the projection to $\PP^n$.
    Similarly, the dual variety $X^*$ is the image of $\PP(\cQ^*_{|\G(X)})$ under the projection to $(\PP^n)^*$.
    \end{lemma}
    \begin{proof}
    This follows from a straightforward unraveling of the definitions of $\tau(X)$ and $X^*$.
    \end{proof}

    It is easy to write down equations cutting out
    \[
    	\PP(\cS)\subset \Gr(m+1,n+1)\times \PP^n.
    \]
    Indeed, it follows from the Laplace expansion of the determinant that the ${n+1}\choose {m+2}$ equations
    \begin{equation}\label{eqn:taut}
    	\sum_{i=0}^{m+1} (-1)^i p_{J\setminus j_i} \cdot x_i=0;\qquad J=\{j_0,\ldots,j_{m+1}\}\subset \{0,\ldots,n\},\ j_i<j_{i+1}
    \end{equation}
    cut out $\cS$. Here, the $p_I$ are Pl\"ucker coordinates on $\Gr(m+1,n+1)$ and the $x_i$ are coordinates on $\PP^{n}$.
    We will be particularly interested in the case $m=n-2$, in which case we obtain exactly $n+1$ equations.
    
    Under the canonical isomorphism $\Gr(m+1,\KK^{n+1})\to \Gr(n-m,(\KK^{n+1})^*)$, the tautological sub-bundle on $\Gr(m+1,\KK^{n+1})$ is identified with the dual of the tautological quotient bundle on $\Gr(n-m,(\KK^{n+1})^*)$, and vice versa. In particular,  using \eqref{eqn:dualp} along with \eqref{eqn:taut} we have that $\PP(\cQ^*)$ is cut out by
    \begin{equation}
    \begin{aligned}\label{eqn:dual}
    	\sum_{i=0}^{n-m} (-1)^{i+j_i+\sum J} & p_{\overline{J\setminus j_i}} \cdot y_i=0; \\
    	 &J=\{j_0,\ldots,j_{n-m}\}\subset \{0,\ldots,n\},\ j_i<j_{i+1}
    \end{aligned}
    \end{equation}
    where the $y_i$ are coordinates on $(\PP^{n})^*$.

    \subsection{Tropicalization}
    We now fix a curve $X\subset \PP^n$. Let $\cJ$ be the collection of Pl\"ucker indices $J$ with $|J|=2$ such that $p_J$ vanishes on  $\G(X)\subset \Gr(2,n+1)$, that is, $\G(X)$ is contained in the hyperplane $V(p_J)$. We will view $\G(X)$ as a subvariety of the projective space $\PP(\cJ)$ where we have eliminated $p_J$, $J\in \cJ$. See Theorem \ref{thm:gausszero} for a tropical characterization of $\cJ$.

    In the following theorem, we \emph{do not} need to make any of the assumptions of \S\ref{sec:setup}, in particular, we are not assuming that $X$ is a complete intersection curve.
    \begin{thm}\label{thm:dual}
    The tropical variety $\trop(X^*)$ is the projection to $(\TP)^*$ of  the intersection of
    $\trop(\G(X))\times (\TP)^*$
    with the $n+1$ tropical hypersurfaces
    \[
    \trop(V(\sum_{\substack{i\neq j\\ \{i,j\}\notin \cJ}}p_{ij}y_i))
    \]
    for $j=0,\ldots,n$.
    
    Similarly, the tropical variety $\trop(\tau(X))$ is the projection to $\TP$ of the intersection of $\trop(\G(X))\times\TP$ with the 
    ${n+1}\choose3$ tropical hypersurfaces
    \[
    \trop(V(p_{ij}x_k + p_{ik}x_j+p_{jk}x_i))
    \]
    where we are setting $p_J=0$ if $J\in\cJ$.
    \end{thm}
    \begin{proof}
    	We begin with $X^*$. By our choice of $\cJ$, $\G(X)$ has non-trivial intersection with the dense torus of $\PP(\cJ)$.  
    Since tropicalization commutes with monomial maps, 
    \[
    	\trop(X^*)=\trop (\pi (\PP(\cQ^*_{\G(X)})))=(\trop \pi)\left(\trop (\PP(\cQ^*_{\G(X)}))\right)
    \]
    where $\pi:\G(X)\times (\PP^n)^*\to (\PP^n)^*$ is the projection to the second factor, and $\trop(\pi)$ is the tropicalization of this map (which is also just a projection).
    Hence, it remains to show that $\trop (\PP(\cQ^*_{\G(X)}))$ has the desired form.
    
    The $n+1$ tropical hypersurfaces stated in the theorem are just the tropicalizations of the hypersurfaces cut out by the equations \eqref{eqn:dual} (specialized to the case $m=1$). As noted above, these are the equations cutting out the dual quotient bundle. On the other hand, for any fixed values of Pl\"ucker coordinates in $\G(X)$, the equations \ref{eqn:dual} form a tropical basis \cite[Proposition 1.6]{sturmfels-polytopes}.
    It follows that 	$\trop (\PP(\cQ^*_{\G(X)}))$ is the intersection of $\trop(\G(X))$ with the $n+1$ hypersurfaces in the statement of the theorem.
    
    The argument for $\tau(X)$ is identical, except that we are instead looking at $\trop (\PP(\cS_{\G(X)}))$ and the equations \eqref{eqn:taut}. These equations again form a tropical basis and the claim follows.
    \end{proof}
    \begin{rem}
	    The equations \eqref{eqn:dual} may also be used to obtain similar descriptions of $\trop(X^*)$ and $\trop(\tau(X))$ for varieties of dimension larger than one. We only state the result for curves in Theorem \ref{thm:dual} to keep notation simpler, and since this is the case of interest for this paper. 
    \end{rem}

    Since we are able to effectively compute $\trop(\G(X))$ for curves satisfying the hypotheses of Theorem \ref{thm:alltangents} (see also Remark \ref{rem:GX}), we are able to effectively compute $\trop(X^*)$ and $\trop(\tau(X))$ as well. We will illustrate this with our example of a space curve, Example \ref{ex:running}. For the plane curve in Example \ref{ex:runningP}, $\trop(X^*)$ may simply be identified with $\trop(\G(X))$, and $\trop(\tau(X))$ is the entire tropical plane.
    
    \begin{ex}[Continuation of Example \ref{ex:running} (a curve in $\PP^3$)]\label{ex:rundt}
    	We consider $\trop (\tau(X))$ and $\trop(X^*)$ for $X$ as in Example \ref{ex:running}. We will only compute some pieces of them here, and will finish the computation in Example \ref{ex:runningtrop}.
    
    	We first consider the vertex $V_3=(3,0,0)$. We know the possible tropical tangents from Example \ref{ex:simult}. Theorem \ref{thm:dual} allows us to compute the contributions to $\trop(X^*)$ and $\trop(\tau(X))$. Using the notation of Table \ref{table:vertices}, we record the top-dimensional contribution of each kind of tropical tangent at $V_3$ to $\trop(X^*)$ in Table \ref{table:rundt}.
    	In each line of the table, if $s_{ij}$ is not mentioned, then we assume it to be zero. Combining these contributions, we obtain a total top-dimensional contribution to $\trop(X^*)$ consisting of the union of $-V_3+\RR_{\geq 0}\cdot e_i+\RR_{\geq 0}\cdot e_j$, where $0\leq i < j \leq 3$.
    In a similar fashion, one obtains that the total top-dimensional contribution to $\trop(\tau(X))$ is the union of $V_3-3(e_2+e_3)+\RR_{\geq 0}\cdot e_i+\RR_{\geq 0}\cdot e_j$, where $0\leq i < j \leq 3$.
    
    The top-dimensional contributions of the edges $E_2$, $E_6$, and $E_7$ to $\trop(X^*)$ and $\trop(\tau(X))$ are respectively listed in Tables \ref{table:rundt2} and \ref{table:rundt3}.
    We will describe the other contributions to $\trop(X^*)$ and $\trop(\tau(X))$ in Example \ref{ex:runningtrop}.
    \end{ex}
    \begin{table}
    \begin{tabular}{l l }
    Tropical tangents at $V_3$ & Top-dimensional contributions to $\trop(X^*)$\\
    	\toprule
    	$s_{23}\in [0,\infty)$ & $-V_3+\RR_{\geq 0}\cdot e_0+\RR_{\geq 0}\cdot e_1$\\
    		\addlinespace
    	$s_{01}\in [0,\infty)$ & $-V_3+\RR_{\geq 0}\cdot e_2+\RR_{\geq 0}\cdot e_3$\\
    		\addlinespace
    		$s_{13}=s_{03}\in[0,3]$ & $\conv\{-V_3,-V_3+3e_2+3e_3\}+\RR_{\geq 0}\cdot e_2$, \\
     &$\conv\{-V_3,-V_3+3e_2\}+\RR_{\geq 0}\cdot e_i$, $i=0,1$\\
    \addlinespace		
     $s_{12}=s_{02}\in[0,3]$ & $\conv\{-V_3,-V_3+3e_2+3e_3\}+\RR_{\geq 0}\cdot e_3$, \\
     &$\conv\{-V_3,-V_3+3e_3\}+\RR_{\geq 0}\cdot e_i$, $i=0,1$\\
     \addlinespace
     $s_{13}=3,s_{03}\in[3,\infty)$& $-V_3+3e_2+\RR_{\geq 0}\cdot e_1+\RR_{\geq 0}\cdot e_2$\\
     \addlinespace
     $s_{03}=3,s_{13}\in[3,\infty)$& $-V_3+3e_2+\RR_{\geq 0}\cdot e_0+\RR_{\geq 0}\cdot e_2$\\
     \addlinespace
     $s_{12}=3,s_{02}\in[3,\infty)$& $-V_3+3e_3+\RR_{\geq 0}\cdot e_1+\RR_{\geq 0}\cdot e_3$\\
     \addlinespace
     $s_{02}=3,s_{12}\in[3,\infty)$& $-V_3+3e_3+\RR_{\geq 0}\cdot e_0+\RR_{\geq 0}\cdot e_3$\\
    	\bottomrule
    	\end{tabular}
    	\caption{$\trop(X^*)$ from $V_3$ in Example \ref{ex:running} and \ref{ex:rundt}}\label{table:rundt}
    \end{table}
    
    \begin{table}
    \begin{tabular}{l l }
    Cell of $\trop(X)$ & Top-dimensional contributions to $\trop(X^*)$\\
    \toprule
    $E_2$&$-E_2+\RR_{\geq 0}\cdot e_i,\qquad i=0,\ldots,3$\\
    \addlinespace
    $E_6$& $-E_6+\RR_{\geq 0}\cdot e_i\qquad i=0,1$;\\
    	&$\conv\{-V_3,-V_3+3(e_2+e_3)\}-E_6+\RR_{\geq 0}\cdot (-e_2)$;\\
    	&$-V_3+3(e_2+e_3)+\RR_{\geq 0}\cdot (-e_2)+\RR_{\geq 0} \cdot e_3$\\
    	\addlinespace
    $E_7$ & $-E_7+\RR_{\geq 0}\cdot e_i\qquad i=0,1$;\\
    &$\conv\{-V_3,-V_3+3(e_2+e_3)\}-E_7+\RR_{\geq 0}\cdot (6e_2+5e_3)$;\\
    	&$-V_3+3(e_2+e_3)+\RR_{\geq 0}\cdot (6e_2+5e_3)+\RR_{\geq 0} \cdot e_3$\\
    \bottomrule
    \end{tabular}
    	\caption{$\trop(X^*)$ in Example \ref{ex:running} and \ref{ex:rundt}}\label{table:rundt2}
    \end{table}
    
    \begin{table}
    \begin{tabular}{l l }
    Cell of $\trop(X)$ & Top-dimensional contributions to $\trop(\tau(X))$\\
    \toprule
    $E_2$&$\conv \{E_2,V_3-3(e_2+e_3)\}$;\\
    &$E_2+\RR_{\geq 0}\cdot e_i\qquad i=2,3$;\\
    &$\conv\{V_3,V_3-3(e_2+e_3)\}+\RR_{\geq 0}\cdot e_i\qquad i=0,1$\\
    	\addlinespace
    $E_6$&$V_3+\RR_{\geq 0}\cdot e_2+\RR_{\geq 0}\cdot e_3$;\\
    	&$\conv\{V_3,V_3-3(e_2+e_3)\}+\RR_{\geq 0}\cdot e_2$;\\
    	&$V_3-3(e_2+e_3)+\RR_{\geq 0}\cdot e_2+\RR_{\geq 0}\cdot e_i,\qquad i=0,1$\\
    	\addlinespace
    $E_7$&$E_7+\RR_{\geq 0}\cdot e_i^*,\qquad i=2,3$;\\
    	&$V_3-3(e_2+e_3)+\RR_{\geq 0}\cdot (-6e_2-5e_3)+\RR_{\geq 0} \cdot e_i,\qquad i=0,1$\\
    	\bottomrule
    \end{tabular}
    	\caption{$\trop(\tau(X))$ in Example \ref{ex:running} and \ref{ex:rundt}}\label{table:rundt3}
    \end{table}
    
    \subsection{Special Cases}
    Throughout this subsection, we continue using the notation of Section~\ref{sec:notation}.
    	While $\trop(X^*)$ and $\trop(\tau(X))$ can be rather complicated in general, there are some pieces that are especially easy to describe. 
    	\begin{prop}\label{prop:bergman}
		Assume that $X$ satisfies the hypotheses of \S \ref{sec:setup} and Assumption \ref{ass:setup} and fix an edge $E$ of $\trop(X)$. Suppose that for each $J\subset \{0,\ldots,n\}$ of size $2$, either $\langle E \rangle \cap \langle \overline J \rangle = \{0\}$ or $\langle \trop(X)\rangle \subset \langle \overline J \rangle$.
    		  Let $L\subset (\KK^{n+1})^*$ be the rowspan of $\bv$ (viewed as an $(n-1)\times (n+1)$ matrix with entries in $\KK$) and $L^\perp\subset (\KK^{n+1})$ its orthogonal complement.
    In this case, the contribution of $E$ to $\trop(X^*)$ consists of 
    	\[
    -E+\trop(\PP(L))
    	\]
    and the contribution of $E$ to $\trop(\tau(X))$ consists of 
    	\[
    E+\trop(\PP(L^\perp)).
    	\]
    \end{prop}
    \begin{proof}
    	Let $\cJ$ consist of those $J\subset \{0,\ldots,n\}$ of size $2$ such that the Pl\"ucker coordinate $p_J$ vanishes on $\G(X)$. By Theorem \ref{thm:gausszero}, this is exactly those $J$ such that $\langle \trop(X)\rangle \subset \langle \overline J \rangle$.
    	By Lemma \ref{lemma:edge},  the assumption  $\langle E \rangle \cap \langle \overline J \rangle= \{0\}$ guarantees that $\nu_{J}=0$ for all $J$ such that $ J\notin \cJ$.
    	For any $\alpha\in E$ there is a unique tropical tangent by Proposition \ref{prop:cancel}. This has Pl\"ucker coordinates $0$ for $J\in \cJ$, and satisfies
    \begin{equation}\label{eqn:bergman}
    	\val(p_{J})=\val(q_{J})=-\sum_{j\notin J} \alpha_j,\qquad  J\notin \cJ.
    \end{equation}
    For $P\in X$ tropicalizing to $\alpha$, let $\bv\cdot P^{-1}$ denote the matrix obtained by multiplying the $i$th column of $\bv$ with $\frac{1}{x_i(P)}$ (after choosing homogeneous coordinates for $P$). The linear space $(\rowspan\bv\cdot P^{-1})^\perp$ has Pl\"ucker coordinates $p_{ J}$ that vanish for $ J\in \cJ$, and satisfy \eqref{eqn:bergman}.
    Hence, the tropicalization of the projectivization of $(\rowspan\bv\cdot P^{-1})^\perp$  agrees with the contribution of $\alpha$ to $\trop(\tau(X))$.
    Similarly, the tropicalization of the projectivization of $\rowspan\bv\cdot P^{-1}$   agrees with the contribution of $\alpha$ to $\trop(X^*)$.
    
    We complete the proof with the straightforward observations that
    \begin{align*}
    \trop(\PP((\rowspan\bv\cdot P^{-1})^\perp))=
    \trop(\PP((\rowspan\bv)^\perp\cdot P))\\= \trop(\PP((\rowspan\bv)^\perp))+\trop(P)=\trop(\PP(L^\perp))+\alpha
    \end{align*}
    and
    \begin{align*}
    \trop(\PP(\rowspan\bv\cdot P^{-1}))= \trop(\PP(\rowspan\bv))-\trop(P)=\trop(\PP(L^\perp))-\alpha.
    \end{align*}
    
    \end{proof}
    
    \noindent The support of the tropical varieties $\trop(\PP(L))$ and $\trop(\PP(L^\perp))$ is the same as respectively the Bergman and co-Bergman fans of $\bv$, see e.g.~\cite[\S9.3]{sturmfels-poly}. This contribution to $\trop(X^*)$ is very similar to the description of the tropical dual variety to a toric variety, see \cite[Theorem 1.1]{sturmfels}.

    We call a vertex $\alpha\in\trop(X)$ \emph{tame} if the affine tangent plane $\Lambda$ of $\trop(X)$ at $\alpha$ does not satisfy $\langle \Lambda \rangle \subseteq \langle \overline J \rangle$ for any $J\subset \{0,\ldots,n\}$, $|J|=2$. The contribution to $\trop(X^*)$ of tame vertices is especially simple:
    \begin{prop}\label{prop:vertcontrib}
	    Assume that $X$ satisfies the hypotheses of \S \ref{sec:setup} and Assumption \ref{ass:setup} and $X_1,\ldots,X_{n-1}$ have very general valuations. Fix a tame vertex $\alpha$ of $\trop(X)$. The top-dimensional contribution of $\alpha$ to $\trop(X^*)$ consists of the union of 
    \[
    		-\alpha+\sum_{i\notin J}\RR_{\geq 0}\cdot e_i,\qquad 
    	\]
    	over all $J\subset \{0,\ldots,n\}$, $|J|=2$ such that for every edge $E$ adjacent to $\alpha$, $\langle E \rangle \cap \langle \overline J\rangle = \{0\}$.
    Furthermore, any tropical hyperplane corresponding to a generic point in one of the above sets contains a unique tropical tangent.
    \end{prop}
    Before proving this, we state a description of the contribution for $\trop(\tau(X))$. 
    For any tropical tangent $\beta$ of a vertex $\alpha\in\trop(X)$, let $\beta_{ij}$ be tropical Pl\"ucker coordinates. We may always choose these so that 
    $\beta_{ij}\geq \alpha_i+\alpha_j$, with equality holding for at least one $i,j$.
    We then set 
    \[\delta_{ij}=\beta_{ij}-(\alpha_i+\alpha_j).\]
    In the setting of the following proposition, we will see below that $\delta_{ij}=\val(z_J)$ for $J=\{ij\}$, see \eqref{eqn:pz}.
    
    \begin{prop}\label{prop:vertcontribtau}
	    Assume that $X$ satisfies the hypotheses of \S \ref{sec:setup} and Assumption \ref{ass:setup} and $X_1,\ldots,X_{n-1}$ have very general valuations. Fix a tame vertex $\alpha$ of $\trop(X)$. The contribution of $\alpha$ to $\trop(\tau(X))$ consists of the union over all tropical tangents $\beta$ to $\alpha$ of 
    \[
    \alpha+	\sum_{j\neq i} \delta_{ij}e_j+\max_{j\neq i} \delta_{ij}e_i+\RR_{\geq 0} e_i,\qquad i=0,\ldots,n.
    \]
    \end{prop}
    
    We will begin the proofs of Propositions \ref{prop:vertcontrib} and \ref{prop:vertcontribtau} together. Throughout the rest of this section, we fix a tame vertex $\alpha$. By Lemma \ref{lemma:vert}, this guarantees that $\nu_J(\alpha)=0$ for all $J$. We let $\Lambda$ be the affine tangent plane of $\trop(X)$ at $\alpha$.
    
    \begin{lemma}\label{lemma:notall}
    	For any point $p\in X$ tropicalizing to $\alpha$, there exists $J\subset \{0,\ldots,n\}$, $|J|=2$ such that $\val(z_J)=0$.
    \end{lemma}
    \begin{proof}
    	Suppose for the sake of contradiction that the statement does not hold. Since $\nu_J(\alpha)=0$ for all $J$, this means that we have simultaneous cancellation for all index sets $J$. By applying Proposition \ref{prop:simultv} to some $ {\{i,j\}}$ and $ {\{i,k\}}$, this implies that $\langle \Lambda\rangle \cap \langle \overline {\{i,j,k\}} \rangle\neq \{0\}$. Since by assumption $\Lambda \nsubseteq \langle \overline J \rangle$ for any $J$ with $|J|=2$, we conclude that  $0\neq \langle \Lambda\rangle \cap \langle \overline {\{i,j\}} \rangle$
     is independent of $i,j$. But 
     \[
    	 \bigcap_{i,j}\langle \overline {\{i,j\}} \rangle=\{0\},
     \]
     a contradiction.
    \end{proof}
    
    Consider any tropical tangent $\beta$ of $\alpha$. By the  Lemma~\ref{lemma:notall} and \eqref{eqn:pz}, it follows that $\delta_{ij}$ is the valuation of $z_{ij}$  at any point $P\in X$ with $\trop(\G_X(P))=\beta$.
    
    \begin{lemma}
    	For any tropical tangent $\beta$ at $\alpha$, there exists $\mcH(\beta)\subsetneq \{0,\ldots,n\}$ such that for all $i,j$, $\delta_{ij}>0$ if and only if $i,j\in \mcH(\beta)$.
    \end{lemma}
    \begin{proof}
    	If $\delta_{ij}=0$ for every $\{i,j\}$ then there is nothing to prove.
    	  Otherwise, suppose that  $\delta_{k\ell}>0$ for some $k,\ell$. Let $\mcH(\beta)=\{\ell\}\cup\{i\ |\ \delta_{i\ell}>0\}$. It follows from Proposition \ref{prop:simultv} that $\delta_{ij}>0$ if and only if $i,j\in \mcH(\beta)$. Lemma \ref{lemma:notall} shows that $\mcH(\beta)\neq \{0,\ldots,n\}$.
    \end{proof}
    
    Any tropical tangent $\beta$ satisfies the \emph{tropical Pl\"ucker relations} (cf. \cite[Proof of Theorem 3.4]{ss}): the minimum of
    \[
    	\beta_{ij}+\beta_{k\ell},
    	\beta_{ik}+\beta_{j\ell},
    	\beta_{i\ell}+\beta_{jk}
    \]
    is attained at least twice for any distinct $i,j,k,\ell$. Applying this for $i,j,k\in \mcH(\beta)$ and $\ell\notin \mcH(\beta)$, we obtain that the minimum of 
    \[
    	\delta_{ij},\delta_{ik},\delta_{jk}
    \]
    is attained at least twice. It follows that for any tropical tangent, there exists $\mcH_{\max}(\beta)\subseteq\mcH(\beta)$ such that $\delta_{ij}$ is maximal if and only if $i,j\in\mcH_{\max}(\beta)$.
    
    \begin{lemma}\label{lemma:omega}
    	Fix a tame vertex $\alpha$ and $0\leq i<j\leq n$.
    	\begin{enumerate}
    		\item	There exist numbers $\omega_{ij}(k,\ell)\in\RR_{\geq 0}\cup\{\infty\}$ such that for any tropical tangent $\beta$ to $\alpha$ with $\delta_{ij}$ maximal, \[\delta_{k\ell}\leq \omega_{ij}(k,\ell),\] and strict inequality holds only if $\delta_{k\ell}$ is also maximal.
    		\item	There exists a finite subset $\Omega\subset \RR$ such that for any tropical tangent $\beta$ to $\alpha$, $\delta_{ij}\in \Omega$ unless $\delta_{ij}$ is maximal.
    	\end{enumerate}
    		\end{lemma}
    \begin{proof}
    	The first claim follows directly from Proposition \ref{prop:algorithm}. The second claim follows by taking the union of all $\omega_{ij}(k,\ell)$ as $i,j,k,\ell$ vary.
    \end{proof}
    
    \begin{proof}[Proof of Proposition \ref{prop:vertcontrib}]
    Any tropical hyperplane containing a tropical tangent to $\alpha$ must intersect $\trop(X)$ at $\alpha$. It follows that the corresponding point in $\trop(X^*)$ must be contained in a set of the form
    \[
    		-\alpha+\sum_{i\notin J}\RR_{\geq 0}\cdot e_i
    	\]
    for some $J$ with $|J|=2$. 
    
    Fix some $J$ with $|J|=2$. After reordering coordinates, we will assume that $J={\{(n-1),n\}}$. Any point  
    $\gamma\in -\alpha+\sum_{i\notin J}\RR_{\geq 0}\cdot e_i$ can be written as 
    
    \[\gamma=-\alpha+\lambda\]
    with $\lambda_i\geq 0$ for all $i$ and $\lambda_{n-1}=\lambda_n=0$. 
    
    For $\gamma$ to be a contribution of $\alpha$ to $\trop(X^*)$, by Theorem \ref{thm:dual} the minimum of
    \begin{equation}\label{eqn:mintwice}
    	\delta_{0n}+\lambda_0,\delta_{1n}+\lambda_1,\ldots,\delta_{(n-2)n}+\lambda_{n-2},\delta_{(n-1)n}+\lambda_{n-1}=\delta_{(n-1)n}
    \end{equation}
    must be attained at least twice. Restrict to the non-empty open set of those $\gamma$ such that $\lambda_i>0$  and $\lambda_i-\lambda_j\notin \Omega-\Omega$ for all $i<j<n$. We see that the minimum must be attained by a term for which $\delta_{in}$ is maximal (and non-zero). It follows that $\delta_{(n-1)n}$ must be maximal among all $\delta_{ij}$. In particular, $\delta_{(n-1)n}>0$, so $\langle E \rangle \cap \langle \overline J \rangle = \{0\}$ for all edges $E$ adjacent to $\alpha$ by Lemma \ref{lemma:vert}.
    
    We have seen that the top-dimensional contribution of $\alpha$ to $\trop(X^*)$ is contained in the set claimed in the theorem. We next show that, as long as 
    $\langle E \rangle \cap \langle\overline J \rangle = \{0\}$ for all edges $E$ adjacent to $\alpha$, every point $\gamma$ as above is a contribution of $\alpha$ to $\trop(X^*)$. Using Lemma \ref{lemma:omega}, we will reorder indices (fixing $n-1$ and $n$) such that $\omega_{(n-1)n}(i,n)\leq \omega_{(n-1)n}(i+1,n)$ for $i=0,\ldots,n-2$. This implies that 
    as long as $\delta_{(n-1)n}$ is maximal, $\delta_{in}\leq \delta_{(i+1)n}$.
    
    To show that $\gamma$ is a contribution of $\alpha$ to $\trop(X^*)$, again by Theorem \ref{thm:dual} we must exhibit a tropical tangent $\beta$ such that for each $i=0,\ldots,n$, the minimum of 
    \begin{equation}\label{eqn:deltalambda}
    	\{\delta_{ij}+\lambda_j\}_{j\neq i}
    \end{equation}
    is attained at least twice.
    We set $\delta_{0n}=0$ and inductively define
    \[
    	\delta_{in}=
    		\min \left(\{\delta_{jn}+\lambda_j\}_{j<i}\cup\{\omega_{(n-1)n}(i,n)\}\right).
    \]
    By Lemma \ref{lemma:omega} it is straightforward to verify that, after imposing the condition that $\delta_{(n-1)n}$ is maximal, this choice of the $\delta_{in}$ determines a unique valid tropical tangent of $\alpha$.
    
    Furthermore, the minimum of \eqref{eqn:deltalambda} is attained at least twice for $i=0,\ldots,n$. Indeed, for $i=0,\ldots, n-2$, the tropical Pl\"ucker relations and inductive construction guarantee \[
    	\delta_{i(n-1)}=\delta_{in}\leq \delta_{ij}+\lambda_j\qquad \textrm{for all}\ j<n-1.
    \]
    Likewise, for $i=n,n-1$ 
    \[
    	\delta_{(n-1)n}=\min \{\delta_{ji}+\lambda_j\}_{j<i}.
    \]
    
    For the final claim of the proposition, it remains to show that for sufficiently general $\gamma$, there a unique $\delta$ giving rise to it. We already know that $\delta_{(n-1)n}$ must be maximal, and as soon as its value is determined, so are all $\delta_{ij}$. In particular, for each $i\leq n-2$, $\delta_{in}=\min \{\omega_{(n-1)n}(i,n),\delta_{(n-1)n}\}$. Furthermore, the minimum of \eqref{eqn:mintwice} is attained by $\delta_{(n-1)n}$. 
    It follows that the only possibility for $\delta_{(n-1)n}$ is
    \[\delta_{(n-1)n}=\min_{i<n-1} \omega_{(n-1)n}(i,n)+\lambda_i.\]
    \end{proof}
    
    \begin{proof}[Proof of Proposition \ref{prop:vertcontribtau}]
    	Let $\beta$ be a tropical tangent. The corresponding line in $\TP$ can be described using the equations from Theorem \ref{thm:dual}; it consists of a compact part, along with rays in each direction $e_i$. Using these equations, it is straightforward to check that the ray in direction $e_i$ has 
    \[\sum_{j\neq i} \beta_{ij}e_j+\left(\max_{j,k\neq i} \beta_{ij}+\beta_{ik}-\beta_{jk}\right)e_i\]
    as its vertex. By the discussion above, this is the same as
    \[
    \alpha+	\sum_{j\neq i} \delta_{ij}e_j+\max_{j\neq i} \delta_{ij}e_i.
    \]
    Hence, the set described in the proposition is contained in the contribution of $\alpha$ to $\trop(\tau(X))$. 
    
    It remains to see that the compact part of the lines considered above are contained in the set described in the theorem. To that end, consider any tropical tangent $\beta$, and let $s$ be the maximum value of $\delta_{ij}$.
    Since $\delta_{ij}=s$ for all $i,j\in\mcH_{\max}(\beta)$, it follows that the rays in directions $e_i$ with $i\in \mcH_{\max}(\beta)$ all share a common vertex $\gamma_s$. We may decrease the value of $s$ to $s'$ without changing any other tropical Pl\"ucker coordinates.
    As $s$ decreases, the corresponding line $\beta_{s'}$ changes by shrinking the compact edge $E$ attached to $\gamma_s$. This continues until $\mcH_{\max}(\beta)\subsetneq \mcH_{\max}(\beta_s)$. The edge $E$ is contained in the union of the vertices $\gamma_{s'}$. In particular, the edge $E$ is contained in the set stated in the proposition.
    
    We may continue this procedure iteratively until $\max \delta_{ij}=0$, at which point the corresponding line has no compact edges. By the argument above, it follows that every edge of $\beta$ is included in the set described in the proposition.
    \end{proof}

    \begin{ex}[Continuation of Example \ref{ex:running} (a curve in $\PP^3$)]\label{ex:runningtrop}
    	We continue the computations from Example \ref{ex:rundt} to complete the description of $\trop(X^*)$ and $\trop(\tau(X))$. The remaining contributions to $\trop(X^*)$ all follow from Proposition \ref{prop:bergman} and Proposition \ref{prop:vertcontrib} and are listed in Table \ref{table:dual}.
    	The remaining contributions to $\trop(\tau(X))$ all follow from Proposition \ref{prop:bergman} and Proposition \ref{prop:vertcontribtau} and are listed in Table \ref{table:tau}.
    
    	It is interesting to observe that while the contributions from the vertices $V_1$ and $V_2$ to $\trop(X^*)$ are not affected by simultaneous cancellation (in the sense of \S \ref{sec:simult}), the contributions to $\trop(\tau(X))$ very much are.
    \end{ex}
    	
    \begin{table}
    \begin{tabular}{l l l l}
    Cell of $\trop(X)$ & Top-dimensional contributions to $\trop(X^*)$\\
    	\toprule
    	$E_1$ & $-E_1+\RR_{\geq 0}\cdot e_i$, $i=0,\ldots,3$\\ 
    	\addlinespace
    	$E_3$ & $-E_3+\RR_{\geq 0}\cdot e_i$, $i=0,\ldots,3$\\
    	\addlinespace
    	$E_4$ & $-E_4+\RR_{\geq 0}\cdot e_i$, $i=0,\ldots,3$\\
    	\addlinespace
    	$E_5$ & $-E_5+\RR_{\geq 0}\cdot e_i$, $i=0,\ldots,3$\\
    \midrule
    $V_1$ & $-V_1+\RR_{\geq 0}\cdot e_i+\RR_{\geq 0}\cdot e_j$, $0\leq i<j\leq 3$\\
    \addlinespace
    $V_2$ & $-V_2  +\RR_{\geq 0}\cdot e_i+\RR_{\geq 0}\cdot e_j$, $0\leq i < j \leq 3,\ (i,j)\neq (2,3)$\\
    	\bottomrule
    	\end{tabular}
    	\caption{$\trop(X^*)$ from Example \ref{ex:running} and \ref{ex:runningtrop}}\label{table:dual}
    \end{table}
    
    \begin{table}
    \begin{tabular}{l l l l}
    Cell of $\trop(X)$ & Top-dimensional contributions to $\trop(\tau(X))$\\
    	\toprule
    	$E_1$ & $E_1+\RR_{\geq 0}\cdot e_i$, $i=0,\ldots,3$\\ 
    	\addlinespace
    	$E_3$ & $E_3+\RR_{\geq 0}\cdot e_i$, $i=0,\ldots,3$\\
    	\addlinespace
    	$E_4$ & $E_4+\RR_{\geq 0}\cdot e_i$, $i=0,\ldots,3$\\
    	\addlinespace
    	$E_5$ & $E_5+\RR_{\geq 0}\cdot e_i$, $i=0,\ldots,3$\\
    \midrule
    $V_1$ & $V_1-3e_0+\RR_{\geq 0}\cdot e_i+\RR_{\geq 0}\cdot e_j$, $0\leq i<j\leq 3$\\
    \addlinespace
    $V_2$ & $V_2 +\RR_{\geq 0}\cdot (-e_1)+\RR_{\geq 0}\cdot e_j$, $j=0,2,3$\\
          & $V_2 +\RR_{\geq 0}\cdot e_1+\RR_{\geq 0}\cdot e_j$, $j=2,3$\\
    	\bottomrule
    	\end{tabular}
    	\caption{$\trop(\tau(X))$ from Example \ref{ex:running} and \ref{ex:runningtrop}}\label{table:tau}
    \end{table}
    
    \section{Multiplicities}\label{sec:mult}
    \subsection{Basics}\label{sec:mbasics}
    To any maximal-dimensional cell of a tropical variety, one may associate a \emph{multiplicity}. We refer the reader to \cite[\S 2.3]{lifting} for details. 
    For us, the following ad hoc definition will suffice.
    Consider hypersurfaces $Y_1,\ldots,Y_k\subset (\KK^*)^n$ such that the $\trop(Y_i)$ intersect transversely  in a maximal dimensional cell $E\subset \trop(Y_1\cap\ldots\cap Y_k)$. After rescaling by a monomial, we may assume that each $Y_i$ is cut out by a polynomial whose terms of minimal valuation on $E$ are $1+c_ix^{\bu_i}$ for some $c_i\in \KK^*$. Then the multiplicity of $E$ in $\trop(Y_1\cap\ldots\cap Y_k)$ is
    \begin{equation}\label{eqn:ci}
    	\bm(E)=[\langle \bu_i\rangle\cap M: \sum_i \ZZ\cdot \bu_i ],
    \end{equation}
    that is, the index of the abelian group generated by the $\bu_i$ in the intersection of $M$ with the vector space generated by the $\bu_i$. This is compatible with the \emph{fan displacement rule} \cite[Theorem 3.2]{FultonSturmfels}.
    
    In what follows, we will first compute multiplicities for the tropicalization of the graph of the Gauss map. We will then utilize this to compute multiplicities for the tropicalization of $\G(X)$, $\tau(X)$, and $X^*$. Throughout, we will assume that $X\subset \PP^n$ is a curve satisfying the hypotheses of \S\ref{sec:setup} and Assumption \ref{ass:setup} and that $X_1,\ldots,X_{n-1}$ have very general valuations and sufficiently general lowest order parts.
    \begin{ex}[Continuation of Example \ref{ex:runningP} (a curve in $\PP^2$)]\label{ex:multP}
    	We may use \eqref{eqn:ci} to compute that the edges $E_+$ and $E_-$ have multiplicity one, while the edge $E'$ has multiplicity two.. 
    \end{ex}
    
    \begin{ex}[Continuation of Example \ref{ex:running} (a curve in $\PP^3$)]\label{ex:mult}
    	We may use \eqref{eqn:ci} to compute that all edges of $\trop(X)$ have multiplicity one. 
    \end{ex}
    \subsection{Multiplicities for the Gauss map}
    Let $\widetilde{\G(X)}\subset \PP^n\times \Gr(2,n+1)$ be the closure of the graph of the Gauss map $\G_X:X\dashrightarrow \Gr(2,n+1)$. 
    There are natural projections $\pi_X:\trop(\widetilde{\G(X)})\to \trop(X)$ and $\pi_\G:\trop(\widetilde{\G(X)})\to \trop(\G(X))$. We first show how to compute multiplicities for $\trop(\widetilde{\G(X)})$. Note that the maximal-dimensional cells of $\trop(\widetilde{\G(X))}$ are one-dimensional.
    
    \begin{prop}\label{prop:mult}
    	Let $E$ be an edge of $\trop(\widetilde{\G(X)})$.
    \begin{enumerate}
    	\item If $\pi_X(E)$ is one-dimensional, then the multiplicity of $E$ is the same as the multiplicity of $\pi_X(E)$.
    	\item If $\pi_X(E)$ is a point in $\trop(X)$, then $\pi_X(E)$ is in the critical locus of $\trop(X)$, and there is a corresponding cancellative pair $\br,\br'$. The multiplicity of $E$ is the absolute value of the determinant of any $n\times n$-submatrix of the matrix with rows  
    		$\bv_1,\ldots,\bv_{n-1}$ along with $\sum_j (\br_j-\br_j')$.
    \end{enumerate}
    \end{prop}
    \begin{proof}
    Suppose first that $\pi_X(E)$ is one-dimensional. Then $E$ is arising from tropical tangents at points not contained in the critical locus, since by assumption the critical locus is finite. Minimal valuation terms along $E$ for polynomials cutting out $\widetilde{\G(X)}$ are 
    \begin{align}
    	\begin{split}\label{eqn:gx1}
    	1+c_{i\bv_i}x^{\bv_i} &\qquad i=1,\ldots,n-1\\
    z_J-\Delta_J(\br_J)a_{\br_J} &\qquad |J|=2
    \end{split}
    \end{align}
    for some $\br_J$'s depending on $J$.
    On the other hand, minimal valuation terms along $\pi_X(E)$ for $X$ are just the $1+c_{i\bv_i}x^{\bv_i}$. It follows from \eqref{eqn:ci} that the multiplicities of $E$ and $\pi_X(E)$ agree.
    
    Suppose instead that $\pi_X(E)$ is contained in the critical locus. Then
    lowest order terms for polynomials cutting out $\widetilde{\G(X)}$ are 
    \begin{align}
    	\begin{split}\label{eqn:gx2}
    1+c_{i\bv_i}x^{\bv_i} &\qquad i=1,\ldots,n-1\\
    \Delta_I(\br)a_{\br}+\Delta_I(\br')a_{\br'}&\\
    z_J-h_J&\qquad J\neq I
    \end{split}
    \end{align}
    for some monomial $h_J$ in the $a_{\br}$ and $z_I$. Here, for the latter polynomials, we have made use of Proposition \ref{prop:algorithm}.
    We again use \eqref{eqn:ci} to obtain the desired claim.
    \end{proof}
    
    We can now use this to compute multiplicities for $\G(X)$.
    \begin{thm}\label{thm:gmult} Assume that $X$ is not a line. Then the multiplicity of an  edge $E$ of $\trop(\G(X))$ is 
    \[
    	\sum_{\substack{F\subset \trop(\widetilde{\G(X)})\\ \pi_\G(F)=E}} \bm(F)\cdot [N_E:\pi_\G(N_F)],
    \]
    where  $F$ ranges over the edges of $\trop(\widetilde{\G(X)})$, $\bm(F)$ is the multiplicity of $F$ as determined in Proposition \ref{prop:mult}, and $N_E,N_F$ are the intersections of $\langle E\rangle$ and $\langle F \rangle$ with the appropriate cocharacter lattices.
    \end{thm}
    \begin{proof}
    	Since we are assuming that $\KK$ has characteristic zero, the degree of the map $\widetilde{\G(X)}\to \G(X)$ is always one \cite[Proposition 15.3]{harris}. The claim now follows directly from \cite[Theorem C.1]{lifting}.
    \end{proof}
    \begin{ex}[Continuation of Example \ref{ex:running} (a curve in $\PP^2$)]\label{ex:multP2}
    Since $X$ is a plane curve, we may identify $\trop(X^*)$ and $\trop(\G(X))$. 
	    We may compute the multiplicities of edges of $\trop(\G(X))$ using Proposition \ref{prop:mult} and Theorem \ref{thm:gmult}.
	    The edge $E_+$ contributes the edge 
	    \[\{(\alpha ,-3\alpha)\ |\ \alpha\in E_+\}\subset \trop(\widetilde{\G(X)}),\]
	   which projects to the edge $-E_+$ in $\trop(\G(X))$. The sublattice index in this case is $3$, so this edge has multiplicity $3$.

	    The edge $E_-$ contributes the edge 
	    \[\{(\alpha ,-\alpha)\ |\ \alpha\in E_-\}\subset \trop(\widetilde{\G(X)}),\]
	    which projects to the edge $-E_-$ in $\trop(\G(X))$. The sublattice index in this case is $1$, so this edge has multiplicity $1$.

The edge $E'$ contributes the edge 
	    \[\{(\alpha ,-2\alpha)\ |\ \alpha\in E'\}\subset \trop(\widetilde{\G(X)}),\]
for $\alpha\in E'$,
which projects to the edge $-E'$ in $\trop(\G(X))$. The sublattice index in this case is $2$. Since $E'$ already had multiplicity two, this edge has multiplicity $2\cdot 2=4$.

 The vertex  $V$ contributes the edge $(0,\RR_{\geq 0}\cdot e_1) \subset \trop(\widetilde{\G(X)}$. By Proposition \ref{prop:mult}, its multiplicity in $\trop(\widetilde{\G(X)}$ is the same as the absolute value of the determinant of any $2\times 2$ minor of 
the matrix whose rows are
$(0,-1,1)$, and $(0,-1,1)-(-2,1,1)$. Hence, the multiplicity is $2$. When projecting to the edge $\RR_{\geq 0}\cdot e_1$ of  $\trop(\G(X))$, the multiplicity remains $2$.

See Figure \ref{fig:planarExample} for a depiction of $\trop(\G(X))$ with multiplicities.
    \end{ex}

    \subsection{Multiplicities for the Tangential and dual varieties}
    We now consider two incidence varieties:
    \begin{align*}
    	I_{\tau(X)}=\overline{\{(P,L,Q)\in X\times \Gr(2,n+1)\times \PP^n\ |\ T_P(X)=L,\ Q\in L\}}\\
    	I_{X^*}=\overline{\{(P,L,Q)\in X\times \Gr(2,n+1)\times (\PP^n)^*\ |\ T_P(X)=L,\ L\subset Q\}}.
    \end{align*}
    These are respectively the pull-backs of $\PP(\cS)$ or $\PP(\cQ^*)$ to $\widetilde{\G(X)}$.
    Similar to in Theorem \ref{thm:dual}, we obtain explicit descriptions for $\trop(I_{\tau(X)})$ and $\trop(I_{X^*})$ by intersecting $\trop(\widetilde{\G(X)})$ with tropical hypersurfaces obtained from \eqref{eqn:taut} or \eqref{eqn:dual}. 
    
    To compute multiplicities for the tangential and dual varieties, we will use the following lemma:
    
    \begin{lemma}\label{lemma:bundle}
    	Let $W$ be the pull-back of either $\PP(\cS)$ or $\PP(\cQ^*)$ to $\widetilde{\G(X)}$. Consider any maximal cell $E$ of $\trop(W)$. Its image under the projection to $\trop(\widetilde{\G(X)})$ is a maximal cell $E'$, and $\bm(E)=\bm(E')$.
    \end{lemma}
    \begin{proof}
    	The tropicalization of a linear space is determined by the valuations of its Pl\"ucker coordinates \cite[Theorem 3.8]{ss}. Hence, the fibers of the projection $\trop(W)\to \trop(\widetilde{\G(X)})$ are all tropical linear spaces, and in particular have dimension $m$. It follows that maximal dimensional cells of $\trop(W)$ project to maximal dimensional cells of $\trop(\widetilde{\G(X)})$.
    
    	The second claim will follow from \eqref{eqn:ci}. We argue in the case $W=\PP(\cQ^*)$; the other case is similar. The variety $\widetilde{\G(X)}$ is a complete intersection, and on $E'$ the corresponding equations have minimal valuation terms as in \eqref{eqn:gx1} or \eqref{eqn:gx2}. To obtain the variety $W$, we add the variables $y_0,\ldots,y_n$ along with the equations \eqref{eqn:dual}. These additional equations are no longer a complete intersection. However, for points tropicalizing to $E$, we may find $n-(\dim W-\dim \widetilde{\G(X)})=2$ of these equations whose corresponding tropical hypersurfaces intersect transversely along $E'$. After rescaling by a monomial, the terms of minimal valuation of these equations will each have the form
    	\begin{equation}\label{eqn:y}
    		1+h\cdot \frac{y_i}{y_j}
    	\end{equation}
    	for some $i\neq j$ and $h$ a Laurent monomial in the Pl\"ucker coordinates (or equivalently a monomial in the $x^\bu$ and $z_J$).
    	
    	Applying \eqref{eqn:ci} to compute $\bm(E)$, it is straightforward to see that 	
    	we will obtain the same result whether or not we include \eqref{eqn:y}. It follows that $\bm(E)=\bm(E')$.
    \end{proof}
    
    Using Proposition \ref{prop:mult} and Lemma \ref{lemma:bundle}, we may determine the multiplicities for $\trop(I_{\tau(X)})$ and $\trop(I_{X^*})$. We now also obtain the multiplicities of $\trop(\tau(X))$ and $\trop(X^*)$.
    Let $\pi_{\tau(X)}$ and $\pi_{X^*}$ respectively be the projections of $\trop(I_{\tau(X)})$ and $\trop(I_{X^*})$ to $\trop(\tau(X))$ and $\trop(X^*)$.
    We use notation as in Theorem \ref{thm:gmult}.
    \begin{thm}\label{thm:mult}
    	Assume that $X$ is not a line.
    Consider any maximal cell $E$ of $\trop(X^*)$. Its multiplicity is
    \[
    	\sum_{\substack{F\subset \trop(I_{X^*})\\ \pi_{X^*}(F)=E}} \bm(F)\cdot [N_E:\pi_{X^*}(N_F)].
    \]
    \end{thm}
    \begin{proof}
    	Since the Gauss map is birational (\cite[Proposition 15.3]{harris}), the projection from $I_{X^*}$ to the \emph{conormal variety} $\PP(\cQ^*_{|\G(X)})$ is birational.
    	By projective duality, the conormal variety for $X$ agrees with that of $X^*$ \cite[\S 1.2A]{tevelev}. Hence as long as $\dim X^*=n-1$, the projection from the conormal variety to $X^*$ is also birational. Since $X$ is not a line, $\dim X^*$ is indeed $n-1$.
    	The claim now follows directly from \cite[Theorem C.1]{lifting}.
    \end{proof}
    
    \begin{cor}\label{cor:tamemult}
    	Consider any tame vertex $\alpha\in\trop(X)$ as in Proposition \ref{prop:vertcontrib}. Then the contributions of this vertex to $\trop(X^*)$ all have multiplicity equal to the absolute value of the determinant of any of the $n\times n$-submatrices of the matrix with rows $\bv_1',\bv_1,\ldots,\bv_{n-1}$. 
    \end{cor}
    \begin{proof}
    By the second claim of Proposition \ref{prop:vertcontrib}, maximal dimensional cells $E$ of $\trop(X^*)$  contributed by $\alpha$ each come from a single maximal dimensional cell $F$ of  $\trop(\widetilde{\G(X)})$ contributed by $\alpha$.
    	In this situation, the lattice $N_F$ surjects onto the lattice $N_E$: by the discussion in the proof of Proposition \ref{prop:vertcontrib}, any point $\gamma=-\alpha+\lambda\in E$ is the image of $(\alpha,\beta,\gamma)$, where 
    	$\delta_{ij}=\beta_{ij}-(\alpha_i+\alpha_j)$ is either fixed for all $\gamma\in E$, or equals 
    	\[
    		\delta_{(n-1)n}=\omega_{(n-1)n}(i,n)+\lambda_k
    	\]
    for some fixed $k$. It follows that if $\gamma-\gamma'\in N_E$, then 
    \[
    (\alpha,\beta,\gamma)-(\alpha,\beta',\gamma')=(0,\beta-\beta',\gamma-\gamma')\in N_F
    \]
    since $(\beta-\beta')_{ij}$ is either zero, or $\lambda_k-\lambda_k'$.
    The claim now follows from Proposition \ref{prop:mult} and Theorem \ref{thm:mult}. 
    \end{proof}

    \begin{thm}\label{thm:multtau}
    	Assume that $X$ is not contained in a plane.
    Consider any maximal cell $E$ of $\trop(\tau(X))$. Its multiplicity is
    \[
    		\sum_{\substack{F\subset \trop(I_{\tau(X)})\\ \pi_{\tau(X)}(F)=E}}\bm(F)\cdot [N_E:\pi_{X^*}(N_F)]. 
    \]
    \end{thm}
    \begin{proof}
    	We may assume that $X$ is not a line, since then it is contained in a plane. It follows that $\dim(\tau(X))=\dim \PP(\cS_{|\widetilde{\G(X)}})=2$. 
    
    	We now claim that a generic point of $\tau(X)$ lies on only one tangent line of $X$. 
    	Suppose instead that any generic point lies on at least two tangent lines. Therefore, for a generic line $L$, every generic point $P\in L$  lies on another tangent line. It follows that every tangent line of $X$ meets $L$. Indeed, the restriction of $\PP(\cS_{|\G(X)})$ to the locus in $G(2,n+1)$ of lines meeting $L$ must have dimension two, so it agrees with $\PP(\cS_{|{\G(X)}})$.
    
    Since $L$ was generic, and every  tangent line of $X$ meets $L$, we conclude that any pair of tangent lines to $X$ is coplanar. But this implies that $\tau(X)$ is contained in a plane, and so is $X$, contradicting a hypothesis of the theorem. We conclude that a generic point of $\tau(X)$ lies on only one tangent line of $X$. 
    
    	This shows that the map $\PP(\cS_{|\G(X)})\to \tau(X)$ is finite of degree one.
    It follows that the projection from $I_{\tau(X)}$ to $\tau(X)$ is also finite of degree one. The claim follows from \cite[Theorem C.1]{lifting}.
    \end{proof}

    \begin{ex}[Continuation of Example \ref{ex:running} (a curve in $\PP^3$)]\label{ex:mult2}
    	We may compute the multiplicities of top-dimensional cells of $\trop(X^*)$ using Proposition \ref{prop:mult} and Theorem \ref{thm:mult}. We list all multiplicities differing from $1$ in Table \ref{table:mult}. We describe two of the computations in more detail. 
    
    	For the vertex $V_1$, we may apply Corollary \ref{cor:tamemult}. For this, we must compute the absolute value of the determinant of any $3\times 3$ submatrix of
    	\[
    		\left(\begin{array}{c c c c}
    1&0&3&-4\\
    2&0&1&-3\\
    0&1&1&-2
    		\end{array}\right).
    	\]
    	This turns out to equal $5$.
    
    	For the edge $E_3$, we first note that it has multiplicity $1$ (Example \ref{ex:mult}). Thus, the multiplicity of its contributions will come from the lattice factor in Theorem \ref{thm:mult}. The top-dimensional cells in $\trop(I_{X^*})$ coming from $E_3$ have the form
    	\[
    	(\alpha,\beta(\alpha),-\alpha+\RR_{\geq 0}e_i), \qquad \alpha\in E_3=V_2+\RR_{\geq 0}\cdot (3,2,4)\]
    	where the $jk$-th coordinate of $\beta(\alpha)$ is $\alpha_j+\alpha_k$. We must thus compute the index of the subgroup generated by $(3,2,4)$ and $e_i$
    	in the group obtained by intersecting the linear space spanned by these elements with $\ZZ^3$. This equals $1$ except for $i=1$, in which case it equals $2$.
    	The remaining computations are similar.
    
    	We may use our description of $\trop(X^*)$ together with these multiplicities to obtain the Newton polytope for $X^*$, see e.g.~\cite[Proof of Proposition 3.3.11 and Remark 3.3.12]{tropical} for details. The Newton polytope for $X^*$ is the three-dimensional convex hull in $\RR^4$ of the columns of
    
    \[{\left({\begin{array}{ccccccccccccc}
          17&1&19&0&6&2&0&0&16&1&0&16&1\\
          4&20&0&18&0&3&5&3&4&19&17&3&18\\
          4&4&6&7&19&20&20&19&1&1&4&0&0\\
          0&0&0&0&0&0&0&3&4&4&4&6&6\\
          \end{array}}\right)}\]
          and
    \[\begin{pmatrix}
          0&2&0&17&1&0&2&0&4&2\\
          16&0&1&0&5&3&0&1&0&1\\
          3&17&17&0&0&3&4&4&0&0\\
          6&6&7&8&19&19&19&20&21&22\end{pmatrix}.\]
    
    \vspace{.5cm}
    \noindent This polytope has $23$ vertices, $36$ edges, $15$ facets, and contains a total of $2698$ lattice points.  If one wishes to find the polynomial defining $X^*$, one could explicitly find sufficiently many points on $X$, and set up a linear system of equations in the $2698$ unknowns playing the role of coefficients for the monomials corresponding to the above-mentioned lattice points. 
          By contrast, a direct computation of $X^*$ via elimination theory does not appear tractable.
    
          A posteriori, we see that the degree of $X^*$ is $25$. We can also verify this using \texttt{Macaulay2} \cite{M2}:
          \begin{verbatim}
    i1 : R=frac(QQ[t])[x_0..x_3];
    i2 : I=ideal {x_0*x_2^3+x_0^2*x_2*x_3+x_3^4,
                      t^3*x_0*x_3^2+x_1^2*x_2+x_1*x_3^2};
    i3 : polar=(gens minors(2,(jacobian I))*random(QQ^6,QQ^1))_(0,0);
    i4 : degree saturate(I+polar,ideal{x_0*x_1*x_2*x_3})
    o4 = 25
          \end{verbatim}
    Note that we are saturating to remove the boundary component of $X_1\cap X_2$ as well as the singular points of $X$.
    \end{ex}
    
    \begin{table}
    	\begin{tabular}{l l l}
    Cell of $\trop(X)$ & Contribution to $\trop(X^*)$ & Multiplicity\\
    \toprule
    $V_1$& all & $5$\\
    \addlinespace
    $V_2$& all & $3$\\
    \addlinespace
    $V_3$& all & $5$\\
    \midrule
    $E_2$ & $\conv\{-V_2,-V_3+3(e_2+e_3)\}+ \RR_{\geq 0}\cdot e_2$ & $5$\\
          & $\conv\{-V_2,-V_3+3(e_2+e_3)\}+ \RR_{\geq 0}\cdot e_3$ & $4$\\
    \addlinespace
    $E_3$ & $-E_3+\RR_{\geq 0}\cdot e_1$ & $2$\\
    \addlinespace
    $E_4$ & $-E_4+\RR_{\geq 0}\cdot e_0$ & $2$\\
    \addlinespace
    $E_5$ & $-E_5+\RR_{\geq 0}\cdot e_3$ & $2$\\
    \addlinespace
    $E_6$ & $\conv \{-V_3,-V_3+3(e_2+e_3)\}-E_6+\RR_{\geq 0}\cdot (-e_2)$ & $2$\\
           & $-V_3+3(e_2+e_3)+ \RR_{\geq 0}\cdot e_3+\RR_{\geq 0}\cdot (-e_2)$ & $2$\\
    \addlinespace
    $E_7$  & $-V_3+3(e_2+e_3)+ \RR_{\geq 0}\cdot e_2+\RR_{\geq 0}\cdot (6e_2+5e_3)$ & $5$\\
           & $-V_3+3(e_2+e_3)+ \RR_{\geq 0}\cdot e_3+\RR_{\geq 0}\cdot (6e_2+5e_3)$ & $6$\\
    \bottomrule
    \end{tabular}
    \vspace{.5cm}
    \caption{Multiplicities greater than 1 for $\trop(X^*)$}\label{table:mult}
    \end{table}

    \bibliographystyle{amsalpha}
    \bibliography{paper}
    \end{document}